\documentclass{article}
\usepackage{amsmath,amsthm,amsfonts,amssymb,esint,mathtools,extarrows,witharrows,enumitem}

\newcommand{\N}{\mathbb{N}}   
\newcommand{\R}{\mathbb{R}} 
  
\newcommand{\Z}{\mathbb{Z}}

\makeatletter
\renewcommand*\env@matrix[1][*\c@MaxMatrixCols c]{%
  \hskip -\arraycolsep
  \let\@ifnextchar\new@ifnextchar
  \array{#1}} 
\makeatother

\newtheorem{corollary}{Corollary}[section]
\newtheorem{remark}{Remark}[section]    
\newtheorem{lemma}{Lemma}[section]
\newtheorem{proposition}{Proposition}[section] 
\newtheorem{thm}{Theorem}

\newtheorem{thmx}{Theorem}
\numberwithin{equation}{section}

\title
{
Asymptotics for minimizers of a Donaldson functional
and  mean curvature 1-immersions of surfaces  
into hyperbolic 3-manifolds
}

\author{Gabriella Tarantello}

\begin{document}

\maketitle

\begin{abstract}
It has been shown in \cite{Huang_Lucia_Tarantello_2} that,
for given $\vert c \vert <1$, the moduli space
of constant mean curvature (CMC) $c$-immersions of a closed orientable surface
of genus $\mathfrak{g} \geq 2$ into a hyperbolic $3$-manifold 
can be parametrized  by elements of the tangent
bundle of the corresponding  Teichm\"uller space.
This is attained by showing the unique solvability of the 
Gauss-Codazzi equations governing (CMC) c-immersions. 
The corresponding unique solution is identified as the global minimum 
(and only critical point)
of the Donaldson functional $D_t$ (introduced 
in \cite{Goncalves_Uhlenbeck}) given in \eqref{D_t} with $t=1-c^{2}$.

When $\vert c \vert \geq 1$ (i.e. $t\leq 0$), 
so far nothing is known about the existence of analogous 
(CMC) c-immersions. 
Indeed, for $t\leq 0$  the functional $D_{t}$
may no longer be bounded from below  
and evident non-existence situations do occur.

Already the case $\vert c \vert =1$ (i.e. $t=0$) appears rather involved 
and actually (CMC) 1-immersions
can be attained only as "limits" of (CMC) c-immersions 
for $\vert c \vert \longrightarrow 1^{-}$ (see Theorem \ref{thm_1}).
To handle this situation, here we analyse the asymptotic behaviour of minimizers
of $D_{t}$
as  $t \longrightarrow 0^{+}$.

We use an accurate asymptotic analysis (see \cite{Tar_1})
to describe possible blow-up phenomena.
In this way, we can relate
the existence of (CMC) 1-immersions to the Kodaira map
defined in \eqref{Kodaira_map_Intro}. 
As a consequence, 
when the genus $\mathfrak{g}=2$, 
we obtain the first 
existence and uniqueness result
about  (CMC) 1-immersions of surfaces  
into hyperbolic 3-manifolds.
\end{abstract}

%\tableofcontents

\section{Introduction}%\label{sec_asymptotics}

Let $S$ be an oriented closed surface with genus $\mathfrak{g}\geq 2$ and denote 
by $\mathcal{T}_{\mathfrak{g}}(S)$ the Teichm\"uller space of conformal structures on $S$, 
modulo biholomorphisms in the homotopy class of the identity.  

The Teichm\"uller space
has proved to be particularly useful in the description of various moduli spaces,
and here we explore such possibility for 
constant mean curvature (CMC) $c$-immersions of $S$ into hyperbolic $3$-manifolds.
We recall (from \cite{Uhlenbeck},\cite{Goncalves_Uhlenbeck}) that, 
for a given conformal class $X\in \mathcal{T}_{\mathfrak{g}}(S)$, 
such immersions are governed by solutions of the Gauss-Codazzi equations
relating the pull-back  metric on $X$ with the second fundamental form $II$
of the immersion. 
Actually, the Codazzi equation just states that the $(2,0)$-part of $II$
is a holomorphic quadratic differential $\alpha$ 
(i.e. the Hopf differential)  which completely identifies $II$. 
In particular for minimal (i.e. $c=0$) 
immersions we have $II =  Re(\alpha)$. Thus,
Uhlenbeck in \cite{Uhlenbeck} proposed a parametrization 
of such minimal immersions in terms of 
elements of the cotangent bundle of 
$\mathcal{T}_{\mathfrak{g}}(S)$, as described by the pairs:
$
(X,\alpha) \in \mathcal{T}_{\mathfrak{g}}(S) \times C_{2}(X)
$,
where $C_{2}(X)$ is the space of holomorphic 
quadratic differentials on $X$.
The role of holomorphic quadratic differentials 
in Teichm\"uller theory
has emerged naturally also in relation with harmonic maps.

Thus, to pursue Uhlenbeck's parametrization, 
it appears reasonable to seek minimal immersions with assigned 
second fundamental form, that is to set: 
$II=Re(\alpha)$, for a given $\alpha \in C_{2}(X)$.  
In this way,
one is reduced to solve the Gauss equation (of Liouville type)
for the pull-back metric in terms of the (given) second fundamental form. 

However, 
as discussed in \cite{Huang_Lucia},\cite{Huang_Lucia_Tarantello_1},   
a minimal immersion with assigned second fundamental form may not exist, 
or when it exists, 
it may not be unique (see also \cite{Hung_Loftin_Lucia}).
So, although suggestive, this approach does not  yield 
a one-to-one correspondence between a minimal immersion 
and a solution of the Gauss-Codazzi equations,
with respect to a given pair $(X,\alpha)$.

On the contrary, it has proved more successful the approach adopted 
by Goncalves and Uhlenbeck in \cite{Goncalves_Uhlenbeck}
where, more generally, the authors propose to parametrize 
constant mean curvature (CMC) immersions of $S$
into $3$-manifolds of constant sectional curvature $-1$, 
in terms of elements of the tangent bundle of
$\mathcal{T}_{\mathfrak{g}}(S)$.
By recalling the isomorphism:
$
C_{2}(X) \simeq (\mathcal{H}^{0,1}(X,E))^{*}
$, 
where 
$E=T^{1,0}_{X}$ is the holomorphic tangent bundle of $X$ and 
$\mathcal{H}^{0,1}(X,E)$ is the Dolbeault cohomology  group of $(0,1)$-forms
valued in $E$ (see \eqref{H_0_1_definition}),  
we find a parametrization of the tangent bundle of
$\mathcal{T}_{\mathfrak{g}}(S)$  by the pairs:
$
(X,[\beta])\in \mathcal{T}_{\mathfrak{g}}(S) \times \mathcal{H}^{0,1}(X,E) 
$. 

Accordingly, it has been proved
in \cite{Huang_Lucia_Tarantello_2} that 
(as anticipated by \cite{Goncalves_Uhlenbeck})
the following holds:

\begin{thmx}[\cite{Goncalves_Uhlenbeck},\cite{Huang_Lucia_Tarantello_2}]
\label{thm_A}
For given $c\in (-1,1)$ there is a one-to-one correspondence between the space of constant mean curvature $c$ immersions of $S$ into a $3$-manifold
of constant sectional curvature $-1$ and the 
tangent bundle of 
$\mathcal{T}_{\mathfrak{g}}(S)$, parametrized by the pairs: 
$
(X,[\beta])\in \mathcal{T}_{\mathfrak{g}}(S) \times \mathcal{H}^{0,1}(X,E),
\; 
E=T_{X}^{1,0}.  
$ 
\end{thmx}	
As we shall discuss in Section \ref{statements_main_results} 
(see Remark \ref{CMC_correspondence}),
the datum $(X,[ \beta ])$ 
characterizes the Donaldson functional $D_{t}$ 
given in \eqref{D_t}-\eqref{c_t_relation}, 
whose global minimum (and only critical point)
uniquely identifies the corresponding (CMC) c-immersion.   
We recall that, 
the functional $D_{t}$ was introduced in \cite{Goncalves_Uhlenbeck}
as the appropriate "lagrangean" of the Gauss-Codazzi equations 
governing the immersion.

\

At this point it is natural to ask whether, 
for a given pair $(X,[\beta])$, such (CMC)-immersions do exist also 
for $\vert c \vert \geq 1$. 

We face an evident non-existence situation when $[\beta]=0$
(see Section \ref{Preliminaries} for details), 
while for $[\beta]\neq 0$ such question
is completely open to investigation. 
In this note, we start to tackle the case $\vert c \vert =1$ and refer 
to such immersions as (CMC) 1-immersions. 
Starting with the work of \cite{Bryant}, 
(CMC) 1-immersions of surfaces into the hyperbolic space $\mathbb{H}^{3}$ 
have played a relevant role in hyperbolic geometry, 
in view of their striking analogies with minimal immersions into the euclidean
space $\mathbb{E}^{3}$, see also 
\cite{Rossman_Umehara_Yamada},\cite{Umehara_Yamada}.

Since here we deal with compact surfaces, 
in view of those analogies, we expect that for $\vert c \vert =1$ such immersions 
should "favor" the presence of "punctures" at a finite number of points. 
Indeed, in our analysis those points  occur as the "blow-up" points, 
and form the so called "blow-up set" appearing in 
Theorem \ref{thm_main_1_intro} and Theorem \ref{main_thm_2} below. 
But surprisingly, when the genus $\mathfrak{g}=2$, we are able to
find "regular" (CMC) 1-immersions. Actually, 
their existence will be formulated in terms of the Kodaira map
discussed in section 12.1.3 of \cite{Donaldson_Book},
and specified as follows:
\begin{equation}\label{Kodaira_map_Intro}
\tau:X\longrightarrow \mathbb{P}(V^{*}),\; V=C_{2}(X);
\end{equation}
which defines a holomorphic map of $X$ into the projective space of 
$V^{*}=\mathcal{H}^{0,1}(X,E)$ with $E=T^{1,0}_{X}$,
see \cite{Donaldson_Book} and Section \ref{Preliminaries} for details. 

The role of the projective space 
$\mathbb{P}(\mathcal{H}^{0,1}(X,E))$, $E=T^{1,0}_{X}$,
is readily explained. 
Indeed, if to a pair $(X,[\beta])$ there corresponds a 
(CMC) 1-immersion, then $[\beta]\neq 0$ and it exists 
a (CMC) 1-immersion also corresponding to the data
$(X,\lambda [\beta])$, for all $\lambda \in \mathbb{C}\setminus \{  0 \} $. 
 
Recall that, for $E=T^{1,0}_{X}$ we have:  
$\mathbb{P}(\mathcal{H}^{0,1}(X,E)) \simeq \mathbb{P}^{3\mathfrak{g}-4}$
and
$3\mathfrak{g}-4 \geq 2$ for $\mathfrak{g}\geq 2$, 
therefore $\dim_{\mathbb{C}}\mathbb{P}(\mathcal{H}^{0,1}(X,E)) \geq 2$. 
On the other hand, we shall see in Lemma \ref{lem_complex_curve} that the image 
$\tau(X)$ defines a complex curve 
(hence of complex dimension one)
into $\mathbb{P}(\mathcal{H}^{0,1}(X,E))$, 
whence $\tau(X) \subsetneq \mathbb{P}(\mathcal{H}^{0,1}(X,E))$,
and actually $\tau(X)$ is a "tiny" subset of 
$\mathbb{P}(\mathcal{H}^{0,1}(X,E))$.

If for $[\beta] \in \mathcal{H}^{0,1}(X,E) \setminus \{  0 \} $ we let
$[\beta]_{\mathbb{P}}$ denote the element of the the projective space 
$\mathbb{P}(\mathcal{H}^{0,1}(X,E))$
corresponding to the  complex line
$\lambda [\beta], \lambda \in \mathbb{C}\setminus \{  0 \} $, 
then we prove:

\begin{thm}\label{thm_one}
If $\mathfrak{g}=2$, then to every  
$(X,[\beta]) 
\in 
\mathcal{T}_{\mathfrak{g}}(X) \times (\mathcal{H}^{0,1}(X,E)\setminus \{  0 \} )$
$E=T^{1,0}_{X}$, with projective representative 
$[\beta]_{\mathbb{P}}\not \in \tau(X)$, 
there  corresponds a \underline{unique} (CMC) 1-immersion of $X$  
into a $3$-manifold $M(\simeq S \times \R$) 
with sectional curvature $-1$.
\end{thm}	
 
Let us emphasize once more  that, since $\tau(X)$ is a complex curve 
into a complex space of dimension $2$ 
(for $\mathfrak{g}=2$) then Theorem \ref{thm_one} applies for a "generic" class 
$[\beta] \in \mathcal{H}^{0,1}(X,E)\setminus \{  0 \} $, 
and it furnishes the first existence result about (CMC) 1-immersions. 

We obtain such an existence result  "variationally", 
from a \underline{strict} global minimum 
(and unique critical point) 
of the Donaldson functional $D_{0}$ in \eqref{def_D_0_intro}. 
Based on "stability" arguments, we expect 
existence also of (CMC) $c$-immersions with $\vert c \vert  > 1$
but close to $1$. 

It would be interesting to inquire if (CMC) 1-immersions as described in Theorem
\ref{thm_one} could exist also when
$[\beta]_{\mathbb{P}} \in \tau(X)$. 
 
\

To establish Theorem \ref{thm_one}, firstly 
for any genus $\mathfrak{g}\geq 2$ we provide
a detailed asymptotic analysis about the (CMC)-immersions given in 
Theorem \ref{thm_A} (with $\vert c \vert < 1$), 
as $\vert c \vert \longrightarrow 1^{-}$.
Indeed, by Theorem \ref{thm_1}, (CMC) 1-immersions
can be attained only as "limits" in this way. 

For this purpose, we need to handle possible "blow-up" situations 
for the pull-back metrics. 
This analysis becomes particularly delicate when "blow-up"
occurs around points of "collapsing" zeroes 
of the holomorphic quadratic differentials governing the second fundamental form. 
Recall that indeed, any holomorphic quadratic differential admits 
$4(\mathfrak{g}-1)$ zeroes in $X$, counted with multiplicity. 

By the recent estimates obtained in \cite{Tar_1} for solutions of Liouville 
equations with "collapsing" Dirac sources, 
we are able to overcome such difficulties when "blow-up" occurs 
with the least possible "blow-up mass" (see Section \ref{blow_up_minimizers} for details). 
This property is always guaranteed for genus $\mathfrak{g}=2$. 
But we expect it to hold also for any genus $\mathfrak{g}\geq 2$, by virtue
of the "minimising" character of the (CMC) c-immersions of Theorem \ref{thm_A},
from the point of view of the Donaldson functional \eqref{D_t}.

We refer to the following sections for details and 
more specifically to our main 
Theorem \ref{thm_main_1_intro}, Theorem \ref{main_thm_3} 
and Theorem \ref{main_thm_2}
(and their corollaries) for precise statements of our results. 
They contain useful information in case of  "blow-up", including a 
crucial "orthogonality" condition
imposed on the class $[\beta]$, as stated in 
\eqref{orthogonality_condition_intro},\eqref{if_and_only_if}
and 
\eqref{orthonality_on_S_0}.

This enables us to relate the "blow-up" phenomenon to 
the image of the Kodaira map
and  to establish  Theorem \ref{thm_one}. 
Equivalently, we refer to  Theorem \ref{thm_infimum_in_lambda_outside_tau_X} 
concerning the existence for extremals of the associated 
(non-coercive) 
Donaldson function $D_{0}$ (in \eqref{def_D_0_intro}).

\subsection{Statement of the Main Results.}\label{statements_main_results}

We take a slightly more general point of view. 

For a given conformal class $X\in \mathcal{T}_{\mathfrak{g}}(S)$ 
(identified with the corresponding Riemann surface),  
we recall that,  $T_{X}^{1,0}$ denotes  the holomorphic tangent bundle of $X$, whose 
dual $(T(X)^{1,0})^{*}=K_{X}$ coincides with the canonical bundle 
$K_{X}$ of $X$. 
Then, for $\kappa \geq 2$, we define the holomorphic line bundle:
$ 
E= \otimes^{\kappa-1}T^{1,0}_{X}$, equipped with the corresponding 
holomorphic structure $\bar{\partial}=\bar{\partial}_{E}$. 

We let   
$A^{0}(E)$
be the space of smooth sections of $E$ and 
$A^{0,1}(X,E)=A^{0,1}(X,\mathbb{C}) \otimes E$ be 
the space of $(0,1)$-forms of $X$ valued on $E$. 
By considering the d-bar operator 
$
\overline{\partial}
:
A^{0}(E)
\longrightarrow 
A^{0,1}(X,E)
$,      
we may define the $(0,1)$-Dolbeault cohomology group:
$
\mathcal{H}^{0,1}(X,E)
=
A^{0,1}(X,E)/\overline{\partial}(A^{0}(E))
$
so that, every cohomology class 
$[\beta]  \in \mathcal{H}^{0,1}(X,E)$ is formed by 
$(0,1)$-forms valued in $E$ of the type:  
$\beta + \bar{\partial}\eta, \; \forall \; \eta \in A^{0}(E)$. 

Furthermore, 
we can use the induced complex structure
$\bar{\partial}$ over the holomorphic line bundle 
$\otimes^{k}(T_{X}^{1,0})^{*}$, 
to define the space $C_{\kappa}(X)$ of holomorphic
$\kappa$-differentials on $X$ as given by the   
holomorphic sections of
$E^{*} \otimes K_{X}=\otimes^{\kappa} (T_{X}^{1,0})^{*}$, that is: 
\begin{equation*}%\label{definition_C_kappa_X}
\begin{split}
C_{\kappa}(X)
= H^{0}(X,E^{*}\otimes K_{X}) 
:= \; &
\{ 
\alpha \in A^{0}(X,\otimes^{k}(T_{X}^{1,0})^{*})
\; : \; 
\bar{\partial} \alpha=0
\}.
\end{split}
\end{equation*}
So 
$C_{\kappa}(X)\subset 
A^{0}(E^{*}\otimes K_{X})=A^{1,0}(X,E^{*})
$, 
and by the Riemann Roch theorem, we know that:
\begin{equation}\label{dimension_C_kappa_X}
\dim_{\mathbb{C}}C_{\kappa}(X) 
=
(2\kappa-1)(\mathfrak{g}-1),
\end{equation} 
see  \cite{Griffiths_Harris}. 

In particular, if $\kappa=2$ then
$\dim_{\mathbb{C}}C_{2}(X)=3(\mathfrak{g}-1)$, and 
since  $\mathcal{T}_{\mathfrak{g}}(S)$ has the structure of
a differential cell of real dimension $6(\mathfrak{g}-1)$ (see \cite{Jost}), 
as already mentioned, we can parametrize the cotangent bundle of 
$\mathcal{T}_{\mathfrak{g}}(S)$ by the pairs
$(X,\alpha) \in \mathcal{T}_{\mathfrak{g}}(X)\times C_{2}(X)$.
Similarly, in view of the isomorphism:
$
C_{\kappa}(X) \simeq (\mathcal{H}^{0,1}(X,E))^{*} 
$
(see Section \ref{Preliminaries} for details), 
then for $\kappa=2$,
we see also that the tangent bundle of 
$\mathcal{T}_{\mathfrak{g}}(S)$  can be parametrized by  the pairs 
$(X,[\beta]) \in \mathcal{T}_{\mathfrak{g}}(S) \times \mathcal{H}^{0,1}(X,E)$.

Next, we consider
on $X$ the unique hyperbolic metric $g_{X}$
(from the "uniformization" of $X$)
with induced scalar product $\langle \cdot,\cdot \rangle $, 
norm $\vert \cdot \vert $ and volume form $dA$.
Unless specified otherwise, we shall always consider $g_{X}$
as the background metric on $X$. 

By using the Hermitian metric  induced by $g_{X}$ on $X$, 
we can define an Hermitian structure on $E$, 
with induced fiberwise Hermitian product 
$\langle \cdot,\cdot \rangle_{E}$
and corresponding (fiberwise) norm 
$\Vert \cdot \Vert_{E}$ of sections and forms valued in $E$.

Also, as detailed in Section \ref{Preliminaries}, 
we can introduce the wedge product: $\beta \wedge \alpha$, 
for $\beta \in A^{0,1}(X,E)$ and $\alpha \in A^{1,0}(X,E^{*})$, 
and then define the 
\underline{Hodge star} operator
$*_{E}:A^{0,1}(X,E)\longrightarrow A^{1,0}(X,E^{*})$
in terms of the Hermitian product $\langle \cdot,\cdot \rangle_{E}$.
We know that, $*_{E}$ is an isomorphism with inverse $*_{E}^{-1}$,  
and actually it defines  an isometry  
with respect to the "dual" (fiberwise) Hermitian product on $E^{*}$:
$$
\langle \cdot,\cdot,\rangle_{E^{*}}
=
\langle *_{E}^{-1} \cdot,*_{E}^{-1}\cdot\rangle_{E}
\; \text{ with norm } \; 
\Vert \cdot \Vert_{E^{*}} = \Vert *_{E}^{-1} \cdot \Vert_{E}.
$$
In the sequel we shall drop the lower script $E^{*}$ and $E$
for the Hermitian product and norm, 
unless there is some ambiguity.

By Dolbeault decomposition, any 
$\beta \in A^{0,1}(X,E)$ can be expressed uniquely as follows:
$$
\beta = \beta_{0} + \bar{\partial}\eta
\; \text{ with  harmonic }\;
\beta_{0} \in A^{0,1}(X,E)
\text{ 
(with respect to $g_{X}$), } \;
\eta \in A^{0}(E).
$$
So $\beta_{0} \in [\beta]$, and every cohomology class 
in $\mathcal{H}^{0,1}(X,E)$
is uniquely identified 
by its harmonic representative. 
More importantly, we have: 
$$
\beta_{0} \in A^{0,1}(X,E) \; \text{ is harmonic } \;
\Longleftrightarrow
*_{E}\beta_{0} \in C_{\kappa}(X) 
$$
see Section \ref{Preliminaries} for details.

In particular, if for a smooth function $u$ in $X$ 
we have: $\bar{\partial} *_{E} e^{2u}\beta = 0$
then $\beta$ is \underline{harmonic} with respect to the metric $h
=e^{u}g_{X}$. 
 
\ 
 
For any pair $(X,[\beta])$, with 
$\beta_{0} \in [ \beta ]$ the harmonic representative of the class $[\beta]$,
we define 
(in the terminology of \cite{Goncalves_Uhlenbeck}) 
the \underline{Donaldson functional}:
\begin{equation}\label{D_t}
D_{t}(u,\eta)
= 
\int_{X}
(
\frac{\vert \nabla u \vert^{2}}{4}
-
u
+
te^{u}
+
4e^{(\kappa-1)u}\Vert \beta_{0} + \overline{\partial} \eta \Vert^{2}
)
\,dA
, 
\end{equation}
$t\in \R$, with "natural" (convex) domain:
\begin{equation*}%\label{def_Lambda}
\begin{split}
\Lambda
=
\{  
(u,\eta)\in H^{1}(X) \times W^{1,2}(X,E)
\; : \; 
\int_{X}e^{(\kappa-1)u}
\Vert \beta_{0} + \bar{\partial}\eta \Vert^{2}\,dA < \infty
\},
\end{split}
\end{equation*}
with the usual Sobolev spaces  $H^{1}(X)$ of $X$,
and $W^{1,2}(X,E)$ of sections of $E$ (see \eqref{W_1_p}).  
Clearly, for $t>0$, the functional $D_{t}$ is bounded from below in $\Lambda$.

As observed in \cite{Goncalves_Uhlenbeck}, at least locally,
it is possible to construct a (CMC) immersion of $X$ with constant $c$
into a $3$-manifold having sectional curvature $-1$, 
directly from a critical point of the Donaldson functional $D_{t}$, when we take:
\begin{equation}\label{c_t_relation}
t=(1-c^{2}) \; \text{ and } \; \kappa=2.
\end{equation}
Indeed,   
if $(u,\eta)$ is a (weak) critical point 
of $D_{t}$ (in the sense of \eqref{weak_criticality} below), 
then it is smooth and satisfies:
\begin{equation}\label{system_of_equations_intro}
\left\{
\begin{matrix*}[l]
\Delta u +2 -2te^{u} -8(\kappa-1)e^{(\kappa-1)u}\Vert \beta_{0}+\overline{\partial}\eta \Vert^{2} =0  &  \;\text{ in }\;  &  X,  \\
\overline{\partial}(e^{(\kappa-1)u}*_{E}(\beta_{0}+\overline{\partial}\eta))=0, &  \;\text{}\;  &   \\ 
\end{matrix*}
\right.
\end{equation}
From the second equation in \eqref{system_of_equations_intro} we see that, 
$\beta_{0}+\bar{\partial}\eta \in [\beta]$ is harmonic with respect to the metric 
$h=e^{\frac{\kappa-1}{2}u}g_{X}$. 

More importantly, from \cite{Uhlenbeck} we have:
\begin{remark}\label{CMC_correspondence}
If $(u,\eta)$ satisfies \eqref{system_of_equations_intro} with
$\kappa=2$ and $t=(1-c^{2})$,  
then
$(X,e^{u}g_{X})$ 
can be immersed as a (CMC) surface with constant $c$ 
into a suitable $3$-manifold $M^{3}\simeq S \times \mathbb{R}$
of sectional curvature $-1$. 
Furthermore, the $(2,0)$-part of the second
fundamental form $II$ of the immersion is given by:
$
\alpha
=
*_{E} e^{u} (\beta_{0}+\overline{\partial}\eta) 
\in
C_{2}(X) 
$.
\end{remark}

\noindent	
We refer to 
\cite{Goncalves_Uhlenbeck},\cite{Huang_Lucia_Tarantello_1}
and
\cite{Huang_Lucia_Tarantello_2} 
for details.
 
Actually, the system \eqref{system_of_equations_intro}
can be formulated as the Hitchin selfduality equations \cite{Hitchin}
for a suitable nilpotent $SL(2,\mathbb{C})$ Higgs bundle, 
we refer to \cite{Alessandrini_Li_Sanders}, \cite{Huang_Lucia_Tarantello_2}
for details, see also \cite{Li} for a similar formulation 
concerning Harmonic maps in relation to  minimal immersions. 
Under this point of view, 
it becomes clear why we refer to $D_{t}$ as a Donaldson functional.  

Interestingly, (as also anticipated in \cite{Goncalves_Uhlenbeck}) 
for $t>0$, it is possible to show the unique solvability of 
\eqref{system_of_equations_intro}.

\begin{thmx}[\cite{Huang_Lucia_Tarantello_2}]
\label{thm_A_intro}
For given 
$(X,[\beta])\in \mathcal{T}_{\mathfrak{g}}(S) \times \mathcal{H}^{0,1}(X,E)$ 
and $t>0$, 
the functional $D_{t}$
admits a \underline{unique} critical point $(u_{t},\eta_{t})$
which corresponds to the global minimum of $D_{t}$ in $\Lambda$.
In particular, $(u_{t},\eta_{t})$ is smooth and it is the \underline{only} 
solution of  \eqref{system_of_equations_intro}. 
\end{thmx}

Such a uniqueness result yields also to several interesting algebraic consequences.
For example (as already mentioned) for $\kappa=2$ and $c=0$, we derive  
a parametrization for the moduli space of
minimal immersions of $S$ into a "germ" of  
a hyperbolic $3$-manifold (cf. \cite{Taubes}) by elements of
$\mathcal{T}_{\mathfrak{g}}(S) \times \mathcal{H}^{0,1}(X,E)$, $E=T_{X}^{1,0}$. 
We can lift such information to minimal immersions of the 
Poincar\'e disk $\mathbb{D}$ into the hyperbolic space $\mathbb{H}^{3}$ 
which are equivariant with respect to an irreducible
representation:
$\rho:\pi_{1}(S)\longrightarrow PSL(2,\mathbb{C})$ and 
$PSL(2,\mathbb{C})$ is the
(orientation preserving) isometry group of 
$\mathbb{H}^{3}$ (see \cite{Uhlenbeck}).  
As a consequence, we obtain an analogous parametrization 
for such irreducible representations as well, 
see \cite{Huang_Lucia_Tarantello_2} and \cite{Loftin_Macintosh_1},
\cite{Loftin_Macintosh_2}.  

By taking $\kappa=3$, a similar conclusion may be attained for equivariant
minimal Lagrangian immersions of $\mathbb{D}$
into $\mathbb{C}\mathbb{H}^{2}$. 
Again we refer to \cite{Huang_Lucia_Tarantello_2},\cite{Loftin_Macintosh_1},\cite{Loftin_Macintosh_2} 
and the references therein for details. 

Also (by taking $\kappa=2$ and $t=(1-c^{2})>0$) we derive immediately 
the statement  of Theorem \ref{thm_A} in the introduction. 

\
  
At this point, to find (CMC) immersions of $X$ with constant 
$c\; : \; \vert c \vert \geq 1$,
we need to see whether
$D_{t}$ admits a critical point when we take $t \leq 0$.  
As we shall see, this is not an easy problem to tackle, even for $t=0$. 
Indeed, the functional:
\begin{equation}\label{def_D_0_intro}
D_{0}(u,\eta)
=
\int_{X}
(
\frac{1}{4}\vert \nabla u \vert^{2}
-
u
+
4e^{(\kappa-1)u}\Vert \beta_{0} + \overline{\partial} \eta \Vert^{2}
)
\,dA,
\end{equation}
may no longer be bounded from 
below or coercive in $\Lambda$, and actually the system 
\eqref{system_of_equations_intro} may not admit a solution for $t=0$.
This occurs for example when we take $[\beta]=0$ (i.e. $\beta_{0}=0$), 
where we find:
$u_{t}=\ln \frac{1}{t}\longrightarrow +\infty$, $\eta_{t}=0$, 
$D_{t}(u_{t},\eta_{t})\longrightarrow -\infty$, 
as $t\longrightarrow 0^{+}$, so $D_{0}$ is unbounded from below in $\Lambda$,
and in fact \eqref{system_of_equations_intro} cannot be solved for $t=0$ and
$\beta_{0}=0$.

On the other hand, it is clear that, if for $[\beta]\in \mathcal{H}^{0,1}(X,E)\setminus \{  0 \} $,
the Donaldson functional $D_{0}$ attains its global minimum in $\Lambda$
(i.e. \eqref{system_of_equations_intro} at $t=0$ is solvable),
then, by a simple scaling argument, also the functional corresponding to 
$\lambda[\beta]=[\lambda \beta]$, $\lambda \in \mathbb{C}\setminus \{  0 \} $,
will have the same property.

Furthermore, if we assume the existence of a solution  
for the system \eqref{system_of_equations_intro} at $t=0$, then we can prove that
the corresponding functional 
$D_{0}$ preserves the same properties of $D_{t}$, for $t>0$, 
in the sense that the following holds:
\begin{thm}\label{thm_1}
If $(u_{0},\eta_{0})$ is a (smooth) solution 
for the system \eqref{system_of_equations_intro} with
$t=0$ then,
\begin{enumerate}[label=(\roman*)]
\item $(u_{t},\eta_{t})\longrightarrow (u_{0},\eta_{0})$, 
as $t\longrightarrow 0^{+}$, 
in  $\mathcal{V}_{p}:= H^{1}(X) \times W^{1,p}(X,E)$, $\; \forall \;  p>2$.
\item $D_{0}$ is bounded from below in 
$\Lambda$ and attains its global minimum at 
$(u_{0},\eta_{0})$.
Furthermore $(u_{0},\eta_{0})$ is the \underline{only} critical point of $D_{0}$
and hence the only solution of \eqref{system_of_equations_intro}
for $t=0$ . 	
\end{enumerate}
\end{thm}

In view of Theorem \ref{thm_1}, to identify possible critical points for $D_{0}$, 
we must investigate  whether $(u_{t},\eta_{t})$ 
survives the passage to the limit, as $t\longrightarrow 0^{+}$.  

\

Besides the existence of (CMC)-immersions with $c=\pm 1$, when $\kappa\geq 2$, 
such an asymptotic analysis 
permits to follow the behaviour of the global minimizer $(u_{\lambda},\eta_{\lambda})$
of the Donaldson functional:
\begin{equation}\label{6.11_intro}
D(u,\eta) 
=
\int_{X} 
(
\frac{\vert \nabla u \vert^{2}}{4}
-
u 
+ 
e^{u}
+
4e^{(\kappa-1)u}
\Vert \lambda \beta_{0} + \overline{\partial}\eta \Vert^{2}
) 
\,dA
\end{equation}
along the ray of cohomology classes: 
$[\lambda \beta]$, with  $\lambda \in \mathbb{C}\setminus \{  0 \} $
and
$[\beta]\neq 0$ fixed in $\mathcal{H}^{0,1}(X,E)$. 
Indeed, via the transformations:
\begin{equation*}%\label{t_lambda_transformation}
\begin{split}
&t=\vert \lambda \vert ^{-\frac{2}{\kappa-1}}
,\; 
u_{t}=u_{\lambda}+\frac{2}{\kappa-1}\log \vert \lambda \vert 
,\;
\eta_{t}=\frac{1}{\lambda}\eta_{\lambda},  
\\ &
D_{t}(u_{t},\eta_{t})
=
D(u_{\lambda},\eta_{\lambda})
-
4\pi(\mathfrak{g}-1)\log \vert \lambda \vert ^{\frac{2}{\kappa-1}}
\end{split}
\end{equation*}
we can recast the analysis of $(u_{\lambda},\eta_{\lambda})$
(the global minimum of $D$ in \eqref{6.11_intro}), 
as $\vert \lambda \vert \longrightarrow +\infty$,
to the analysis of $(u_{t},\eta_{t})$ 
(the global minimum of $D_{t}$ in \eqref{D_t}), 
as $t\longrightarrow 0^{+}$.
  
\

We begin our asymptotic analysis by using the strict positivity of 
the Hessian $D_{t}^{\prime \prime }$ at $(u_{t},\eta_{t})$
(see \cite{Goncalves_Uhlenbeck}, \cite{Huang_Lucia_Tarantello_2}) 
and the Implicit Function Theorem to show  
the $C^{2}$-dependence of $(u_{t},\eta_{t})$
with respect to $t\in (0,+\infty)$. More interestingly, we show that 
the expression: 
$ t \int_{X} e^{u_{t}} \,dA$
is \underline{increasing} as a function of 
$t\in (0,+\infty)$, see Lemma \ref{star}. 
Since 
(after integration over $X$ of the first equation in \eqref{system_of_equations_intro})
we have: 
\begin{equation*}%\label{6.17}
t \int_{X} e^{u_{t}}\,dA
+
4(\kappa-1)\int_{X} e^{(k-1)u_{t}} 
\Vert 
\beta_{0} + \overline{\partial}\eta_{t}
\Vert^{2}
\,dA
=
4\pi (\mathfrak{g}-1)
\end{equation*} 
we may conclude that
\begin{equation*}%\label{6.18_intro}
\rho_{t}([\beta]):= 
4(\kappa-1)\int_{X} e^{(\kappa-1)u_{t}}
\Vert \beta_{0} + \overline{\partial}\eta_{t}\Vert^{2}
\,dA 
\in 
(0,4\pi(\mathfrak{g}-1))
\end{equation*}
is \underline{decreasing}  in $t \in (0,+\infty)$.
Thus, it is well defined the value:
\begin{equation}\label{rho_beta_definition}
\rho([\beta])
:=
\lim_{t \to 0^{+} }
\rho_{t}([\beta])
=
\sup_{t>0}\rho_{t}([\beta])
\leq 4\pi (\mathfrak{g}-1),
\end{equation}
which we wish to identify in terms of the fixed cohomology class
$[\beta]\in \mathcal{H}^{0,1}(X,E)$, 
or equivalently in terms of its \underline{harmonic} 
representative $\beta_{0} \in [\beta]$. 
To this purpose we start by showing the following:
\begin{proposition}\label{proposition_on_rho_intro} 
\quad
\begin{enumerate}[label=(\roman*)]
\item $\rho([\beta])=0$ if and only if $[\beta]=0$; and
if $[\beta]\neq 0$ then $\rho([\beta])\geq \frac{4\pi}{\kappa-1}$.
\item 
If $D_{0}$ is bounded from below on $\Lambda$ 
then $\rho([\beta])=4\pi(\mathfrak{g}-1)$. 
\end{enumerate}
\end{proposition}

From (i) it follows in particular that, for $[\beta ]\neq 0$
the interval  $(0,\rho([\beta]))$ gives  the range of 
$\rho_{t}([\beta])$, for $ t \in (0,+\infty)$.  
Furthermore, 
from \eqref{rho_beta_definition} and Proposition \ref{proposition_on_rho_intro}
we derive:
\begin{equation}\label{k_2_g_2_gives_4_pi}
\; \text{ if } \; \kappa=2
\; \text{ and } \;
\mathfrak{g}=2
\; \text{ then } \;
\rho([\beta])=4\pi,
\; \forall \; [\beta] \neq 0.  
\end{equation}
Also Proposition \ref{proposition_on_rho_intro} prompts us to ask 
the following questions:
\begin{enumerate}[label=(\roman*)]
\item[(1)]
if $\rho([\beta])=4\pi(\mathfrak{g}-1)$ then
is it true that $D_{0}$ is bounded from below in $\Lambda$?
\item[(2)]  if $D_{0}$ is bounded from below in $\Lambda$, then for which class 
$[\beta]\neq 0$ is the infimum attained? 
\end{enumerate}

In view of \eqref{k_2_g_2_gives_4_pi},
in the sequel we shall provide an affirmative answer to question (1) and (2)
when the genus $\mathfrak{g}=2$ 
(see Theorem \ref{main_thm_3} and Theorem \ref{thm_infimum_in_lambda_outside_tau_X}). 
In addition, when $\mathfrak{g}\geq 3$,
we shall give strong indications towards an affirmative answer to question (1).
  
On the ground of Theorem \ref{thm_1},  
to explore question (2) for $\mathfrak{g}\geq 3$, 
we shall introduce appropriate tools to investigate the asymptotic behavior 
of the minimizers of $D_{t}$, as $t\longrightarrow 0^{+}$.

To be more precise, 
for fixed $[\beta] \in \mathcal{H}^{0,1}(X,E) \setminus \{ 0 \} $
with harmonic representative $\beta_{0} \in [\beta]\neq 0$
we set:
\begin{equation*}%\label{6.12_prime}
\beta_{t}=\beta_{0}+\overline{\partial}\eta_{t} \in A^{0,1}(X,E)
\; \text{ and } \; 
\alpha_{t}=e^{(\kappa-1)u_{t}}*_{E}\beta_{t} \in C_{\kappa}(X) \setminus \{ 0 \}.
\end{equation*} 
We recall in particular that, $\beta_{t}$ is \underline{harmonic} 
with respect to the metric 
$h=e^{\frac{\kappa-1}{2}u_{t}}g_{X}$ on $X$.

It is a consequence of the Riemann-Roch theorem \cite{Jost}, that any holomorphic 
$\kappa$-differential in $X$
admits $2\kappa(\mathfrak{g}-1)$ zeroes counted with multiplicity.
Thus, we let $Z_{t}$ be the finite set of \underline{distinct} zeroes of $\alpha_{t}$
whose multiplicities add up to $2\kappa(\mathfrak{g}-1)$. 
Hence, in terms of the fiberwise norm of $\alpha_{t}$ we have:
$$\Vert \alpha_{t} \Vert(q)=\Vert \alpha_{t} \Vert_{E^{*}}(q)>0,
\; \forall \; q\in X\setminus Z_{t}.$$ 
Since $C_{\kappa}(X)$ is finite dimensional (see \eqref{dimension_C_kappa_X}),
all norms of $\alpha_{t}$ are equivalent, and it is usual 
(recall the Weil-Petterson form \cite{Jost})
to consider the (well defined) $L^{2}$-norm:
$\Vert \alpha \Vert_{L^{2}}=(\int_{X}\Vert \alpha \Vert^{2}dA)^{\frac{1}{2}}$, 
for $\alpha \in C_{\kappa}(X)$. 
We let,
\begin{equation}\label{def_s_t}
s_{t}\in \R \; : \; 
e^{(\kappa-1)s_{t}}
=
\Vert \alpha_{t} \Vert_{L^{2}}^{2}
\; \text{ and } \; 
\hat{\alpha}_{t}
=
\frac{\alpha_{t}}{\Vert \alpha_{t} \Vert_{L^{2}}}
=
e^{-\frac{(k-1)s_{t}}{2}}\alpha_{t}.
\end{equation}

In order to attain an accurate asymptotic description 
about the behavior of $(u_{t},\eta_{t})$, 
as $t\longrightarrow 0^{+}$, 
we need to account for possible blow-up phenomena of
\begin{equation}\label{u_to_xi}
\xi_{t}:=-u_{t}+s_{t},
\end{equation} 
which occur when there holds: 
 $\max_{X} \xi_{t}\longrightarrow \infty$, as $t\longrightarrow 0^{+}$.

In this respect, it will not suffice to use 
the well known blow-up analysis developed in  
\cite{Brezis_Merle},\cite{Li_Shafrir},\cite{Bartolucci_Tarantello}
for solutions of Liouville equations.
Indeed, we face a particularly delicate situation 
in what we call the "collapsing" case, namely when, 
along a sequence $t_{k}\longrightarrow 0^{+}$, 
we have that $\xi_{k}=\xi_{t_{k}}$ "blows-up" 
around a point
where different zeroes  of $\hat{\alpha}_{k}=\hat{\alpha}_{t_{k}}$ collapse together. 
To be more precise, by \eqref{def_s_t} we may suppose that:
$$
\hat{\alpha}_{k}=\hat{\alpha}_{t_{k}}\longrightarrow \hat{\alpha}_{0},
\; \text{ as } \; k\longrightarrow \infty, 
\; \text{ with } \; 
\hat{\alpha}_{0}\in C_{\kappa}(X)
\; \text{ and } \; 
\Vert \hat{\alpha}_{0} \Vert_{L^{2}}=1,
$$ 
and the $2\kappa(\mathfrak{g}-1)$ zeroes 
(counted with multiplicity) of $\hat{\alpha}_{0}$ 
correspond to the limit points of the zeroes of 
$\hat{\alpha}_{k}$ in $Z_{k}=Z_{t_{k}}$, as $k\longrightarrow \infty$.

In particular, we find a suitable integer
$1\leq N \leq 2\kappa(\mathfrak{g}-1)$ such that (for $k$ large) 
the set $Z_{k}$ of \underline{distinct} zeroes of $\hat{\alpha}_{k}$
is given by:
$$
Z_{k}=\{ z_{1,k},\ldots,z_{N,k} \} 
\; \text{ and } \;  
z_{j,k}\neq z_{l,k},\; l\neq j \in \{ 1,\ldots,N \},
$$
and every $z_{j,k}$ admits multiplicity $n_{j} \in \mathbb{N}$ 
with $\sum_{j=1}^{N}n_{j}=2\kappa(\mathfrak{g}-1)$. 
Furthermore, $z_{j,k}\longrightarrow z_{j}$, as $k\longrightarrow \infty$,
and the set
$$
Z=\{ z_{1},\ldots,z_{N} \} 
$$
is the zero set of $\hat{\alpha}_{0}$. 
However, now we cannot guarantee that the points in $Z$ are \underline{distinct}. 
So we denote by $Z_{0}$ the subset 
(possibly empty)
of $Z$ where different zeroes of $\hat{\alpha}_{k}$
collapse together. Namely,
\begin{equation}\label{definition_Z_0_intro}
\begin{split}
Z_{0}
=
\{ 
z \in Z
\; : \;
&
 \exists \; s\geq 2,\;
1\leq j_{1}<\ldots<j_{s}\leq N
\; \text{ such that} \; \\ 
& 
z=z_{j_{1}}=\ldots=z_{j_{s}}
\; \text{ and } \;
z \not \in Z \setminus \{  z_{j_{1}},\ldots,z_{j_{s}} \}  
\} . 
\end{split}
\end{equation}
The "blow-up" analysis of $\xi_{k}$ needs a particular attention
when blow-up occurs around a point in $Z_{0}$. 

To reduce technicalities, 
we handle such a delicate "collapsing" situation only in case:
$\kappa=2$. 

After \cite{Suzuki_Ohtusuka}, a scenario of blow-up in the "collapsing" situation 
was first illustrated in 
\cite{Lee_Lin_Tarantello_Yang} and \cite{Lin_Tarantello},  
where the new phenomenon of "blow-up without concentration" was recorded. 
See \cite{Lee_Lin_Wei_Yang},\cite{Lee_Lin_Yang_Zhang} for a description of similar
phenomena in the context of systems.
 
A more detailed blow-up analysis was
recently presented in \cite{Tar_1}. 
A first fact, explicitly stated 
in Theorem  \ref{thm_blow_up_global_from_part_1} below,  
allows us to ensure
that (even in the "collapsing" case) 
blow-up can occur around at most a finite number of "blow-up points"
with quantized "blow-up mass" of at least $8\pi$, 
see \cite{Lee_Lin_Tarantello_Yang},\cite{Lee_Lin_Wei_Yang} and \cite{Tar_1}. 

Interestingly, in case blow-up  occurs  
with the least possible blow-up mass  $8\pi$,  
the pointwise estimates obtained in the "collapsing" case in \cite{Tar_1} 
(stated explicitly in Theorem \ref{theorem_from_part_1} below),  
are in striking analogy with the sharp "single bubble" estimates 
obtained in \cite{Chen_Lin_1} and \cite{Li_Harnack}
for the non-vanishing
(hence non-collapsing) case. 
Observe that no "bubble" profile is available in the "collapsing" situation.
We refer the reader to \cite{Tar_1}  for details. 
  
For the sequence $\xi_{k}=\xi_{t_{k}}$,  
we can use Theorem \ref{thm_blow_up_global_from_part_1} to conclude that,
\begin{enumerate}[label=(\roman*)]
\item 	either (compactness) \;:\;
$
\xi_{k}\longrightarrow \xi_{0}
\; \text{ in   } \; C^{2}(X)
,  \; \text{ as } \;  
k\longrightarrow +\infty,
$
and the functional $D_{0}$ is bounded from below and attains its infimum in $\Lambda$;
\item[(ii)] or (blow-up) \;:\;
$\xi_{k}$ admits a finite blow-up set 
\begin{equation*}%\label{S_the_blow_up_set_intro}
\mathcal{S}=\{ q_{1},\ldots,q_{m}  \}
\; \text{ with } \; 
m \in \{   1,\ldots,\mathfrak{g}-1\},
\end{equation*} 
(i.e. $\xi_{k}^{+}$ is bounded uniformly on compact sets of $X\setminus \mathcal{S}$)  
with \underline{blow-up mass}: 
\begin{equation}\label{sigma_q} 
\sigma(q)
:=
\lim_{r \to 0^{+} }
\left(
\lim_{k \to +\infty }  8 \int_{B_{r}(q)} e^{u_{t_{k}}}
\Vert \beta_{0} + \overline{\partial}\eta_{t_{k}} \Vert^{2}dA
\right)
\in 
8\pi \N
,
\; \forall \;  
q\in \mathcal{S};
\end{equation}
(recall that $\kappa=2$). 
Furthermore, we may have "blow-up with concentration", or "blow-up without concentration", 
as described respectively in part (ii)-(a) or (ii)-(b) of 
Theorem \ref{thm_blow_up_global_from_part_1} below.
\end{enumerate}

Thus,  we focus on the "blow-up" case.
Due to the minimizing property of $(u_{t},\eta_{t})$, we expect that
the corresponding "blow-up mass" should be the least possible
(namely $8\pi$). 
In fact, it is likely that $\xi_{k}$ admits only one blow-up point
(see Remark \ref{rem_only_one_blow_up_point})
and it cannot occur in $Z\setminus Z_{0}$  
since, by \cite{Bartolucci_Tarantello}, it would have a "blow-up mass" larger than $8\pi$. 
This latter possibility could also be excluded  
by the "vanishing condition", 
recently identified by Wei and Zhang in 
\cite{Wei_Zhang_1},\cite{Wei_Zhang_2}.  

Therefore, we expect
that $\xi_{k}$ blows up either at a point in $X\setminus Z$
(not a zero of $\hat{\alpha}_{0}$) 
or it must be a point of 
"collapsing" zeroes in $Z_{0}$.
So, next we shall analyze the blow-up behavior of $\xi_{k}$ in those situations. 

To this purpose, for given \underline{distinct} points  
$\{ x_{1},\ldots,x_{\nu} \} \subset X $
we introduce the following 
subspace of
$C_{2}(X)$:
\begin{equation*}%\label{Q_2_P}
Q_{2}[ x_{1},\ldots,x_{\nu} ] 
=
\{ 
\alpha \in C_{2}(X) 
\; : \; 
\alpha \; \text{vanishes at  } \;
x_{1},\ldots, x_{\nu}
\},
\end{equation*}
and, by the Riemann-Roch theorem, we know that,
\begin{equation}\label{dimension_of_Q_2}
\dim_{\mathbb{C}} Q_{2}[ x_{1},\ldots,x_{\nu} ]
=
3(\mathfrak{g}-1)-\nu, 
\; \text{ for } \; 
1\leq \nu < 2(\mathfrak{g}-1),  
\end{equation}
see Section \ref{Preliminaries} for details. 

We start by considering the case $\mathcal{S} \cap Z = \emptyset$, 
where we know that only "blow-up with concentration" occurs
(\cite{Brezis_Merle},\cite{Li_Harnack}) and so, for $\alpha_{k}=\alpha_{t_{k}}$, we have:
\begin{equation*}
8e^{u_{k}}\Vert \alpha_{k} \Vert^{2}
=
8\Vert \hat{\alpha}_{k} \Vert^{2}e^{\xi_{k}}
\rightharpoonup
8\pi \sum_{l=1}^{m}\delta_{q_{l}} 
\; \text{ weakly in the sense of measures,} \;
\end{equation*}
and $\rho([\beta]) = 4\pi m$. We have:

\begin{thm}\label{thm_main_1_intro}
Let $\kappa=2$, $[\beta]\neq 0$ and $\mathcal{S}=\{ q_{1},\ldots,q_{m} \}$ 
with $m\in \{ 1,\ldots, \mathfrak{g-1} \} $
be the (non empty) blow-up set of $\xi_{k}=\xi_{t_{k}}$. 
If $\mathcal{S} \cap Z=\emptyset$ 
then (along a subsequence), as $k\longrightarrow +\infty$, we have:
\begin{align}
&
\alpha_{k}\longrightarrow \alpha_{0} \in C_{2}(X)
\; \text{ (in any norm) with } \; 
\alpha_{0}\neq 0
\; \text{  vanishing exactly at  } \; 
Z \notag
\\ 
&
e^{-u_{t_{k}}}
\rightharpoonup
\pi \sum_{q\in \mathcal{S}}\frac{1}{\Vert \alpha_{0} \Vert^{2}(q)}\delta_{q},
\; \text{  weakly in the sense of measures } \; 
\notag 
\\
&
c_{k}=D_{t_{k}}(u_{t_{k}},\eta_{t_{k}})
=
-4\pi(\mathfrak{g}-1-m)d_{k}+O(1),
\; \text{with} \; 
d_{k}=\fint_{X}u_{t_{k}}\,dA \longrightarrow +\infty.
\label{c_k_estimate_intro} \\
&
\int_{X} \beta_{0} \wedge \alpha =0,
\; \forall \; 
\alpha \in Q_{2}[ q_{1},\ldots,q_{m} ].
\label{orthogonality_condition_intro}
\end{align}	
Furthermore,  
$\rho([\beta])=4\int_{X} \beta_{0} \wedge \alpha_{0} =4\pi m$. 
\end{thm}	
Since $\dim_{\mathbb{C}}Q_{2}[q_{1},\ldots,q_{m}]=3(\mathfrak{g}-1)-m$,
then the orthogonality condition \eqref{orthogonality_condition_intro} 
together with the estimate \eqref{c_k_estimate_intro} 
for the global minimum value of $D_{t_{k}}$ 
support the fact that $\xi_{k}$ should admit only 
\underline{one} blow-up point ($m=1$),  
where the  holomorphic quadratic differential
$*_{E} \beta_{0}$
does not vanish. 
This would also match with the asymptotic analysis in 
\cite{Huang_Lucia_Tarantello_1} 
concerning minimal immersions with fixed second fundamental form.

Furthermore, the estimate
\eqref{c_k_estimate_intro} allows us to answer question (1), 
posed above, in case $\mathcal{S}\cap Z=\emptyset$.  
Indeed, if $\rho([\beta])=4\pi(\mathfrak{g}-1)$ then 
$m=\mathfrak{g}-1$ and therefore, by using \eqref{c_k_estimate_intro}, 
we find that $D_{0}$ is bounded from below in $\Lambda$. 
However, it remains as a challenging open question to see whether 
$D_{0}$ attains its infimum in $\Lambda$ in this case. 
The orthogonality condition \eqref{orthogonality_condition_intro}, 
seems to suggest that, 
for "almost all" classes $[\beta]$ the infimum must be attained. 
 
This is confirmed here for the case $m=1$ 
(i.e. $\mathfrak{g}=2$) 
in Theorem \ref{thm_infimum_in_lambda_outside_tau_X}, where 
the orthogonality condition \eqref{orthogonality_condition_intro} is nicely 
interpreted in terms of the Kodaira map of $X$ into
the projective space $\mathbb{P}(V^{*})$, with $V=C_{2}(X)$. 
Extension of such a fact for $m\geq 2$ will be discussed in future work.

\

Next, we wish to acquire some useful information about the blow-up behavior
of $(u_{t_{k}},\eta_{t_{k}})$ in the "collapsing" case and when 
blow-up occurs with "least" blow-up mass $8\pi$. 
Thus, we assume that, in \eqref{sigma_q}, there holds:
\begin{equation}\label{condition_least_blow_up_mass_intro}
\sigma(q)=8\pi, \; \forall \; q\in \mathcal{S}.
\end{equation}

When we assume \eqref{condition_least_blow_up_mass_intro}  
then every blow-up point 
$q\in \mathcal{S}\cap Z$ must correspond to a "collapsing" of zeroes, 
i.e. $q\in Z_{0}$ and so $\mathcal{S} \cap Z = \mathcal{S} \cap Z_{0}$. 
As a consequence, under \eqref{condition_least_blow_up_mass_intro}
the conclusion of Theorem \ref{thm_main_1_intro} holds
under the (weaker) assumption that, 
$\mathcal{S} \cap Z_{0} = \emptyset$, see Corollary \ref{cor_6.8}. 

To avoid technicalities, but still
give a flavour of the blow-up behaviour in the "collapsing" case,
we shall state here our result in the case of a single blow-up point 
(i.e. $\mathcal{S}=\{ q \}$).
Thus we find:
\begin{equation*}
x_{k} \in X
\; : \; 
x_{k}\longrightarrow q
\; \text{ and  } \; 
\xi_{k}(x_{k})=\max_{X} \xi_{k}
\longrightarrow \infty,
\; \text{ as } \; 
k\longrightarrow \infty.
\end{equation*}
In particular, notice that if $q\in Z$, then 
$\Vert \hat{\alpha}_{t_{k}} \Vert(x_{k})\longrightarrow 0$,
as $k\longrightarrow \infty$, 
and we prove:
\begin{thm}\label{thm_last_intro}
Assume \eqref{condition_least_blow_up_mass_intro} 
and let $\mathcal{S}=\{ q \}$. 
If $q\in Z$ (i.e. $\hat{\alpha}_{0}(q)=0$) then $q\in Z_{0}$ and 
(along a subsequence), as $k\longrightarrow +\infty$, we have:
\begin{enumerate}[label=(\roman*)]
\item 
$s_{k}\longrightarrow +\infty$,\; 
$
\Vert \alpha_{t_{k}} \Vert^{2}(x_{k})
=
e^{s_{k}}\Vert \hat{\alpha}_{t_{k}} \Vert^{2}(x_{k})\longrightarrow \mu >0 
$
\item 
$
e^{-u_{t_{k}}}
\longrightarrow  
\frac{\pi}{\mu}\delta_{q}$,	
weakly in the sense of measures;
\item
$c_{k}=-4\pi(\mathfrak{g}-2)d_{k}+O(1)$
with 
$d_{k}=\fint_{X} u_{t_{k}} \,dA \longrightarrow +\infty, 
$
\item

\quad \vspace{-18pt}
\begin{equation}\label{orthogonality_condition_yet_again}
\hspace{-138pt}
\int_{X} \beta_{0}\wedge \alpha  =0,
\; \text{ if and only if } \; 
\alpha \in Q_{2}[q ],
\end{equation}
In particular,
$\int_{X} \beta_{0}\wedge \hat{\alpha}_{0}=0$. 	
\end{enumerate}
\end{thm} 	
\noindent
We refer to Theorem \ref{main_thm_2} and Corollary \ref{cor_6.10} 
for complete statements of our results in case
$\mathcal{S}$ contains more than one blow-up point. 

\ 

Theorem \ref{thm_main_1_intro}  and Theorem \ref{thm_last_intro} apply 
particularly well in case $\kappa=2$ and $\mathfrak{g}=2$, 
where by \eqref{k_2_g_2_gives_4_pi} we know that, 
$\rho([\beta])=4\pi$, 
for every $[\beta] \in \mathcal{H}^{0,1}(X,E)\setminus \{ 0 \} $. 
This implies that the blow-up set $\mathcal{S}$ (if not empty) must 
contain one single blow-up point ($m=1$).
Consequently, we can summarize the above results into the following:
\begin{thm}\label{main_thm_3}
Let $\kappa=2$ and genus $\mathfrak{g}=2$. Then for any $[\beta]\neq 0$
we have $\rho([\beta])=4\pi$ and
the Donaldson functional $D_{0}$ is bounded from below in $\Lambda$.

Furthermore, if $D_{0}$ does \underline{not} attain its infimum in
$\Lambda$ then there exists $q\in X$ such that,
$*_{E}\beta_{0}(q)\neq 0$ and
\begin{equation}\label{if_and_only_if}
\int_{X}\beta_{0}\wedge \alpha=0   
\Longleftrightarrow
\alpha \in Q_{2}[q]. 
\end{equation}
\end{thm}
Note that, Theorem \ref{main_thm_3} contains a nontrivial information, 
since for $[\beta]=0$ the functional $D_{0}$ is always unbounded below. 

\

Furthermore, from \eqref{dimension_of_Q_2} we know: 
$\dim_{\mathbb{C}}Q_{2}[q]=\dim_{\mathbb{C}}C_{2}(X)-1$ 
so that the complex vector space 
$V=C_{2}(X)=H^{0}(X,\otimes^{2}_{k=1}(T_{X}^{1,0})^{*})$ 
is base-point free.
Hence, by following 
section 12.1.3 of \cite{Donaldson_Book},
for any $q\in X$ we can identify a one dimensional subspace of 
$V^{*}=\mathcal{H}^{0,1}(X,E)$, $E=T^{1,0}_{X}$; namely
an element of $\mathbb{P}(\mathcal{H}^{0,1}(X,E))$
given by the ray of functionals in $V^{*}$ which admit $Q_{2}[q]$ as their kernel. 
Thus, it is well defined the Kodaira map:
\begin{equation}\label{kodaira_map}
\tau:X \longrightarrow \mathbb{P}(\mathcal{H}^{0,1}(X,E)), \; E=T^{1,0}_{X}
\end{equation}
and it is easy to check that $\tau$ is holomorphic, see
\cite{Donaldson_Book},\cite{Griffiths_Harris}. 
More importantly, 
$$
\; \text{ \eqref{if_and_only_if} holds } \;
\Longleftrightarrow
\;\;\, [\beta]_{\mathbb{P}}\in \tau(X) 
$$
with $[\beta]_{\mathbb{P}}$ the representative 
in $\mathbb{P}(\mathcal{H}^{0,1}(X,E))$ 
of the class
$[\beta] \in \mathcal{H}^{0,1}(X,E )\setminus \{ 0 \}$ 
identified by the harmonic  $\beta_{0} \in [\beta]$. 

We shall show in Lemma \ref{lem_complex_curve} that 
$\tau(X)$ defines a complex curve (i.e. of complex dimension 1) into 
$\mathbb{P}(\mathcal{H}^{0,1}(X,E)) \simeq \mathbb{P}^{3\mathfrak{g}-4}$ 
(recall \eqref{dimension_C_kappa_X} with $\kappa=2$)
and we have: $3\mathfrak{g}-4 \geq 2$ for $\mathfrak{g}\geq 2$. Therefore,
$\tau(X)$ is a "tiny" subset of $\mathbb{P}(\mathcal{H}^{0,1}(X,E)) $,
and  $[\beta]_{\mathbb{P}}\not \in \tau(X)$ for "generic"
$[\beta] \in \mathcal{H}^{0,1}(X,E)\setminus \{  0 \} $.

\begin{remark}
In case $[\beta]_{\mathbb{P}} \not \in \tau(X)$,
it follows from Theorem \ref{thm_last_intro} that, 
if $\mathcal{S}$ is not empty and  \eqref{condition_least_blow_up_mass_intro}   holds, 
then $\mathcal{S}$ must contain at least \underline{two} points. 
\end{remark}	

More interestingly, from Theorem \ref{main_thm_3} we conclude: 

\begin{thm}\label{thm_infimum_in_lambda_outside_tau_X}
Let $\kappa=2$ and $\mathfrak{g}=2$. Then for any
$[\beta] \in \mathcal{H}^{0,1}(X,E)\setminus \{  0 \} $, 
$E=T^{1,0}_{X} $ such that $ [\beta]_{\mathbb{P}}\not \in \tau(X)$, 
the functional $D_{0}$ attains its infimum in $\Lambda$.
\end{thm}	

Clearly, Theorem \ref{thm_one} of the introduction 
is a direct consequence of Theorem \ref{thm_1}
and Theorem \ref{thm_infimum_in_lambda_outside_tau_X}.

\

Finally, we wish to comment about the case left out by our analysis.
Namely when blow-up occurs at a zero of $\hat{\alpha }$
corresponding to the limit of "non-collapsing" zeroes of $\hat{\alpha}_{k}$,
and hence preserving their same multiplicity, 
say $n\in \{ 1,\ldots,4(\mathfrak{g}-1) \} $. 
More specifically, in this case we need to deal (after scaling) 
with possible "multiple" bubble profiles "symmetrically" placed 
at the $(n+1)$-roots of the identity, 
as described in 
\cite{Bartolucci_Tarantello_JDE} and \cite{Wei_Zhang_1}, \cite{Wei_Zhang_2}.

Such symmetry is at the origin of crucial cancellations in the "blow-up" estimates,
which prevent to obtain useful information,
for example about the sequence $s_{k}$.
The nontrivial new estimates established 
recently by Wei and Zhang in \cite{Wei_Zhang_1} and \cite{Wei_Zhang_2} 
for  such a  non-simple "blow-up" situation, 
may help to overcome those difficulties and to accomplish analogous conclusions 
as in Corollary \ref{cor_6.10} below. 

\section{Preliminaries}\label{Preliminaries}

In this section we introduce the basic notations and recall some facts useful in the sequel. 

\

Let $X\in \mathcal{T}_{\mathfrak{g}}(S)$ be a given Riemann surface with 
hyperbolic metric $g_{X}$ and induced scalar product
$\langle\cdot,\cdot\rangle$, norm $\Vert \cdot \Vert$ and volume 
element $dA$. 
We let
\begin{equation*}%\label{holomorphic_tangent_bundle_of_X}
T_{X}^{1,0}
=
\text{ holomorphic tangent bundle of $X$  } \; 
\end{equation*}
equipped with the corresponding complex, holomorphic and 
Hermitian structure. For given $\kappa \geq 2$, we consider 
the holomorphic line  bundle:
\begin{equation}\label{E}
E=\otimes^{\kappa-1}T^{1,0}_{X}
\end{equation}
with holomorphic structure
$\bar{\partial}=\bar{\partial}_{E}$ and the corresponding spaces:
\begin{align}
& A^{0}(E)=
\{\text{ smooth sections of $E$ }\},
\notag
\\
& 
A^{0,1}(X,E)
=
\{ \text{ $(0,1)$-forms valued on $E$ }  \},
\notag \\
&
A^{1,0}(X,E^{*})
=
\{ \text{ $(1,0)$-forms valued on $E^{*}$ }  \}. 
\notag 
\end{align}	
We can use the  d-bar operator
$\bar{\partial}_{E}:A^{0}(X) \longrightarrow A^{0,1}(X,E)$
to define the $(0,1)$-Dolbeault cohomology group:
\begin{equation}\label{H_0_1_definition}
\mathcal{H}^{0,1}(X,E)=A^{0,1}(X,E)/\bar{\partial}(A^{0}(E))
\end{equation}
with cohomology class $[\beta]\in \mathcal{H}^{0,1}(X,E)$ given as follows:
$$[\beta]=\{ \beta + \bar{\partial}\eta \in A^{0,1}(X,E),
\; \forall \; \eta \in A^{0}(E) \}.$$ 

Furthermore we use the induced Hermitian structure on $E$, 
to define the fiberwise Hermitian product
$\langle \cdot,\cdot \rangle_{E}$ and norm $\Vert \cdot \Vert_{E}$
for sections of $A^{0}(E)$ and $(0,1)$-forms in $A^{0,1}(X,E)$. So, for $p>1$, 
we obtain the $L^{p}$-spaces of sections and forms valued on $E$, respectively 
as follows:
\begin{align}
&
L^{p}(X,E)
=
\{ \eta:X\longrightarrow E 
\; : \; 
\Vert  \eta \Vert_{L^{p}}:=
(\int_{X}\Vert \eta \Vert_{E}^{p}dA)^{\frac{1}{p}} < +\infty
\},
\notag
\\
&
L^{p}(A^{0,1}(X,E))
=
\{ 
\beta \in A^{0,1}(X,E)
\; : \; 
\Vert \beta \Vert_{L^{p}}
:=
(\int_{X}\Vert \beta \Vert_{E}^{p}dA)^{\frac{1}{p}} < +\infty
\}
\notag
,  
\end{align}	
which define a Banach space equipped with the given norm: 
$\Vert \cdot  \Vert_{L^{p}}$.
Thus, for $p\geq 1$, we obtain also the Sobolev space:
\begin{equation}\label{W_1_p}
W^{1,p}(X,E)
=
\{ \eta \in L^{p}(X,E)
\; : \; 
\bar{\partial} \eta \in L^{p}(A^{0,1}(X,E))
\}, 
\end{equation}
again a Banach space equipped with the norm:
$$
\Vert \eta \Vert_{W^{1,p}}
=
\Vert \eta \Vert_{L^{p}} + \Vert \bar{\partial}\eta \Vert_{L^{p}},
\; \forall \; \eta \in W^{1,p}(X,E).
$$
Actually, for the holomorphic line bundle $E$ in \eqref{E},
the following Poincar\'e inequality holds:
\begin{lemma}[\cite{Huang_Lucia_Tarantello_2}] %\label{poincare_inequality}
Let $E=\otimes^{\kappa-1}T^{1,0}_{X} $ and $p>1$.  
Then there exists a suitable constant $C_{p}>0$ such that,
\begin{equation}\label{poincare_inequality_equation}
\Vert \eta \Vert_{L^{p}} \leq C_{p}\Vert \bar{\partial}\eta \Vert_{L^{p}},
\; \forall \; \eta \in W^{1,p}(X,E).
\end{equation}
\end{lemma}	
\begin{proof}
See proposition 2.2 in \cite{Huang_Lucia_Tarantello_2}. 
\end{proof}
In view of \eqref{poincare_inequality_equation}, we obtain the equivalence
of the following norms:
\begin{equation*}%\label{norm_equivalence}
\Vert \eta \Vert_{W^{1,p}} \simeq \Vert \bar{\partial} \eta\Vert_{L^{p}}, 
\;
\eta \in W^{1,p}(X,E), p>1.  
\end{equation*}
In general, the sharp value of the constant $C_{p}$ 
in \eqref{poincare_inequality_equation} is not known, 
unless we take 
$p=2$, where the following holds: 
\begin{lemma}[\cite{Huang_Lucia_Tarantello_2}]
Let $E=\otimes^{\kappa-1}T^{1,0}_{X}$ and 
$h$ be a metric on $X$ with 
Gaussian curvature $K_{h}$. 
Then
\begin{equation*}%\label{poincare_constant_estimate}
\int_{X}\langle \bar{\partial} \eta, \bar{\partial} \eta \rangle_{h} dA_{h}
\geq
-(\kappa-1)\int_{X} K_{h} \langle \eta,\eta \rangle_{h} dA_{h},
\; \forall \;  \eta \in W^{1,2}(X,E)
\end{equation*}
with 
$\langle \cdot,\cdot \rangle_{h}=\langle \cdot,\cdot \rangle_{E,h}$ the fiberwise
Hermitian product on $E$  and $dA_{h}$ the volume form induced by the metric $h$. 
\end{lemma}	
\begin{proof}
See proposition 2.1 in \cite{Huang_Lucia_Tarantello_2}. 
\end{proof}
In particular, if $h=e^{2\phi}g_{X}$ with $\phi$ a smooth function in $X$, 
then by the following transformation rules:
\begin{align}
& 
\langle \eta,\eta \rangle_{h}
=
e^{4(\kappa-1)\phi} \langle \eta,\eta \rangle_{h},\;
\langle \bar{\partial}\eta,\bar{\partial}\eta \rangle_{h}
=
e^{4(\kappa-1)\phi}\langle \bar{\partial}\eta,\bar{\partial}\eta \rangle_{h} e^{-2\phi}, 
\notag
\\ &
dA_{h}=e^{2\phi}dA 
,\; K_{h}=e^{-2\phi}(-\Delta \phi -1),
\notag 
\end{align}	
we deduce the following: 
\begin{equation}\label{poincare_conformal_estimates}
\int_{X}e^{4(\kappa-1)\phi}\langle \bar{\partial} \eta,\bar{\partial}\eta \rangle_{E}dA
\geq 
(\kappa-1)
\int_{X}
(\Delta \phi +1)e^{4(\kappa-1)\phi}\langle \eta,\eta \rangle_{E}dA
\end{equation} 
In particular, for $\phi=0$, that is $h=g_{X}$, then 
\eqref{poincare_conformal_estimates} yields to 
the following \underline{sharp} Poincar\'e inequality, 
valid for the hyperbolic metric $g_{X}$ on $X$:
$$
\int_{X}\langle \bar{\partial} \eta,\bar{\partial}\eta\rangle_{E}dA
\geq (\kappa-1) \int_{X} \langle \eta,\eta \rangle_{E}dA,
\; 
\eta \in W^{1,2}(X,E). 
$$
Interestingly, from \eqref{poincare_conformal_estimates}, 
we obtain the following estimate
for solutions of \eqref{system_of_equations_intro}.
\begin{corollary}\label{cor_d_bar_eta_estimate}
If $(u,\eta)$ is a solution of \eqref{system_of_equations_intro} with $t\geq 0$,  
then
\begin{equation}\label{estimate_of_cor_on_d_bar_eta}
\begin{split}
\int_{X} 
\langle \bar{\partial} \eta, \bar{\partial} \eta \rangle_{E} e^{(\kappa-1)u}dA 
\geq  & \;
2(\kappa-1)^{2}\int_{X}
\Vert \beta_{0} + \bar{\partial}\eta \Vert^{2}_{E} \Vert \eta \Vert_{E}^{2}e^{2(\kappa-1)u}dA \\
& +
\frac{\kappa-1}{2} \int_{X}
\Vert \eta \Vert_{E}^{2} e^{(\kappa-1)u} dA,
\; \forall \; \eta \in A^{0}(E).
\end{split}
\end{equation} 
\end{corollary}	
\begin{proof}
It suffices to take $\phi=\frac{u}{4}$ in \eqref{poincare_conformal_estimates}.  
\end{proof}

Next, we recall the (natural) \underline{wedge product} defined on forms 
valued on $E$ and $E^{*}$ respectively. 
For 
$\alpha \in A^{1,0}(X,E^{*})=A^{1,0}(X,\mathbb{C}) \otimes E^{*}$,
let 
$\alpha \in \alpha_{0} \otimes e^{*}$ 
with 
$\alpha_{0} \in A^{1,0}(X,\mathbb{C})$ and $e^{*} \in A^{0}(E^{*})$, 
and for 
$\beta \in A^{0,1}(X,E)= A^{0,1}(X,\mathbb{C}) \otimes E$, 
let 
$\beta = \beta_{0} \otimes e$ 
with 
$\beta_{0} \in A^{0,1}(X,\mathbb{C})$ and $e\in A^{0}(E)$, 
then
$$
\beta \wedge \alpha
=
e^{*}(e) \beta_{0}\wedge \alpha_{0} \in A^{1,1}(X,\mathbb{C}), 
$$
where $\beta_{0}\wedge \alpha_{0}$ is the usual wedge product of complex valued forms. 

By the well known properties of the wedge product we can define the 
bilinear form: 
\begin{equation}\label{wedge_product_map}
\begin{split}
 A^{1,0}(X,E^{*}) \times A^{0,1}(X,E) \longrightarrow \mathbb{C}
\;:\;
(\alpha,\beta)
\longrightarrow 
\int_{X}   \beta \wedge \alpha,
\end{split}
\end{equation}
which, by Serre duality (see \cite{Voisin}), is non-degenerate and induces the isomorphism:
\begin{equation}\label{isomorphism_A_1_0_To_A_0_1}
A^{1,0}(X,E^{*})
\simeq 
(A^{0,1}(X,E))^{*}.
\end{equation}
Furthermore, by the given (fiberwise) Hermitian product on $E$, 
we have the duality map:
$$
A^{0,1}(X,E) \longrightarrow  (A^{0,1}(X,E))^{*}
:
\beta \longrightarrow \beta^{*}
:
\beta^{*}(\xi)
=
\int_{X} \langle \xi,\beta\rangle dA,
$$
expressing the isomorphism between
$A^{0,1}(X,E)$ and its dual. 
In turn, by \eqref{isomorphism_A_1_0_To_A_0_1}, we obtain the \underline{isomophism} 
$A^{1,0}(X,E^{*}) \simeq A^{0,1}(X,E)$,
explicitly expressed by the Hodge star operator:
\begin{equation*}%\label{hodge_operator_intro}
*_{E} : A^{0,1}(X,E) \longrightarrow  A^{1,0}(X,E^{*}),
\end{equation*}
where for $\beta \in A^{0,1}(X,E)$ we obtain 
$*_{E}\beta$ according to the following relation: 
$$ 
\xi \wedge *_{E}\beta  
=
\langle \xi,\beta \rangle_{E} \,dA,
\; \forall \; \xi \in A^{0,1}(X,E). 
$$
We denote by $*_{E}^{-1}$ the inverse operator, and observe that 
$*_{E}$ defines an isometry with respect to the (fiberwise) Hermitian product on $E^{*}$:
$$
\langle \cdot,\cdot \rangle_{E^{*}}
=
\langle *_{E}^{-1} \cdot, *_{E}^{-1} \cdot \rangle_{E}
\; \text{ and corresponding norm } \;
\Vert \cdot \Vert_{E^{*}}=\Vert *_{E}^{-1} \cdot \Vert_{E}. 
$$
As a relevant (finite dimensional) subspace of $A^{1,0}(X,E^{*})$ we consider
the space of holomorphic $\kappa$-differentials, namely the holomorphic sections
of $E^{*}\otimes K_{X}=\otimes^{\kappa}K_{X}=\otimes^{\kappa}(T^{1,0}_{X})^{*}$\;:

$$
C_{\kappa}(X)
= H^{0}(X,E^{*}\otimes K_{X}) 
:=
\{ 
\alpha \in A^{0}(X,\otimes^{k}(T_{X}^{1,0})^{*})
\; : \; 
\bar{\partial} \alpha=0
\}
\subset
A^{1,0}(X,E^{*}) 
$$ 
with $\bar{\partial}$ the holomorphic structure of the line bundle
$\otimes^{\kappa} (T_{X}^{1,0})^{*}$.
We recall that, 
$\dim_{\mathbb{C}}C_{\kappa}(X)=(2\kappa-1)(\mathfrak{g}-1)$.

Since, by Stokes theorem, for $\alpha \in C_{\kappa}(X)$ we have:
$
\int_{X} \bar{\partial}\eta  \wedge \alpha = 0,
\; \forall \; \eta \in A^{0}(E) 
$, 
we see that the bilinear form \eqref{wedge_product_map} is 
well defined and non degenerate
when considered on the space:
$C_{\kappa}(X) \times \mathcal{H}^{0,1}(X,E)$, 
and so it induces the isomorphism:
\begin{equation}\label{C_kappa_X_isometry}
\begin{split}
C_{\kappa}(X) \simeq (\mathcal{H}^{0,1}(X,E))^{*}.
\end{split}
\end{equation}
Since $C_{\kappa}(X)$ is finite dimensional, then we also
get the "dual" isomorphism:
$(C_{\kappa}(X))^{*} \simeq \mathcal{H}^{0,1}(X,E)$,
and to identify the linear (complex) functional on $C_{\kappa}(X)$
associated (by \eqref{wedge_product_map}) to a class 
$[\beta] \in \mathcal{H}^{0,1}(X,E)$, we recall that,
by Dolbeault decomposition, every 
$\beta \in A^{0,1}(X,E)$
can be uniquely written as follows: 
$$
\beta = \beta_{0} + \bar{\partial}\eta
\; \text{ with  $\beta_{0}$ \underline{harmonic} (with respect to $g_{X}$) and } \; 
 \eta \in A^{0}(E).
$$
Hence, every class  $[\beta] \in \mathcal{H}^{0,1}(X,E)$ is uniquely identified
by its harmonic representative $\beta_{0} \in [\beta]$.
Furthermore, for 
$[\beta]\in \mathcal{H}^{0,1}(X,E)$ with harmonic $\beta_{0} \in [\beta]$,
we obtain an element of $(C_{\kappa}(X))^{*}$ as follows:
\begin{equation}\label{class_beta_to_C_kappa_star}
C_{\kappa}(X)\longrightarrow \mathbb{C}
:
\alpha \longrightarrow \int_{X} \beta_{0} \wedge \alpha,
\end{equation}
which is well defined independently from the chosen element in $ [\beta]$.

In addition, harmonic $(0,1)$-forms valued in $E$ 
are characterized by the property: 
$\int_{X} \langle \bar{\partial} \eta,\beta_{0}\rangle_{E}dA=0$,
$\; \forall \; \eta \in A^{0}(E)$ and so,
$$
\beta_{0} \in A^{0,1}(X,E) \; \text{ harmonic } \;
\Longleftrightarrow
\;
*_{E}\beta_{0}\in C_{\kappa}(X). 
$$
So, every harmonic $\beta_{0}\in A^{0,1}(X,E)$ also identifies an element of
$(\mathcal{H}^{0,1}(X,E))^{*}$ via  the linear map:
\begin{equation*}
\begin{split}
\mathcal{H}^{0,1}(X,E)\longrightarrow \mathbb{C}
\; : \; 
[\xi]\longrightarrow 
\int_{X} \xi \wedge  *_{E} \beta_{0} 
=
\int_{X}\langle \xi , \beta_{0}  \rangle_{E} dA
, 
\end{split}
\end{equation*}
and (by \eqref{C_kappa_X_isometry}) we identify the isomorphism:
$$
\mathcal{H}^{0,1}(X,E) \longrightarrow C_{\kappa}(X)
:
[\beta] \longrightarrow  *_{E}\beta_{0}.
$$
In other words, for $\alpha \in C_{\kappa}(X) \subset A^{1,0}(X,E^{*})$,  
$*_{E}^{-1}\alpha$ is given by the unique \underline{harmonic}
$(0,1)$-form $\beta_{0}$ valued in $E$ : $*_{E}\beta_{0}=\alpha$. 

Since all norms on $C_{\kappa}(X)$ are equivalent, it is usual
(recall the Weil-Patterson form \cite{Jost}) to consider the following $L^{2}$-norm
$$
\Vert \alpha \Vert_{L^{2}}
:=
(\int_{X} \langle *_{E}^{-1} \alpha,*_{E}^{-1} \alpha\rangle_{E}dA)^{\frac{1}{2}}
\; \text{ for } \; 
\alpha \in C_{\kappa}(X). 
$$
Indeed, it is conveniently computed with respect to a fixed \underline{basis} in
$C_{\kappa}(X)$ given as follows: 
\begin{equation}\label{basis_for_C_k_X}
\{s_{1},\ldots,s_{\nu}\} \subset C_{\kappa}(X)
\; \text{ with } \; 
\nu=(2\kappa-1)(\mathfrak{g}-1)
,\;
\int_{X}\langle *_{E}^{-1}s_{j},*_{E}^{-1}s_{k}\rangle_{E} dA
=
\delta_{j,k}
\end{equation}
and $\delta_{j,k}$ are the Kronecker symbols. So, 
for $\alpha \in C_{\kappa}(X)$, 
we may write:
$$\alpha=\sum_{j=1}^{\nu}a_{j}s_{j}, \;a_{j}\in \mathbb{C},
\; \text{ and } \; 
\beta_{0}=*_{E}^{-1}\alpha=\sum_{j=1}^{\nu}a_{j}*_{E}^{-1}s_{j}
$$ 
($\beta_{0}$ the associated harmonic $(0,1)$-form valued on $E$) 
and compute:
 \begin{equation*}%\label{norm_of_alpha}
\Vert \alpha \Vert_{L^{2}}^{2}=\Vert \beta_{0} \Vert^{2}_{L^{2}}
=
\sum_{j=1}^{\nu}\vert a_{j} \vert^{2}.  
\end{equation*}
Conveniently, we have "compactness" for any sequence 
$\alpha_{n} \in C_{\kappa}(X)$ with uniformly bounded $L^{2}$-norm. 
Thus, for example, if $\alpha_{n}$ satisfies 
$\Vert \alpha_{n} \Vert_{L^{2}}=1$ then it admits a convergent subsequence 
$\alpha_{n_{k}}\longrightarrow \alpha_{0}  \in C_{\kappa}(X)$ 
with 
$\Vert \alpha_{0} \Vert_{L^{2}}=1$. 

\

Next, we recall some well known consequences of the Riemann Roch theorem 
\cite{Jost},\cite{Struwe_Book}
that will
be useful in the sequel. To this purpose, 
given a holomorphic line bundle $L$, we denote by $H^{0}(X,L)$ 
the space of holomorphic sections of $L$ and by $deg L$
the degree of $L$. Then the Riemann Roch theorem
states that
\begin{equation}\label{dimension_holomorphic_sections_of_L}
\dim_{\mathbb{C}}H^{0}(X,L)
-
\dim_{\mathbb{C}}H^{0}(X,K_{X}-L)=deg L + 1 - \mathfrak{g},
\end{equation}
with $K_{X}$ the canonical bundle of $X$.   

It is well known that  $\deg K_{X}=2(\mathfrak{g}-1)$ 
(see \cite{Miranda}),
and therefore by \eqref{dimension_holomorphic_sections_of_L} we deduce: 
$\dim_{\mathbb{C}}H^{0}(X,K_{X})=\mathfrak{g}$.
As a consequence, using again \eqref{dimension_holomorphic_sections_of_L} 
we find in particular that, 
\begin{equation}\label{dimension_holomorphic_sections_for_genus_two}
\; \text{ if } \; \mathfrak{g}\geq 2 \; \text{ then } \; 
\dim_{\mathbb{C}}H^{0}(X,K_{X}-q)
=
\dim_{\mathbb{C}}H^{0}(X,K_{X}) -1,
\end{equation}
for any $q\in X$. 
Hence, $H^{0}(X,K_{X})$ is base-point free.

On the other hand, if we let 
$L=E^{*}\otimes K_{X}=\otimes^{\kappa}K_{X}$, then
$deg L=2\kappa(\mathfrak{g}-1)$, and we may conclude:
\begin{remark}\label{vanishing_property_of_C_kappa_X}
Every non-trivial holomorphic $\kappa$-differential 
$$\alpha \in C_{\kappa}(X)=H^{0}(X,L),\;\alpha \neq 0$$
vanishes exactly at 
$2\kappa(\mathfrak{g}-1)$ points in $X$, counted with multiplicity.
\end{remark}

For given distinct points
$\{ x_{1},\ldots,x_{\nu} \}\subset X$
we consider the subspace:
\begin{equation}\label{Q_kappa_definition}
Q_{\kappa}[x_{1},\ldots,x_{\nu}]
=
\{ \alpha \in C_{\kappa}(X) 
\; : \; 
\alpha(x_{j})=0,\; j=1,\ldots, \nu
 \}
 \subset C_{\kappa}(X). 
\end{equation}
We are going to show that,
\begin{equation}\label{dimension_of_Q_kappa}
\dim_{\mathbb{C}}Q_{\kappa}[x_{1},\ldots, x_{\nu}]
=
(2\kappa-1)(\mathfrak{g}-1) - \nu,
\; \text{ for } \; 
1\leq \nu <
 2(\kappa-1)(\mathfrak{g}-1).
\end{equation}

Indeed, $Q_{\kappa}[x_{1},\ldots,x_{\nu}]$
is given by the space of holomorphic sections of the line bundle 
$L=\otimes ^{\kappa}K_{X} - \{ x_{1} , \ldots,x_{\nu} \} $,
and we know that $\deg(L)=2 \kappa (\mathfrak{g}-1) - \nu$.  
While, if  $1\leq \nu < 2(\kappa-1)(\mathfrak{g}-1)$, then for the line bundle 
$K_{X}-L=L^{-1} \otimes K_{X}$
(with $L^{-1}=L^{*}$) 
we obtain: 
$$\deg(K_{X}- L)=-2(\kappa-1)(\mathfrak{g}-1) +\nu < 0,$$
and so
$
\dim_{\mathbb{C}}H^{0}(X,K_{X}-L)=0.
$ 
Therefore, by the Riemann-Roch theorem, we obtain
$$
\dim_{\mathbb{C}}H^{0}(X,L) - \dim_{\mathbb{C}}H^{0}(X,K_{X}-L)
=
\deg(L) + 1 - \mathfrak{g}
=
(2\kappa-1) (\mathfrak{g}-1) - \nu,
$$
and \eqref{dimension_of_Q_kappa} is established, since
$
\dim_{\mathbb{C}} Q_{\kappa}[x_{1},\ldots,x_{\nu}]
=
\dim_{\mathbb{C}}H^{0}(X,L)
$. 

\begin{remark}\label{rem_vanishing_at_all_but_one}
From \eqref{dimension_of_Q_kappa} follows in particular that we can always 
find $\alpha \in C_{\kappa}(X)$
that vanishes at all \underline{but} one point of
$\{ x_{1},\ldots,x_{\nu} \} $. 
Furthermore
\begin{equation}\label{dimension_Q_kappa_q}
\dim_{\mathbb{C}} Q_{\kappa} [q]=\dim_{\mathbb{C}}C_{\kappa}(X)-1,
\end{equation}
for every $q\in X$, namely also $H^{0}(X,\otimes ^{\kappa}K_{X})$ is base-point free. 
\end{remark}	 

By using \eqref{dimension_Q_kappa_q}, we can define the \underline{Kodaira map}:
\begin{equation}%\label{Kodaira_embedding_map}
\tau:X\longrightarrow \mathbb{P}(\mathcal{H}^{0,1}(X,E))\simeq \mathbb{P}^{(2\kappa-1)(\mathfrak{g}-1)-1}
\end{equation}  
simply by associating to any $q \in X$ the element of
$\mathbb{P}(\mathcal{H}^{0,1}(X,E))$ identified by 
the ray of classes generated by
$[\beta] \in \mathcal{H}^{0,1}(X,E)\setminus \{  0 \} $ with harmonic representative
$\beta_{0}\in [\beta]$ and satisfying:
\begin{equation}
\int_{X} \beta_{0}\wedge \alpha=0, \; \forall \; \alpha \in Q_{2}[q].
\end{equation}

Indeed, by recalling \eqref{class_beta_to_C_kappa_star}, 
in this way we can identify the ray of functionals in  
$(C_{\kappa}(X))^{*}$ which admit $Q_{2}[q]$ as their kernel.
Such a map is holomorphic \cite{Donaldson_Book},\cite{Griffiths_Harris}, and we have:
\begin{lemma}\label{lem_complex_curve}
The image $\tau(X)$ defines a complex curve into the projective space:
$\mathbb{P}(\mathcal{H}^{0,1}(X,E))$ of complex dimension
$(2\kappa-1)(\mathfrak{g}-1) -1 \geq 2$. 
\end{lemma}	
\begin{proof}
In case $(\kappa,\mathfrak{g})\neq (2,2)$ then 
$2(\kappa-1)(\mathfrak{g}-1) >2$ and so we can apply \eqref{dimension_of_Q_kappa} 
with $\nu=2$ 
to show that $\tau$ is \underline{injective} and 
defines an embedding in this case
(see proposition 4.20 in \cite{Miranda}). 
Therefore, $\tau(X)$ is a 
(regular) complex curve with the same complex dimension of $X$, namely one.

Next, let us consider the case:  $\kappa=2$ and $\mathfrak{g}=2$ where
\eqref{dimension_of_Q_kappa} is valid only with $\nu=1$, 
and so the argument above does not apply.
Indeed, for $\mathfrak{g}=2$ the map $\tau$ is no longer injective,
and instead we are going to show that it is "generically" a map two-to-one. 

To this purpose, let
$L=\otimes_{k=1}^{2}K_{X}$,
and for given $q \in X$ let $q^{*}\in X$ be such that:
$\tau(q)=\tau(q^{*})$.  That is 
$Q_{2}[q]=Q_{2}[q^{*}]=Q_{2}[q,q^{*}]$,
and  consequently (by  \eqref{dimension_Q_kappa_q}) we have: 
\begin{equation}\label{complex_dimension_q_q_star}
\begin{split}
\dim_{\mathbb{C}} H^{0}(X,L-q-q^{*})
= \; &
\dim_{\mathbb{C}} (X,L-q) 
= 
\dim_{\mathbb{C}} (X,L-q^{*}) \\
= \; &
3(\mathfrak{g}-1)-1
=
2
\;\;\; (\text{for } \; 
\mathfrak{g}=2). 
\end{split}
\end{equation} 
Hence, by using \eqref{complex_dimension_q_q_star} 
together with the Riemann Roch Theorem \eqref{dimension_holomorphic_sections_of_L}, 
we find: 
$\dim_{\mathbb{C}} H^{0}(X,-K_{X}+q+q^{*})=1$. 
In other words, the holomorphic line bundle:  
$-K_{X}+q+q^{*}$
(of degree zero)
must be trivial. 
So there must exist a holomorphic section of $K_{X}$ which vanishes exactly at 
$q$ and $q_{*}$.
We shall use this information for the (well-defined) holomorphic Kodaira map
$\tau_{1}:X\longrightarrow \mathbb{P}(V_{1}^{*})$
relative to the space $V_{1}=H^{0}(X,K_{X})$, 
see \eqref{dimension_holomorphic_sections_for_genus_two}.
Indeed,  $deg(K_{X})=2$ (for $\mathfrak{g}=2$), 
and so $\tau_{1}$ is "generically" a two-to-one map and
(by the information above) necessarily:
$\tau_{1}(q)=\tau_{1}(q^{*})$. 
Furthermore, for $\mathfrak{g}=2$, the Riemann-Hurwitz formula implies that, 
$q=q_{*}$ only for \underline{six} points
which must coincide exactly with the Weierstrass points of $X$. 
Knowing that $\tau_{1}(X)$ defines a compact, smooth complex submanifold of
$\mathbb{P}(V_{1}^{*})\simeq \mathbb{P}^{1}$
and it  cannot reduce to a point, 
then necessarily: 
$\tau_{1}(X) = \mathbb{P}(V_{1}^{*})\simeq \mathbb{P}^{1}$. 
Namely
$\tau_{1}(X)$ defines a smooth complex curve (a conic) in $\mathbb{P}^{1}$ . 

As a consequence of the above discussion, 
also $\tau$ must be "generically" a two-to-one map,
with the same  preimages as $\tau_{1}$, 
in the sense that:
$$
\; \forall \; q\in X\; : \; 
\{ p\in X \; : \;  \tau(p)=\tau(q) \}
=
\{ p\in X\; : \;  \tau_{1}(p)=\tau_{1}(q) \} .
$$
Therefore, it is well defined the embedding:
$
\mathbb{P}^{1}\simeq \mathbb{P}(V_{1}^{*})\xrightarrow{\;\phi\;} \mathbb{P}(V^{*})\simeq \mathbb{P}^{2}
$ 
such that
$\tau(q)=\phi(\tau_{1}(q))$, for all $q\in X$. And so
 $\tau(X)$ defines a complex curve into $\mathbb{P}^{2}$ as claimed. 
\end{proof}

Finally, we recall that around any given $q\in X$, 
we can introduce holomorphic coordinates $\{ z \} $ centered at the origin,  
so that, for $z=x+iy\in B_{r}$ with $r>0$ small, we have:
\begin{equation}
\begin{split}
& \partial=\frac{\partial}{\partial z}=
\frac{1}{2}(\frac{\partial}{\partial x} - i \frac{\partial}{\partial y})
\; \text{ and } \; 
\bar{\partial}=\frac{\partial}{\partial \bar{z}}
=
\frac{1}{2}(\frac{\partial}{\partial x} - i \frac{\partial}{\partial y}),
%\notag 
\\
&
dz=dx + i dy,\; d \bar{z} = dx-idy,\;
%\notag 
\\
&
g_{X}=e^{2u_{0}}\vert dz \vert^{2},\;
u_{0} \; \text{ smooth and } \; u_{0}(0)=0,
%\notag
\\
& 
dA = \frac{i}{2}e^{2u_{0}} dz \wedge d \bar{z},
\;
*dz=ie^{2u_{0}} d\bar{u},\; *d \bar{z}=-ie^{2u_{0}}dz
.
\end{split}
\end{equation}	
So we can express the Laplace-Beltrami operator:
$\Delta_{g_{X}}=\Delta=4e^{-2u_{0}}\partial \bar{\partial}$ 
and  obtain 
$4 \partial \bar{\partial} u_{0} = e^{2u_{0}}$. 

Actually (with abuse of notation) in the sequel we also denote 
the \underline{flat} Laplacian by  
$\Delta=4 \partial \bar{\partial}$, 
unless confusion arises.

Moreover, in such local coordinates, 
the (fiberwise) Hermitian product operates essentially
as the usual Hermitian product on $\mathbb{C}$, and for the local expression 
of the (fiberwise) norm of sections and forms, we record that in $B_{r}$ there holds:
\begin{equation*}
\begin{split}
\Vert \eta \Vert_{E}
\simeq \; &
\vert \eta(z) \vert e^{2(\kappa-1)u_{0}}
\;\; 
\eta\in A^{0}(E) 
,\;\;  
\Vert \beta_{0} \Vert_{E} \simeq\vert \beta(z) \vert e^{2(\kappa-2)u_{0}}
\;\;
\beta \in A^{0,1}(X,E). 
\end{split}
\end{equation*}
In local holomorphic coordinates, any holomorphic $\kappa$-differential
$\alpha \in C_{\kappa}(X)$, 
takes the expression: 
$\alpha = h (dz)^{\kappa}$ with 
$h$ holomorphic in $B_{r}$. 

In this way, it is clear  what we mean by a zero of $\alpha$ and 
corresponding multiplicity,
since such notions are independent of the chosen holomorphic coordinates. 
In particular, if $q$ is a zero of $\alpha$ with multiplicity $n$, then 
in local coordinates we have:
$\Vert \alpha \Vert_{E^{*}} \simeq \vert z \vert^{n} \vert h(z) \vert e^{2(\kappa-1)u_{0}} $
with the function $h$ holomorphic and never vanishing in $B_{r}$. 
In particular, 
$\partial \bar{\partial} \ln \vert h(z) \vert^{2}=0$ in $B_{r}$,
a property we shall use in the sequel.

\section{Asymptotics}%\label{Asymptotics}
From now on we shall use the (fiberwise) Hermitian product
$\langle \cdot,\cdot \rangle$ and norm $\Vert \cdot \Vert$
of sections and forms valued on $E$ and $E^{*}$
without the subscripts $E$ and $E^{*}$ respectively, 
unless some confusion arises. 

For given $(X,[\beta]) \in \mathcal{T}_{\mathfrak{g}}(S) \times \mathcal{H}^{0,1}(X,E)$
and $t\geq 0$, we consider the Donaldson functional:
\begin{equation*}%\label{6.1}
D_{t}(u,\eta)
=
\int_{X} 
(
\frac{\vert \nabla u \vert^{2}}{4}
-
u
+
te^{u}
+
4\Vert \beta_{0} + \bar{\partial}\eta \Vert^{2}e^{(\kappa-1)u}
)
dA
\end{equation*}
with domain
$$\Lambda
=
\{  
(u,\eta)\in H^{1}(X) \times W^{1,2}(X,E)
\; : \; 
\int_{X} e^{(\kappa-1)u}\Vert \beta_{0} + \bar{\partial}\eta \Vert^{2}dA
<
+\infty
\},
$$
where 
$
H^{1}(X)
=
\{ u\in L^{2}(X) \; : \; \vert \nabla u \vert \in L^{2}(X)\} 
$ 
is the Sobolev space with usual scalar product and norm:
$
\Vert u \Vert_{H^{1}}
=
(\int_{X}(u^{2}+\vert \nabla u \vert^{2})dA)^{\frac{1}{2}}
$
and $W^{1,2}(X,E)$ is the Sobolev space defined in \eqref{W_1_p}. 

It is readily verified that, for $t>0$, 
the functional $D_{t}$ is bounded from below in $\Lambda$, while
this is not always the case for $t=0$, as we shall discuss below. 

\

We notice also that the functional
$D_{t}$ admits Gateaux derivatives at a point $(u,\eta) \in \Lambda$ only along smooth directions 
$(v,l) \in C^{1}(X)\times A^{0}(E)$.
In fact, the troublesome term with respect to differentiability is 
given by:
\begin{equation*}%\label{6.1a}
T(u,\eta)
=
\int_{X} e^{(\kappa-1)u}
\Vert \beta_{0} + \bar{\partial}\eta \Vert
dA
,
\; (u,\eta)\in \Lambda.
\end{equation*}
On the other hand, if $(u,\eta) \in \Lambda$ is a \underline{weak} 
critical point of $D_{t}$ in the sense that:
\begin{equation}\label{weak_criticality}
D_{t}^{\prime}(u,\eta)[v,l]=0,
\; \forall \; (v,l) \in C^{1}(X) \times A^{0}(E)
\end{equation}
then,
\begin{equation}\label{weak_criticality_v}
\begin{split}
\int_{X} \nabla u \nabla v
+
2v
(
-1 + te^{u} +4 (\kappa-1)e^{(\kappa-1)u}\Vert \beta_{0}+ \bar{\partial}\eta \Vert^{2}
)
\, dA
=
0,
\; \forall \; v\in C^{1}(X)
\end{split}
\end{equation}
\begin{equation}\label{weak_criticality_l}
\begin{split}
\int_{X} e^{(\kappa-1)u}\langle \beta_{0} + \bar{\partial}\eta, \bar{\partial}l \rangle
\, dA
=0
,
\; \forall \; l\in A^{0}(E).
\end{split}
\end{equation}
The elliptic operators involved (in a weak form) in
\eqref{weak_criticality_v} and \eqref{weak_criticality_l}  allow us to 
gain all the regularity we need about 
$(u,\eta) \in \Lambda$, starting with the following result established in 
\cite{Huang_Lucia_Tarantello_2}. 

\begin{lemma}[\cite{Huang_Lucia_Tarantello_2}]\label{Lemma_regularity_l}
Let $(u,\eta) \in \Lambda $.
\begin{enumerate}[label=(\roman*)]
\item 	If \eqref{weak_criticality_l} holds then 
$\eta \in W^{1,p}(X,E)$ for all $p\geq 2$, 
and we can allow  
$l\in W^{1,2}(X,E)$ in \eqref{weak_criticality_l}. 
\item If \eqref{weak_criticality_v} and \eqref{weak_criticality_l} hold (i.e.
$(u,\eta)$ is a "weak" critical point of $D_{0}$) then $(u,\eta)$ is smooth and
satisfies, in the classical sense, the following system of equations:
\begin{equation}\label{system_P_t}
(\mathcal{P})_{t} 
\quad
\left\{
\begin{matrix*}[l]
\Delta u +2 -2te^{u} -8(\kappa-1)e^{(\kappa-1)u}\Vert \beta_{0}+\bar{\partial}\eta \Vert^{2} =0  & \text{in}  &  X  \\
\bar{\partial}(e^{(\kappa-1)u}*_{E}(\beta_{0}+\bar{\partial}\eta))=0.  &  \;\text{}\; 
\end{matrix*}
\right. 
\end{equation} 
\end{enumerate}
\end{lemma}
\begin{proof}
By using Weil's theorem, from \eqref{weak_criticality_l} we get that 
$e^{(\kappa-1)u}*_{E}(\beta_{0} + \bar{\partial}\eta)\in C_{\kappa}(X)$.
So we can use the basis in \eqref{basis_for_C_k_X} to write
$$
e^{(\kappa-1)u}(\beta_{0}+\bar{\partial}\eta)
=
\sum_{j=1}^{\nu}a_{j}s_{j}
\; \text{ for suitable } \; 
a_{j}\in \mathbb{C}, 
\; \text{ and } \; 
\nu = (2\kappa-1)(\mathfrak{g}-1).
$$
As a consequence, since 
$e^{-(\kappa-1)u}\in L^{p}(X),\; \forall \; p>1$, we find
$$
\bar{\partial}\eta
=
e^{-(\kappa-1)u}(\sum_{j=1}^{\nu}a_{j}*_{E}^{-1}s_{j}) - \beta_{0}
\in L^{q}(X),
\; \forall \; q>1.
$$
Therefore, we can use elliptic regularity and the Poincar\'e inequality
for the elliptic operator $\bar{\partial}$, 
to conclude that 
$\eta \in W^{1,p}(X,E)$ for all $p>1$, as claimed in (i).
At this point (ii) follows easily, since by using part (i) and
elliptic regularity theory together with well known boot-strap arguments, 
we obtain that $(u,\eta)$ is smooth and satisfies  \eqref{system_P_t}. 
\end{proof}
Therefore, to find solutions to \eqref{system_P_t}, we may more conveniently consider 
$D_{t}$ in the stronger space: 
\begin{equation*}%\label{V_p}
\mathcal{V}_{p}=H^{1}(X) \times W^{1,p}(X,E)
,\;
p>2,
\end{equation*}
as  we check easily that, 
$T \in C^{1}(\mathcal{V}_{p})$ 
(see \cite{Huang_Lucia_Tarantello_2}) and so 
 $D_{t} \in C^{1}(\mathcal{V}_{p})$. Summarizing:

$(u,\eta)$ is a (classical) solution of problem 
$(\mathcal{P})_{t}$
if and only if it is a weak critical point of $D_{t}$ in $\Lambda$ 
(in the sense of \eqref{weak_criticality}) 
or equivalently, it is a (usual) critical point 
of $D_{t}$ in $\mathcal{V}_{p}$.

As already mentioned in the introduction, 
in \cite{Huang_Lucia_Tarantello_2} it has been proved the following
uniqueness result which we recall for the convenience:

\begin{thm}[\cite{Huang_Lucia_Tarantello_2}]\label{thm_uniqueness}
For any $t>0$, the functional $D_{t}$ admits a \underline{unique} critical point
$(u_{t},\eta_{t}) \in \mathcal{V}_{p}$ for every $p>2$, 
and it corresponds to the global minimum
of $D_{t}$ in $\Lambda$. 
In particular, problem $(\mathcal{P})_{t}$ admits the unique solution
$(u_{t},\eta_{t})$.
\end{thm} 
 
As already mentioned, for $t=0$, the functional:
\begin{equation*}%\label{definition_D_0}
D_{0}(u,\eta)
=
D_{t=0}(u,\eta)
= 
\int_{X}
\frac{\vert \nabla u \vert^{2} }{4}-u+4e^{(\kappa-1)u}
\Vert \beta_{0}+ \bar{\partial} \eta \Vert^{2}dA
\end{equation*}
may no longer be bounded from below in $\Lambda$ and actually
problem $(\mathcal{P})_{t=0}$ in \eqref{system_P_t} may not admit 
a solution. For example, if 
$[\beta]=0$ (i.e. $\beta_{0}=0$), 
then we easily check that $(\mathcal{P})_{t=0}$ admits no solution and actually:
\begin{equation*}%\label{explicit_example}
u_{t}= \ln \frac{1}{t},
\eta_{t}=0
\; \text{ and } \;
D_{0}(u_{t},\eta_{t})\leq D_{t}(u_{t},\eta_{t}) = \ln t \longrightarrow -\infty,
\end{equation*}
as $t\longrightarrow 0^{+} $,  and so $D_{0}$ is unbounded from below in $\Lambda$. 
Such an example illustrates the only possible obstruction to the solvability 
of $(\mathcal{P})_{t=0}$, in the following sense. 

\begin{thm}\label{thm_primo}
If $(u_{0},\eta_{0})$ is a solution for $(\mathcal{P})_{t=0}$ in \eqref{system_P_t}, then
\begin{enumerate}[label=(\roman*)]
\item $(u_{t},\eta_{t})\longrightarrow (u_{0},\eta_{0})$ in $\mathcal{V}_{p}$, $p>2$,
as $t\longrightarrow 0^{+}$;
\item 
$(u_{0},\eta_{0})$ is the \underline{only} solution of $(\mathcal{P})_{t=0}$ 
and so the only critical point of $D_{0}$. Furthermore, 
$D_{0}$ is bounded from below in $\Lambda$ and attains its global minimum at 
$(u_{0},\eta_{0})$.
\end{enumerate}
\end{thm}
To establish Theorem \ref{thm_primo} we need few preliminaries. 
To start, for a fixed
$u \in H^{1}(X)$, we are going to account for \eqref{weak_criticality_l} by considering the minimization problem:
\begin{equation}\label{minimization_in_eta}
\inf_{\eta \in W^{1,2}(X,E)}\int e^{(\kappa-1)u}\Vert \beta_{0} + \bar{\partial}\eta \Vert^{2}\,dA
:=
c_{0}(u).
\end{equation}
We have:
\begin{lemma}\label{prop_6.1}
For every $u\in H^{1}(X)$ there exits a \underline{unique} global minimum $\eta(u)$
for \eqref{minimization_in_eta} with $\eta(u) \in W^{1,p}(X,E), \; \forall \; p > 2$. 
Furthermore,
\begin{equation*}
\eta \in W^{1,2}(X,E) 
\; \text{ satisfies \eqref{weak_criticality_l} if and only if} \;
\eta=\eta(u). 
\end{equation*}
\end{lemma}	
\begin{proof}
We will be sketchy, as this result was essentially pointed out in \cite{Huang_Lucia_Tarantello_2}.  
Let $\eta_{n} \in W^{1,2}(X,E)$ be a minimizing sequence for \eqref{minimization_in_eta}, that is 
$$
T(u,\eta_{n})\longrightarrow c_{0}(u), 
\; \text{ as } \; 
n\longrightarrow +\infty. 
$$
Since $e^{-u} \in L^{q}(X), \; \forall \; q>1$, we can use H\"older inequality to check that $\eta_{n}$ is uniformly bounded in $W^{1,a}(X,E)$ for $a\in (0,1)$. Therefore,
along a subsequence, we find:
$\eta_{n} \rightharpoonup \eta$ weakly in $W^{1,a}(X,E)$ and therefore,
$$
\int_{X}e^{(\kappa-1)u}\langle \beta_{0}+ \bar{\partial}\eta,\bar{\partial}l \rangle
\, dA =0
,
\; \forall \; l \in A^{0}(E). 
$$
Then, by Lemma \ref{Lemma_regularity_l}, we know that $\eta \in W^{1,p}(X,E)$,
$\; \forall \; p>2$. Hence by the weak convergence we obtain that,
for $\xi \in L^{b}(A^{0,1}(X,E))$ with 
$b=\frac{a}{a-1}$ (the dual exponent of $a$), 
there holds:
$$
\int_{X} \langle  \bar{\partial}(\eta_{n}-\eta),\xi\rangle\, dA
\longrightarrow 0,
\; \text{ as  } \; 
n\longrightarrow +\infty.
$$
In particular, by taking $\xi=\eta e^{(\kappa-1)u}$ we can derive:
$$
c_{0}(u)
=
\lim_{n\to +\infty}
\int_{X}e^{(\kappa-1)u}\Vert \beta_{0} + \bar{\partial}\eta_{n} \Vert^{2}\,dA
\geq 
\int_{X}e^{(\kappa-1)u}\Vert \beta_{0} + \bar{\partial}\eta \Vert^{2}\,dA
\geq 
c_{0}(u),
$$
so $\eta$ is a minimum for \eqref{minimization_in_eta} and $\eta \in W^{1,p}(X,E)$,
$\; \forall \;  p>2$. 
Since for fixed $u\in H^{1}(X)$ the operator $T(u,\cdot)$ is strictly convex, 
then $\eta=\eta(u)$ is the only minimum point for \eqref{minimization_in_eta}.
In addition, for any $\eta$ satisfying \eqref{weak_criticality_l}, we have:
$\eta \in W^{1,p}(X,E)$, $\; \forall \; p>2$, and 
\begin{equation*}
\begin{split}
\int_{X}e^{(\kappa-1)u}   &   \Vert \beta_{0} + \bar{\partial}\eta \Vert^{2}\, dA
\geq 
\int_{X}e^{(\kappa-1)u}\Vert \beta_{0} + \bar{\partial}\eta(u) \Vert^{2}\, dA \\
= \; &
\int_{X}e^{(\kappa-1)u}\Vert \beta_{0} + \bar{\partial}\eta \Vert^{2}\, dA
+
\int_{X}e^{(\kappa-1)u}\Vert \bar{\partial}\eta(u) - \bar{\partial}\eta \Vert^{2}\, dA,
\end{split}
\end{equation*}
and necessarily $\eta=\eta(u)$, as claimed. 
\end{proof}
We wish to point out some useful properties about the map: 
\begin{equation}\label{map_H_1_to_W_1_p}
\begin{split}
H^{1}(X)\longrightarrow W^{1,p}(X,E)
\; : \; 
u \longrightarrow \eta(u).
\end{split}
\end{equation} 
To this purpose, for $u\in H^{1}(X)$, we let 
\begin{equation*}%\label{beta_u}
\beta(u)=e^{(\kappa-1)u}(\beta_{0}+\bar{\partial}\eta(u))\in W^{1,p}(X,E),
\; p>2. 
\end{equation*}
Since $*_{E}\beta(u) \in C_{\kappa}(X)$, using the frame 
$\{ s_{1},\ldots,s_{\nu} \}$ of $C_{\kappa}(X)$ with 
$\nu=(2\kappa-1)(\mathfrak{g}-1)$,
as given in \eqref{basis_for_C_k_X},
we may write:
$*_{E}\beta(u)=\sum_{j=1}^{\nu}a_{j}(u)s_{j}$ with suitable
$a_{j}(u) \in \mathbb{C}$. Consequently, 
\begin{equation}\label{beta_u_sum} 
\beta(u)=\sum_{j=1}^{\nu}a_{j}(u)*_{E}^{-1}s_{j} 
,\;
a_{j}(u)
=
\int_{X}\langle \beta(u),*_{E}^{-1}s_{j} \rangle\, dA 
,\;
j=1,\ldots,\nu.
\end{equation} 
For $u$ and $u_{0}\in H^{1}(X)$ we point the following simple (but useful) identities:
\begin{equation}\label{d_bar_eta_difference}
\bar{\partial}\eta(u) - \bar{\partial}\eta(u_{0})
=
e^{(\kappa-1)u}(\beta(u)-\beta(u_{0}))
+
(e^{(\kappa-1)(u-u_{0})}-1)\beta(u_{0})
\end{equation}
or equivalently: 
\begin{equation}\label{beta_difference}
\beta(u)-\beta(u_{0})
=
e^{(\kappa-1)u}
(\bar{\partial}\eta(u)- \bar{\partial}\eta(u_{0}))
+
(e^{(\kappa-1)u}-e^{(\kappa-1)u_{0}})\beta(u_{0}).
\end{equation}

\begin{lemma}\label{eta_u_map}
\begin{enumerate}[label=(\roman*)]
\item If $u_{n}\rightharpoonup u$ weakly in $H^{1}(X)$ then 
$\eta(u_{n})\longrightarrow \eta(u)$ strongly in $W^{1,p}(X,E)$, $p>2$. 
In particular, the map in \eqref{map_H_1_to_W_1_p} takes bounded sets of $H^{1}(X)$
into bounded sets of $W^{1,p}(X,E)$, $p>2$.

\item 
For given $u_{0} \in H^{1}(X)$ and $p>2$ there exists a positive constant 
$\sigma_{p}= \sigma_{p}(u_{0})$
(depending only on $p$ and $u_{0}$) such that
\begin{equation}\label{d_bar_eta_difference_in_Lemma}
\Vert \bar{\partial} \eta(u) - \bar{\partial}\eta(u_{0})\Vert_{L^{p}}^{2}
\leq \sigma_{p}
(
\Vert \bar{\partial} \eta(u) - \bar{\partial}\eta(u_{0})\Vert_{L^{2}}^{2}
+
\Vert u-u_{0} \Vert_{H^{1}}^{2}
).
\end{equation}
\end{enumerate}
\begin{proof}
To establish $(i)$, we observe that $u_{n}$ is uniformly bounded in 
$H^{1}(X)$, and in particular, $e^{\pm u_{n}}\longrightarrow e^{\pm u}$ in 
$L^{q}(X)$,  $\; \forall \; q>1$. 
As a consequence, by setting $\eta_{n}=\eta(u_{n})$, from \eqref{minimization_in_eta} we have:
\begin{equation*}
\begin{split}
\int_{X} e^{(\kappa-1)u_{n}}\Vert \beta_{0} + \bar{\partial}\eta_{n}\Vert^{2}\, dA
\leq 
\int_{X}e^{(\kappa-1)u_{n}}\Vert \beta_{0} \Vert^{2}
\leq C. 
\end{split}
\end{equation*}
So, for $1<a<2$, we can use H\"older inequality to see that 
$\eta_{n}$ is uniformly bounded in $W^{1,a}(X,E)$. Thus,  
along a subsequence, we have:
$\eta_{n} \rightharpoonup \eta$ weakly in $W^{1,a}(X,E)$, as $n\longrightarrow +\infty$. 
Consequently, for any 
$\xi \in L^{b}(A^{0,1}(X,E)),\, b=\frac{a}{a-1}$, we have:
$$
\int_{X}\langle \bar{\partial} \eta_{n} - \bar{\partial}\eta ,  \xi \rangle_{E} \, dA
\longrightarrow 0
,
\; \text{ as } \; n\longrightarrow +\infty. 
$$
Furthermore, if we take 
$\xi \in L^{q}(A^{0,1}(X,E))$ with $q>b$ and $p=\frac{bq}{q-b}$ we can estimate
\begin{equation*}
\begin{split}
\vert\int_{X} (e^{(\kappa-1)u_{n}}- e^{(\kappa-1)u})
\langle \bar{\partial}\eta_{n}, \xi \rangle\, dA 
\vert 
\leq 
C\Vert e^{(\kappa-1)u_{n}} - e^{(\kappa-1)u}\Vert_{L^{p}}
\longrightarrow 0
,
\end{split}
\end{equation*}
as $n\longrightarrow +\infty$. Hence, as $n\longrightarrow +\infty$, 
$$
\int_{X}e^{(\kappa-1)u_{n}}
\langle \beta_{0} + \bar{\partial}\eta_{n} , \xi \rangle\,dA
\longrightarrow 
\int_{X}e^{(\kappa-1)u}
\langle \beta_{0} + \bar{\partial}\eta , \xi \rangle\,dA,
$$
for any $\xi \in L^{q}(A^{0,1}(X,E))$ with 
$q>b=\frac{a}{a-1}$. Consequently, the property
$$
\int_{X}e^{(\kappa-1)u_{n}} 
\langle \beta_{0} + \bar{\partial}\eta_{n}, \bar{\partial}l \rangle\, dA =0,
\; \forall \;  l \in A^{0}(E)
$$ 
passes to the limit, as $n\longrightarrow +\infty$, and we conclude that
$\eta$ satisfies \eqref{weak_criticality_l}.

Hence, by Lemma \ref{prop_6.1} we conclude that, 
$\eta=\eta(u)\in W^{1,p}(X,E)$, $\; \forall \; p>2$.
Moreover, as $n\longrightarrow +\infty$, 
\begin{equation*}
\begin{split}
a_{j,n} 
: = \; &
a_{j}(u_{n})
=
\int_{X}e^{(\kappa-1)u_{n}}
\langle \beta_{0} + \bar{\partial} \eta_{n},*_{E}^{-1}s_{j}\rangle \,dA \\
\longrightarrow &
\int_{X}e^{(\kappa-1)u}
\langle \beta_{0} + \bar{\partial} \eta,*_{E}^{-1}s_{j}\rangle \,dA 
=
a_{j}(u),
\; \forall \; j=1,\ldots,\nu 
\end{split}
\end{equation*}
and therefore,
$$
\beta(u_{n})=\sum_{j=1}^{\nu}a_{j,n}*_{E}^{-1}s_{j}
\longrightarrow 
\beta(u)=\sum_{j=1}^{\nu}a_{j}(u)*_{E}^{-1}s_{j}
\; \text{ in } \; L^{q}(A^{0,1}(X,E)), \; \forall \; q>1.  
$$
At this point we can use \eqref{d_bar_eta_difference} to conclude that, 
\begin{equation*}
\begin{split}
\Vert \bar{\partial}\eta(u_{n}) - \bar{\partial} \eta(u) \Vert_{L^{p}}
\leq \; &
\Vert e^{-(\kappa-1)u_{n}} \Vert_{L^{2p}}
\Vert \beta(u_{n})-\beta(u) \Vert_{L^{2p}} \\
&  +
\Vert e^{(\kappa-1)(u_{n}-u)}-1 \Vert_{L^{2p}}
\Vert \beta(u) \Vert_{L^{2p}} \\
\leq \; &
C
(
\Vert \beta(u_{n})-\beta(u) \Vert_{L^{2p}}  
+
\Vert e^{(\kappa-1)(u_{n}-u)}-1 \Vert_{L^{2p}}
)
\longrightarrow 0,
\end{split}
\end{equation*}
as $n\longrightarrow +\infty$. Since any other convergent subsequence
of $\eta(u_{n})$ admits the same limit $\eta(u)$, 
we conclude that the full sequence $\eta(u_{n})\longrightarrow \eta(u)$
in $W^{1,p}(X,E)$, $p>2$, as claimed. 

\

To establish \eqref{d_bar_eta_difference_in_Lemma}, we can assume without loss of generality that
\begin{equation}\label{normalization_u_u_0}
\Vert \bar{\partial} \eta(u) - \bar{\partial}\eta(u_{0}) \Vert_{L^{2}}^{2}
+
\Vert u-u_{0} \Vert_{H^{1}}^{2}=1.
\end{equation}
In particular $\Vert u \Vert_{H^{1}} \leq 1 +\Vert u_{0} \Vert_{H^{1}}$, and so 
$
\Vert e^{(k-1)u}-e^{(k-1)u_{0}} \Vert_{L^{q}}
\leq C_{q}\Vert u-u_{0} \Vert_{H^{1}} 
$
and $\Vert e^{\pm u} \Vert_{L^{q}} \leq C_{q}$, $\; \forall \; q>1$ and
with suitable $C_{q}>0$ depending only on $q$ and $u_{0}$. 

Firstly, by \eqref{beta_u_sum}, we see that 
$
\vert a_{j}(u)-a_{j}(u_{0}) \vert\leq C\Vert \beta (u) - \beta(u_{0})\Vert_{L^{1}}
$, 
for every $j=1,\ldots,\nu$. Thus, 
\begin{equation*}
\begin{split}
\Vert \beta(u)-\beta(u_{0}) \Vert_{L^{p}}
\leq \; &
C
(\sum_{j=1}^{\nu}\vert a_{j}(u)-a_{j}(u_{0}) \vert^{p})^{\frac{1}{p}} \\
\leq \; &
C_{p}\sum_{j=1}^{\nu}\vert a_{j}(u)-a_{j}(u_{0}) \vert 
\leq C_{p}
\Vert \beta(u)-\beta(u_{0}) \Vert_{L^{1}}. 
\end{split}
\end{equation*}
By combining \eqref{d_bar_eta_difference} and \eqref{beta_difference}, 
for $p>2$ we derive:
\begin{equation*}
\begin{split}
\Vert \bar{\partial}\eta(u) -\bar{\partial} \eta(u_{0}) \Vert
\leq  \; &
\Vert e^{(\kappa-1)u} \Vert_{L^{2p}} \Vert  \beta(u) - \beta(u_{0})\Vert_{L^{2p}} \\
& \; +
\Vert e^{(\kappa-1)(u-u_{0})}-1 \Vert_{L^{2p}} \Vert \beta(u_{0})\Vert_{L^{2p}} \\
\leq \; &
C_{p} 
(
\Vert \beta(u)-\beta(u_{0}) \Vert_{L^{1}}+ \Vert u-u_{0} \Vert_{H^{1}}
) \\
\leq \; &
C_{p}
(
\Vert e^{(\kappa-1)u} \Vert_{L^{2}}
\Vert \bar{\partial} \eta(u)-\bar{\partial} \eta(u_{0})\Vert_{L^{2}} \\
& \; +
\Vert e^{( \kappa-1)u}-e^{(\kappa-1)u_{0}} \Vert_{L^{2}}
\Vert \beta(u_{0}) \Vert_{L^{2}}
+
\Vert u-u_{0} \Vert_{H^{1}}
) \\
\leq \; &
C_{p}
(
\Vert \bar{\partial}\eta(u)-\bar{\partial}\eta(u_{0}) \Vert_{L^{2}}
+
\Vert u-u_{0} \Vert_{H^{1}}
)
\leq 
\sigma_{p}
\end{split}
\end{equation*}
with a suitable constant $\sigma_{p}>0$ (depending only on $p$ and $u_{0}$) 
obtained in view of   \eqref{normalization_u_u_0}, and the proof is completed. 
\end{proof}
\end{lemma}

By Lemma \ref{eta_u_map}, for $t>0$, we readily get a minimum of 
$D_{t}$ in $\Lambda$ simply by taking (without loss of generality) a minimizing sequence
of the form: $(u_{n},\eta(u_{n}))\in H^{1}(X)\times W^{1,p}(X,E)$, $ p>2$. Indeed, 
for $t>0$, we can take advantage of the estimate:
\begin{equation}\label{estimate_D_t}
\begin{split}
D_{t}(u,\eta) \geq 
\int_{X}(\frac{\vert \nabla u \vert^{2}}{4}
+ 
4e^{(\kappa-1)u}\Vert \beta_{0} + \bar{\partial} \eta \Vert^{2})\,dA
+
4\pi(\mathfrak{g}-1)(\ln t +1 ), 
\end{split}
\end{equation}
which holds for every $(u,\eta) \in \Lambda$, to show that 
$(u_{n})$ is uniformly bounded in $H^{1}(X)$. Then we obtain convergence
(along a subsequence)
to the desired minimum from part (i) of Lemma \ref{eta_u_map}, see 
\cite{Huang_Lucia_Tarantello_2}. 

\

More in general, we can use analogous arguments 
to extend the convergence property in (i) of Lemma \ref{eta_u_map}  and 
obtain the following \underline{weaker} form of the Palais-Smale (PS)-condition, 
valid for $D_{t}$, when $t>0$. 
\begin{lemma}[\cite{Huang_Lucia_Tarantello_2}]\label{Palais_Smale}
Let $t>0$ and assume that $(u_{n},\eta_{n}) \in \mathcal{V}_{p}$, $p>2$ satisfies:
\begin{equation}\label{weak_palais_smale}
\Vert \eta_{n} \Vert_{W^{1,p}} \leq C,\; 
D_{t}(u_{n},\eta_{n})\longrightarrow c
\; \text{ and } \;
\Vert D^{\prime}_{t}(u_{n},\eta_{n}) \Vert_{\mathcal{V}_{p}^{*}}\longrightarrow 0, 
\end{equation}
as $n\longrightarrow +\infty$. Then there exist $(u,\eta) \in \mathcal{V}_{p}$ such that
(along a subsequence):
\begin{enumerate}[label=(\roman*)]
\item 	$u_{n}\longrightarrow u$ in $H^{1}(X)$, $\eta_{n}\longrightarrow \eta$ in $W^{1,2}(X,E)$,
as $n\longrightarrow \infty$,
\item 
$D_{t}(u,\eta)=c$ and 
$D^{\prime}_{t}(u,\eta)=0$.
\end{enumerate}
Namely $(u,\eta)$ is a critical point for $D_{t}$
with corresponding critical value $c$. 
\end{lemma}
\begin{proof}
See lemma 5.1 of \cite{Huang_Lucia_Tarantello_2}. 
\end{proof}

Such a "compactness" property \underline{cannot} be extended for $t=0$.
Indeed, $D_{0}$ no longer enjoys any sort of "coercivity" property with respect to 
the variable $u$, as instead ensured by \eqref{estimate_D_t} for $t>0$.  
This is also the reason for the possible
unboundedness of $D_{0}$ in $\Lambda$. 

So, while we cannot guarantee that 
$D_{0}$ admits a (weak) critical point 
(namely that $(\mathcal{P})_{t=0}$ in \eqref{system_P_t} admits a solution),
we see that, when a critical point of $D_{0}$ does exist, then it shares
the exact same properties of $(u_{t},\eta_{t})$, 
the (only) critical point of $D_{t}$ for $t>0$. 

For example, the smoothness of $(u_{t},\eta_{t})$ allows us to compute the Hessian 
$D^{\prime \prime}_{t}$ of $D_{t}$ at $(u_{t},\eta_{t})$ as follows 
(see \cite{Huang_Lucia_Tarantello_2}):
\begin{equation}\label{seconda_variatione_D_t}
D^{\prime \prime }_{t}[v,l]
=
t \int_{X} e^{u_{t}}v^{2}dA
+
\mathcal{A}_{t}[v,l]+B_{t}[v,l],
\; \text{ for every } \;
(v,l) \in \mathcal{V}_{p}
\end{equation}
with 
\begin{align}
&
A_{t}[v,l]
=
4 \int_{X} 
\Vert (\kappa-1)v \beta_{t} +  \bar{\partial}l\Vert^{2}e^{(\kappa-1)u_{t}}dA
\geq 0
, \label{6.3} \\
&
B_{t}[v,l]
= 
2 \int_{X} 
(
\vert \bar{\partial} v \vert^{2}
-
4(\kappa-1)Re
\langle \beta_{t},\bar{\partial}v
\otimes l \rangle
e^{(\kappa-1)u_{t}}
) dA \notag
\\
& \quad \quad \quad \quad  + 
4 \int_{X}\Vert \bar{\partial} l \Vert^{2}e^{(\kappa-1)u_{t}} 
dA
\geq \frac{\kappa-1}{2} \int_{X}e^{(\kappa-1)u_{t}}\Vert l \Vert^{2}\,dA, 
\label{6.4a}
\end{align}
where the last estimate in \eqref{6.4a} follows by completing the square
and applying Corollary \ref{cor_d_bar_eta_estimate} to $(u_{t},\eta_{t})$,
see \cite{Huang_Lucia_Tarantello_2}. 
Also we shall give the details about the estimate \eqref{6.4a}  for the case 
$t=0$ in \eqref{B_t_estimate}. 

Clearly, if we assume that $D_{0}$ admits a critical point $(u_{0},\eta_{0})$, 
then we can take $t=0$ in the expressions 
\eqref{seconda_variatione_D_t},\eqref{6.3},\eqref{6.4a}, and for
$\beta_{t=0}=\beta_{0} + \bar{\partial}\eta_{0}$ 
and  
$u_{t=0}=u_{0}$,
we obtain the Hessian $D_{0}^{\prime \prime}$ at $(u_{0},\eta_{0})$. 

Also the last inequality in \eqref{6.4a} carries over to the case $t=0$,
since again we can use, for the solution $(u_{0},\eta_{0})$ the estimate
\eqref{estimate_of_cor_on_d_bar_eta} of Corollary \ref{cor_d_bar_eta_estimate} 
and obtain that, 
\begin{equation}\label{B_t_estimate}
\begin{split}
B_{t=0}[v,l]
\geq \; &
4 \int_{X}
(
\Vert \bar{\partial} l \Vert^{2}
e^{(\kappa-1)u_{0}}
-
2(\kappa-1)^{2}\Vert \beta_{t=0} \Vert^{2}
e^{2(\kappa-1)u_{0}}\Vert l \Vert^{2}
) \,
dA \\
\geq  \; &
\frac{\kappa-1}{2} \int_{X}\Vert l \Vert^{2}e^{(\kappa-1)u_{0}}dA.
\end{split}
\end{equation}
By virtue of \eqref{B_t_estimate}, we will deduce that any 
critical point of $D_{0}$ is a strict local minimum in $\mathcal{V}_{p}$ for $p>2$,
a property already established in
\cite{Huang_Lucia_Tarantello_2} for $D_{t}$, $t>0$ .  

\begin{lemma}\label{lem_strict_min}
Assume that the functional $D_{0}$ admits a critical point
$(u_{0},\eta_{0})$. Then
$\; \exists \; \gamma_{0}>0,\; \delta_{0}>0$ such that
\begin{equation}\label{formula_of_the_lemma_on_D_0}
\begin{split}
D_{0}(u,\eta)
\geq &
D_{0}(u_{0},\eta_{0})
+
\int_{X}e^{(\kappa-1)u}\Vert \bar{\partial}\eta - \bar{\partial}\eta(u) \Vert^{2}\,dA
\\
& \; +
\gamma_{0}
(
\Vert u-u_{0} \Vert_{H^{1}}^{2}
+
\Vert \eta(u)- \eta_{0} \Vert^{2}_{W^{1,p}}
)
\end{split}
\end{equation}
for all $(u,\eta) \in \mathcal{V}_{p}
\; : \; 
\Vert u-u_{0} \Vert_{H^{1}} < \delta_{0}.$
In particular, $(u_{0},\eta_{0})$ is a strict local minimum for 
$D_{0}$ in $\mathcal{V}_{p}$, $p>2$. 
\end{lemma}
\begin{proof}
Firstly, we observe that necessarily: $\beta_{0}\neq 0$, 
since for $[\beta]=[\beta_{0}]=0$, 
problem $(\mathcal{P})_{t=0}$ admits no solution, and hence
$D_{0}$ cannot admit a critical point.  

\

\underline{\textbf{Claim:}}\; 
There exists $\tau_{0}>0$ such that
\begin{equation}\label{6.6}
\min_
{
\Vert v \Vert_{H^{1}}
+
\Vert l\Vert_{W^{1,2}}
=
1
}
D_{0}^{\prime \prime }[v,l]
\geq 
\tau_{0}.
\end{equation}
To establish \eqref{6.6}, we argue by contradiction and suppose that there exists 
$(v_{n},l_{n})\in H^{1}(X) \times W^{1,2}(X,E)$ such that:
\begin{equation}\label{6.8}
\Vert v_{n} \Vert_{H^{1}}
+
\Vert l_{n} \Vert_{W^{1,2}}=1
\; \text{ and } \; 
D_{0}^{\prime \prime }[v_{n},l_{n}]
\longrightarrow 0
,
\; \text{ as } \; 
n\longrightarrow +\infty.
\end{equation} 
As a consequence
\begin{equation*}%\label{convergence_A_n_B_n_to_zero}
\begin{split}
A_{n}=A_{t=0}(v_{n},l_{n})\longrightarrow 0,\;
B_{n}=B_{t=0}(v_{n},l_{n})\longrightarrow 0,\;
\; \text{ as $n\longrightarrow +\infty$. } \; 
\end{split}
\end{equation*}
By using \eqref{B_t_estimate}, we derive:
$\int_{X} e^{(\kappa-1)u_{0}}\Vert l_{n} \Vert^{2}
\longrightarrow 
0
$, 
and thus (as $u_{0}$ is smooth in $X$)  
$\Vert l_{n} \Vert_{L^{2}}\longrightarrow 0$, as
$n\longrightarrow +\infty.$ 

As a consequence we have:
$
\int_{X}  
\Vert \beta_{0} \Vert^{2}
e^{2(\kappa-1)u_{0}} 
\Vert l_{n} \Vert^{2}
dA 
\longrightarrow 0
$, and we can use such information in 
\eqref{B_t_estimate} to deduce that, 
$
\int_{X} \Vert \bar{\partial} l_{n} \Vert^{2}dA
\longrightarrow 0
$,
as
$n\longrightarrow +\infty$.
In conclusion we have shown that, 
%\label{l_n_to_zero}
$
\Vert l_{n} \Vert_{W^{1,2}}
\longrightarrow 0
$
as
$,
n\longrightarrow +\infty.
$
So, from the explicit expression of $B_{n}$,
we find also that, 
$\int_{X} \vert \nabla v_{n} \vert^{2}dA
=
4 \int_{X} \vert \bar{\partial} v_{n} \vert^{2}dA	
\longrightarrow 
0$,  
as
$n\longrightarrow +\infty.$
Finally, we decompose:
$v_{n} = w_{n} + c_{n}$
with
$\int_{X} w_{n} dA=0$
and $c_{n}=\fint_{X}v_{n}$.  
We know that:
$\Vert w_{n} \Vert_{L^{2}}\longrightarrow 0$, 
and by means of \eqref{6.3} (with $t=0$) we deduce that, 
$\int_{X} \Vert \beta_{t=0} \Vert^{2}v_{n}^{2}\longrightarrow 0
$, 
and so 
$
c_{n}^{2}
\int_{X} \Vert \beta_{t=0} \Vert^{2}dA
\longrightarrow 0
$,
as
$n\longrightarrow +\infty$. 
But
$
\int_{X} \Vert \beta_{t=0} \Vert^{2}
=
\int_{X} \Vert \beta_{0} \Vert^{2}
+
\int_{X} \Vert \bar{\partial} \eta_{0} \Vert^{2}
\geq
\int_{X} \Vert \beta_{0} \Vert^{2}
>0
$ 
and therefore, also
$c_{n}\longrightarrow 0$, as $n\longrightarrow +\infty$.
In conclusion we have obtained:
\begin{equation*}
\Vert v_{n} \Vert^{2}_{H^{1}}
+
\Vert l_{n} \Vert_{W^{1,2}}
\longrightarrow 0
,
\; \text{ as } \; 
n\longrightarrow +\infty,
\end{equation*}
and this is in contradiction with \eqref{6.8}. Thus \eqref{6.6} is established. 

At this point, we can use Taylor expansion for $D_{0}$ around 
$(u_{0},\eta_{0})$ in $\mathcal{V}_{p}$ and the continuity of the map $\eta(u)$
in \eqref{map_H_1_to_W_1_p} to find a suitable $\delta_{0}>0$ sufficiently small,
such that, for every $u \in H^{1}(X)\; : \; \Vert u-u_{0} \Vert_{H^{1}}<\delta_{0}$, 
we have:
\begin{equation*}
\begin{split}
D_{0}(u,\eta(u))
= \; &
D_{0}(u_{0},\eta_{0})
+
\frac{1}{2}D_{0}^{\prime \prime}[u-u_{0},\eta(u)-\eta_{0}] \\
& \; +
o(
\Vert u-u_{0} \Vert_{H^{1}}^{2}
+
\Vert \bar{\partial}\eta(u) - \bar{\partial} \eta_{0}\Vert_{L^{p}}^{2}
) \\
\geq \; &
D_{0}(u_{0},\eta_{0})
+
\frac{\tau_{0}}{2}
(
\Vert u-u_{0} \Vert^{2}_{H^{1}}
+
\Vert \bar{\partial}\eta(u) - \bar{\partial} \eta_{0}\Vert_{L^{2}}^{2}
)
\\
 & \; +
o(
\Vert u-u_{0} \Vert_{H^{1}}^{2}
+
\Vert \bar{\partial}\eta(u) - \bar{\partial} \eta_{0}\Vert_{L^{p}}^{2}
).
\end{split}
\end{equation*}
Since $\eta_{0}=\eta(u_{0})$, 
we can use the estimate \eqref{d_bar_eta_difference_in_Lemma}  to
conclude that, 
for every $u \in H^{1}(X)\; : \; \Vert u-u_{0} \Vert_{H^{1}}<\delta_{0}$, there hold: 
\begin{equation*}
\begin{split}
D_{0}(u,\eta(u))
\geq \; &
D_{0}(u_{0},\eta_{0})
+
\gamma_{0}
(
\Vert u-u_{0} \Vert^{2}_{H^{1}}
+
\Vert \eta(u) -\eta_{0}\Vert_{W^{1,p}}^{2}
)
\end{split}
\end{equation*}
with suitable $\gamma_{0}>0$.
Consequently, if 
$(u,\eta) \in \mathcal{V}_{p}$ and $\Vert u-u_{0} \Vert_{H^{1}}<\delta_{0}$
then (by  \eqref{weak_criticality_l}) we find: 
\begin{equation*}
\begin{split}
D_{0}(u,\eta)
= \; &
D_{0}(u,\eta(u))
+
\int_{X}e^{(\kappa-1)u}\Vert \bar{\partial}\eta -\bar{\partial}\eta(u)\Vert^{2}\,dA \\
\geq \, &
D_{0}(u_{0},\eta_{0})
+
\gamma_{0}
(
\Vert u-u_{0} \Vert^{2}_{H^{1}}
+
\Vert \eta(u) -\eta_{0}\Vert_{W^{1,p}}^{2}
) \\
& \; +
\int_{X}e^{(\kappa-1)u}\Vert \bar{\partial}\eta -\bar{\partial}\eta(u)\Vert^{2}\,dA,
\end{split}
\end{equation*}
and \eqref{formula_of_the_lemma_on_D_0} is established. 
In particular, if $\Vert u-u_{0} \Vert_{H^{1}}<\delta$ and  
$(u,\eta)\neq (u_{0},\eta_{0})$, then 
$D_{0}(u,\eta) > D_{0}(u_{0},\eta_{0})$
and the proof is completed. 
\end{proof}

However, to know that any critical point of $D_{0}$ is a strict local minimum in 
$\mathcal{V}_{p}$ is not enough to ensure that $D_{0}$ admits only one critical point.
In fact we could be facing a situation similar to the function 
$f(z)=\vert e^{z}-1 \vert^{2} $ which admits infinitely many strict local minima
at $z=2\pi i n, n\in \Z$ and no other critical point.

Nonetheless, the presence of a strict local minimum for 
$D_{0}$ away from $(u_{t},\eta_{t})$ (for $t>0$ small)
allows us to exhibit a "mountain pass" structure (see \cite{Ambrosetti_Rabinowitz}) 
for the functional $D_{t}$, 
when $t>0$ is sufficiently small.
As shown in \cite{Huang_Lucia_Tarantello_2}, 
this fact will contradict the uniqueness of $(u_{t},\eta_{t})$, 
as claimed in Theorem  \ref{thm_uniqueness}. 
In this way we can finally obtain, 

\begin{proof}[\textbf{The Proof of Theorem \ref{thm_primo}}]

By using Lemma \ref{lem_strict_min} for the critical point 
$(u_{0},\eta_{0})$ of $D_{0}$ we have:

\begin{equation}\label{Claim_1_proof_theorem_primo}
\begin{split}
\textbf{Claim 1: }
&  
\; \forall \; \delta \in (0,\delta_{0}) 
\; \exists \; d_{\delta}>0 \; \text{ and } \; t_{\delta}>0,
\; \text{ such that for } \; t\in (0,t_{\delta})
\\
& 
\text{  we have} 
:   \;\; 
D_{t}(u,\eta) \geq D_{t}(u_{0},\eta_{0}) + d_{\delta}
\\
& 
\; \forall \; 
(u,\eta) \in \mathcal{V}_{p} \; : \; 
\Vert u-u_{0} \Vert^{2}_{H^{1}}
+
\Vert \eta(u) -\eta_{0} \Vert^{2}_{W^{1,p}}
=
\delta
.
\end{split}
\end{equation}
To establish \eqref{Claim_1_proof_theorem_primo}, we simply apply
\eqref{formula_of_the_lemma_on_D_0} as follows:
\begin{equation*}
\begin{split}
D_{t}(u,\eta)
= \; &
t \int_{X}e^{u}\,dA + D_{0}(u,\eta) \\
\geq \; &
t\int_{X}e^{u}\,dA
+
D_{0}(u_{0},\eta_{0})
+
\int_{X}e^{(\kappa-1)u}\Vert \bar{\partial}\eta - \bar{\partial} \eta(u) \Vert^{2}\,dA \\
& \; +
\gamma_{0}
( \Vert u-u_{0} \Vert_{H^{1}}^{2}
+
\Vert \eta(u) - \eta_{0} \Vert^{2}_{W^{1,p}} 
)
\\
\geq \; &
D_{t}(u_{0},\eta_{0})
+
\gamma_{0}\delta + t\int_{X}(e^{u}-e^{u_{0}})\,dA 
\geq 
D_{t}(u_{0},\eta_{0})+\gamma_{0}\delta
-C_{\delta}t
\end{split}
\end{equation*}
with a suitable constant $C_{\delta}>0$. Clearly, the estimate above readily 
implies \eqref{Claim_1_proof_theorem_primo}.   

\

Next, we argue by contradiction and assume:
\begin{equation*}
\begin{split}
\; \exists \; \varepsilon_{0}>0 \; \text{ and } \;
t_{n}\longrightarrow 0^{+}
\; : \;  
\Vert u-u_{t_{n}} \Vert_{H^{1}}^{2}
+
\Vert \eta_{0}-\eta_{t_{n}} \Vert_{W^{1,p}}^{2} 
\geq
\varepsilon_{0},
\end{split}
\end{equation*}

for all $n\in \N$. 

\

\noindent
So, we fix $0<\delta< \min\{ \frac{\varepsilon_{0}}{2},\delta_{0} \} $
and take $n_{0}=n_{0}(\delta) \in \N$ sufficiently large, so that
$t_{0}:= t_{n_{0}} \in (0,t_{\delta})$. 
Consequently, 
\begin{equation}\label{mountain_pass_1}
\begin{split}
\Vert u_{0}-u_{t_{0}} \Vert_{H^{1}}^{2} 
+
\Vert \eta_{0} - \eta_{t_{0}} \Vert_{W^{1,p}}^{2} \geq \varepsilon_{0}
\end{split}
\end{equation}
and
\begin{equation}\label{mountain_pass_2}
D_{t_{0}}(u,\eta)
\geq
D_{t_{0}}(u_{0},\eta_{0}) + d_{\delta}
\; \text{ for } \; 
\Vert u-u_{0} \Vert^{2}_{H^{1}}
+
\Vert \eta(u) -\eta_{0} \Vert^{2}_{W^{1,p}}
=
\delta.
\end{equation}
Also recall that
$D_{t_{0}}(u_{0},\eta_{0}) \geq D_{t_{0}}(u_{t_{0}},\eta_{t_{0}})$. So,  
from \eqref{mountain_pass_1} and \eqref{mountain_pass_2},  
we see that $D_{t_{0}}$ admits a mountain pass structure
in the sense of \cite{Ambrosetti_Rabinowitz}. 
But, as in \cite{Huang_Lucia_Tarantello_2}, we show that this is impossible, 
since we can deduce the existence of another critical point for $D_{t_{0}}$ 
different from $(u_{t_{0}},\eta_{t_{0}})$,
in contradiction to Theorem \ref{thm_uniqueness}. 
To be more precise, let $P_{0}=(u_{0},\eta_{0})$ and 
$P_{1}=(u_{t_{0}},\eta_{t_{0}})$. 
We know that $P_{0}\neq P_{1}$ (see \eqref{mountain_pass_1}), 
and that 
$\eta_{0}=\eta(u_{0})$ and $\eta_{t_{0}}=\eta(u_{t_{0}})$. 
We define the family of paths 
$$
\mathcal{P}
=
\{ 
\gamma \in C^{0}([0,1],\mathcal{V}_{p})
\; : \; 
\gamma(0)=P_{0}, \gamma(1)=P_{1}
\}.  
$$
Clearly, $\mathcal{P}$ is not empty, as 
$\gamma(s)=(1-s)P_{0}+sP_{1} \in \mathcal{P}$. 
Moreover, by setting: 
$
d(\gamma_{1},\gamma_{2})
=
\max_{s\in [0,1]}\Vert \gamma_{1}(s) - \gamma_{2}(s)\Vert_{\mathcal{V}_{p}}
$, 
we see that $(\mathcal{P},d)$ defines a complete metric space. 

\begin{equation}\label{claim_2_proof_theorem_}
\begin{split}
\; \text{ \textbf{Claim 2:} } \; \quad \quad \quad \quad \quad  
c=\inf_{\gamma \in \mathcal{P}} \max_{s\in [0,1]} D_{t_{0}}(\gamma(s))
\geq
D_{t_{0}}(u_{0},\eta_{0})
+
d_{\rho}
\end{split}
\end{equation}
Indeed, if we take $\gamma\in \mathcal{P}$ with 
$\gamma(s)=(u(s),\eta(s))$, $s\in [0,1]$ and we define 
$$
f(s)=\Vert u(s)-u_{0} \Vert^{2}_{H^{1}}
+
\Vert \eta(u(s)) -\eta_{0}\Vert^{2}_{W^{1,p}}
\in 
C^{0}([0,1]),
$$
we see that: $f(0)=0$ while 
$
f(1)
=
\Vert u_{t_{0}}-u_{0} \Vert^{2}_{H^{1}} 
+ 
\Vert \eta_{t_{0}} - \eta_{0} \Vert^{2}_{W^{1,p}} \geq \varepsilon_{0} > \delta 
$. 
So, by continuity, there exists $s_{0} \in [0,1]$ such that
$f(s_{0}) =  \delta$. Therefore, by \eqref{mountain_pass_2}, we find: 
$
\max_{s\in [0,1]}D_{t_{0}}(\gamma(s))
\geq 
D_{t_{0}}(\gamma(s_{0}))
\geq
D_{t_{0}}(u_{0},\eta_{0}) + d_{\delta}
$, 
and \eqref{claim_2_proof_theorem_} follows. 

\

We are going to show that $c$ in \eqref{claim_2_proof_theorem_} 
defines a critical value for $D_{t_{0}}$, and since  
$c> D_{t_{0}}(u_{0},\eta_{0}) \geq D_{t_{0}}(u_{t_{0}},\eta_{t_{0}})$,  
the corresponding critical point must be different from $(u_{t_{0}},\eta_{t_{0}})$. 
In this way we reach a contradiction to Theorem \ref{thm_uniqueness}. 

To this purpose we note first that, if $\gamma \in \mathcal{P}$ with 
$\gamma(s)=(u(s),\eta(s))$, $s\in [0,1]$, then 
setting 
$\tilde{\gamma}(s)=(u(s),\eta(u(s)))$, $s\in [0,1]$, we easily check that also
$\tilde{\gamma} \in \mathcal{P}$ and 
$
D_{t_{0}}(\gamma(s)) \geq D_{t_{0}}(\tilde{\gamma}(s)).
$
Next, we use the following Ekeland's $\epsilon$-principle

\begin{thmx}[\cite{Struwe_Book}]\label{thm_Ekeland}
Let $(Y,d)$ be a complete metric space and 
$F:Y\longrightarrow \R$ a non-negative and lower semi-continuous functional.
For every $\epsilon>0$ let $\gamma^{0}_{\epsilon} \in Y$ be such that: 
$F(\gamma^{0}_{\epsilon}) \leq \epsilon + \inf F$. Then  there exists  
$\gamma_{\epsilon} \in Y$ such that
$$
F(\gamma_{\epsilon}) \leq F(\gamma_{\epsilon}^{0})
,\quad
d(\gamma_{\epsilon},\gamma_{\epsilon}^{0}) \leq \sqrt{\epsilon}
\; \text{ and } \; 
F(\gamma) \geq F(\gamma_{\epsilon})- \sqrt{\epsilon}d(\gamma,\gamma_{\epsilon}),
\; \forall \; \gamma\in Y.
$$
\end{thmx}

We are going to apply Theorem \ref{thm_Ekeland} with $(Y,d)=(\mathcal{P},d)$ and
\begin{equation}\label{def_F}
F(\gamma)=\max_{s\in [0,1]}D_{t_{0}}(\gamma(s)).
\end{equation} 
Therefore, for given $\epsilon>0$, 
without loss of generality we can take a path $\gamma^{0}_{\epsilon} \in \mathcal{P}$
of the form 
$\gamma^{0}_{\epsilon}(s)=(u^{0}_{\epsilon}(s),\eta(u^{0}_{\epsilon}(s)))$
and satisfying: 
$F(\gamma^{0}_{\epsilon}) < \epsilon + \inf_{\gamma \in \mathcal{P}} F(\gamma)$
with $F$ in \eqref{def_F}. 
As a consequence, we find a path $\gamma_{\epsilon} \in \mathcal{P}$ such that,
\begin{equation*}
c\leq \max_{s\in [0,1]}D_{t_{0}}(\gamma_{\epsilon}(s))< c+\epsilon
,\;
\max_{s\in [0,1]} \Vert \gamma_{\epsilon} - \gamma^{0}_{\epsilon} \Vert_{\mathcal{V}_{p}}
\leq \sqrt{\epsilon},
\end{equation*}
and, 
\begin{equation*}%\label{Ekeland_3}
\begin{split}
\max_{s\in [0,1]} D_{t_{0}}(\gamma(s)) 
\geq \max_{s\in [0,1]}D_{t_{0}}(\gamma_{\epsilon}(s))
-
\sqrt{\epsilon}\max_{s\in [0,1]}
\Vert \gamma(s) - \gamma_{\epsilon}(s) \Vert_{\mathcal{V}_{p}},
\end{split}
\end{equation*}
for every $\gamma \in \mathcal{P}$. 
Furthermore, in view of \eqref{claim_2_proof_theorem_}, the set 
$$
T_{\epsilon}
:=
\{ 
\tilde{s} \in [0,1]
\; : \; 
D_{t_{0}}(\gamma_{\epsilon}(\tilde{s}))
=
\max_{s\in [0,1]}D_{t_{0}}(\gamma_{\epsilon}(s))
 \} 
$$
is relatively compact in the open interval $(0,1)$, that is 
$T_{\epsilon}\subset \subset (0,1)$. As a consequence,
by Lemma 5.4 of \cite{Huang_Lucia_Tarantello_2}:
$$
\; \exists \; s_{\epsilon}\in T_{\epsilon}
\; : \; 
\Vert D^{\prime}_{t_{0}}(\gamma_{\epsilon}(s_{\epsilon})) \Vert_{\mathcal{V}_{p}} \leq 
\sqrt{\epsilon}. 
$$

So, along a sequence $\epsilon_{n}\longrightarrow 0$, we find 
$(u_{n},\eta_{n})=\gamma_{ \epsilon_{n}}(s_{\epsilon_{n}})$ 
and 
$(u^{0}_{n},\eta(u^{0}_{n}))= \gamma^{0}_{\epsilon_{n}}(s_{\epsilon_{n}}) \in \mathcal{V}_{p}$
such that, as $n\longrightarrow \infty$, 
\begin{align}
&
D_{t_{0}}(u_{n},\eta_{n})\longrightarrow c 
,\;
\Vert D^{\prime}_{t_{0}}(u_{n},\eta_{n}) \Vert_{\mathcal{V}_{p}^{*}}\longrightarrow 0
\label{ps_1}
\\
& 
\Vert u_{n}-u^{0}_{n} \Vert_{H^{1}}+ \Vert \eta_{n}-\eta(u^{0}_{n}) \Vert_{W^{1,p}}
\longrightarrow 0.
\label{ps_2}
\end{align}
As before, from the first limit in \eqref{ps_1}, we deduce that 
$u_{n}$ is uniformly bounded in $H^{1}(X)$, and so, by \eqref{ps_2}, 
also $u^{0}_{n}$ is uniformly bounded in $H^{1}(X)$. 

As a consequence of (i) in Lemma \ref{eta_u_map} we deduce that 
necessarily $\eta(u^{0}_{n})$ is uniformly bounded in $W^{1,p}(X,E)$, and 
(by \eqref{ps_2}) we find $\Vert \eta_{n} \Vert_{W^{1,p}} \leq C$ for suitable $C>0$. 

Therefore the (PS)-sequence $(u_{n},\eta_{n})$ satisfies \eqref{weak_palais_smale} 
and so we can apply Lemma \ref{Palais_Smale} to conclude
that $c$ is a critical value for $D_{t_{0}}$. 
Hence we reach the desired contradiction and conclude that:
\begin{equation}\label{u_t_eta_t_limit}
(u_{t},\eta_{t})\longrightarrow (u_{0},\eta_{0}),
\; \text{ as } \; t\longrightarrow 0^{+}, \; \text{ in } \;  \mathcal{V}_{p}.
\end{equation}
Since any other critical point of $D_{0}$ must satisfy \eqref{u_t_eta_t_limit}, 
by the uniqueness of $(u_{t},\eta_{t})$, 
we deduce that $(u_{0},\eta_{0})$ must be the \underline{only} critical point of 
$D_{0}$. 
Finally, for every $(u,\eta)\in \Lambda$, we have:
$$
D_{0}(u,\eta)
=
\lim_{t\to 0^{+}}D_{t}(u,\eta)
\geq 
\lim_{t \to 0^{+}}D_{t}(u_{t},\eta_{t})
=
D_{0}(u_{0},\eta_{0}).
$$
Consequently, $D_{0}$ is bounded from below in $\Lambda$ and 
$(u_{0},\eta_{0})$ is its global minimum point. 
This concludes the proof of Theorem \ref{thm_primo}.  
\end{proof}

Next, we notice that, by the strict positivity of the Hessian $D_{t}^{\prime \prime }$ at
$(u_{t},\eta_{t})$, as pointed out by the estimates in 
\eqref{seconda_variatione_D_t}-\eqref{6.4a}, 
we can use the Implicit Function Theorem 
(cf. \cite{Nirenberg_Topics_In_Nonlinear_Functional_Analysis})  
for the map:
\begin{equation*}
F:\R^{+}\times H^{1}(X)\times W^{1,p}(X,E)
\longrightarrow 
(H^{1}(X)\times W^{1,p}(X,E))^{*}
,\;
p>2,
\end{equation*}
given by:
\begin{equation*}
F(t,u,\eta) 
=
(\frac{\partial}{\partial u} D_{t}(u,\eta),
\frac{\partial}{\partial \eta}D_{t}(u,\eta))
\end{equation*}
in order to show the $C^{2}$-dependence of $(u_{t},\eta_{t})$ 
with respect to the parameter $t\in (0,+\infty)$. 
Furthermore, by setting:

\begin{equation*}%\label{6.12a}
c_{t}
:=
D_{t}(u_{t},\eta_{t})
\leq 
D_{t}(u,\eta)
,\;
\; \forall \; (u,\eta)\in \Lambda,
\end{equation*}
we have, $c_{t}\in C^{2}(]0,+\infty[)$ and we
may compute:
\begin{equation}\label{6.12b}
\dot c_{t}
=
\frac{d}{dt}c_{t}
=
\frac{\partial}{\partial u}D_{t}(u_{t},\eta_{t})
\dot u_{t}
+
\frac{\partial}{\partial \eta}D_{t}(u_{t},\eta_{t})
\dot \eta_{t}
+
\int_{X} e^{u_{t}}
=
\int_{X} e^{u_{t}},
\end{equation}
and \eqref{6.12b} confirms the fact that $c_{t}$ is increasing for 
$t\in (0,+\infty)$. Furthermore,
\begin{lemma}\label{star}
\begin{enumerate}[label=(\roman*)]
\item $c_{t}=D_{t}(u_{t},\eta_{t})$ is concave in $(0,+\infty)$.
\item The function: $t\longrightarrow t\int^{}_{X} e^{u_{t}} \,dt$
is increasing in $ (0,+\infty)$. 	
\end{enumerate}
\end{lemma}	

\begin{proof}
By straightforward calculations we find:
\begin{equation*}%\label{6.13}
\ddot{c}_{t}
=
\int_{X} e^{u_{t}}\dot u_{t}dA
=
\frac{d^{2}}{dt^{2}}D_{t}(u_{t},\eta_{t})
=
2 \int_{X} e^{u_{t}}\dot u_{t}dA
+
D_{t}^{\prime \prime }[\dot u_{t},\dot \eta_{t}],
\end{equation*}
and so, 
$\int_{X}e^{u_{t}} \dot{u}_{t}dA=-D^{''}[\dot{u}_{t},\dot{\eta}_{t}]$. 
By \eqref{seconda_variatione_D_t}-\eqref{6.4a}, also we know that:
$$
D_{t}^{\prime \prime }[\dot u_{t},\dot \eta_{t}]
>
t \int_{X} e^{u_{t}} \dot u_{t}^{2} dA
$$
and we obtain: 
\begin{equation}\label{6.14}
\ddot c_{t}
=
\int_{X} e^{u_{t}}\dot u_{t}dA
=
-
D_{t}^{\prime \prime }[\dot u_{t},\dot \eta_{t}]
\leq
-t \int_{X} e^{u_{t}}\dot u_{t}^{2} dA \leq 0.
\end{equation}
Hence,  $c_{t}$ is \underline{concave}, and we have: 
\begin{equation}\label{6.15}
\int_{X} e^{u_{t}} \dot u_{t}dA
+
t \int_{X} e^{u_{t}} \dot u_{t}^{2}dA	\leq 0. 
\end{equation}
Therefore, by \eqref{6.14},\eqref{6.15} and Jensen's inequality, we deduce:
\begin{equation*}
t
(
\int_{X} \frac{e^{u_{t}}}{\int_{X} e^{u_{t}}dA}\dot u_{t}dA
)^{2}
\leq
t \int_{X} 
\frac{e^{u_{t}}}{\int_{X} e^{u_{t}}dA}
(\dot u_{t})^{2}dA
\leq
\vert 
\int_{X} \frac{e^{u_{t}}}{\int_{X} e^{u_{t}}dA}\dot u_{t}dA
\vert,
\end{equation*}
namely:
$
t
\vert 
\int_{X} \frac{e^{u_{t}}}{\int_{X} e^{u_{t}}dA} \dot u_{t} dA\vert
\leq 1 
$ 
or equivalently:
$
\int_{X} e^{u_{t}}dA
+
t \int_{X} e^{u_{t}}\dot u_{t}dA \geq 0.
$
Thus, we have proved that,
$\frac{d}{dt}
(
t \int_{X} e^{u_{t}}dA
)
> 
0
$
and (ii) follows.
\end{proof}
On the other hand (by integration of the first equation in \eqref{system_P_t}) we have:
\begin{equation}\label{6.17}
t \int_{X} e^{u_{t}}dA
+
4(\kappa-1)\int_{X} e^{(k-1)u_{t}}
\Vert 
\beta_{0} + \bar{\partial}\eta_{t}
\Vert^{2}
dA
=
4\pi (\mathfrak{g}-1)
\end{equation}
and so, the following integral term:
\begin{equation*}%\label{6.18}
4(\kappa-1)\int_{X} e^{(\kappa-1)u_{t}}
\Vert \beta_{0} + \bar{\partial}\eta_{t}\Vert^{2}
dA
=
4\pi(\mathfrak{g}-1)-t \int_{X} e^{u_{t}}dA
\end{equation*}
is \underline{decreasing} as a function of $t\in (0,+\infty)$. Therefore, it is well defined the value:
\begin{equation}\label{6.19}
\begin{split}
\rho([\beta]) 
= & \;\rho([\beta_{0}])
= 
4(\kappa-1) \lim_{t \to 0^{+} }
\int_{X} e^{(\kappa-1)u_{t}}
\Vert \beta_{0} + \bar{\partial}\eta_{t}\Vert^{2}
dA \\
= &
4(\kappa-1)
\sup_{t>0}
\int_{X} e^{(\kappa-1)u_{t}}
\Vert \beta_{0} + \bar{\partial}\eta_{t}\Vert^{2}
dA, 
\end{split}
\end{equation}
and naturally we we wish to  identify the value of $\rho([\beta])$
in terms of the given cohomology class
$[\beta] \in \mathcal{H}^{0,1}(X,E)$.

So far, concerning the value $\rho([\beta])$ in \eqref{6.19}, 
we know the following: 
\begin{proposition}\label{proposition_on_rho}
For given $[\beta]\in \mathcal{H}^{0,1}(X,E)$ 
with harmonic representative $\beta_{0} \in A^{0,1}(X,E)$ 
we have:
\begin{enumerate}[label=(\roman*)]
\item 
$\rho([\beta])\in [0,4\pi(\mathfrak{g}-1)]$
and  $\rho([\beta])=0  \Longleftrightarrow [\beta]=0$
\item
If $[\beta]\neq 0$ then for every $\rho \in (0,\rho([\beta]))$ there exits a unique 
$t \in (0,+\infty)$ such that,
\begin{equation*}
\rho 
= 
4(\kappa-1)\int_{X} e^{(\kappa-1)u_{t}}
\Vert 
\beta_{0}+\bar{\partial}\eta_{t}
\Vert^{2}
dA,
\end{equation*}
where $(u_{t},\eta_{t})$ is the global minimum 
(and unique critical point) of the functional 
$D_{t}$. 
\end{enumerate}
\end{proposition}	

It is also clear that,
$ D_{0}$ 
 is bounded from below on  $\Lambda$, 
if and only if
$$\inf_{t>0}c_{t}
=
\lim_{t \to 0^{+} }c_{t}:=c_{0}>-\infty 
\; \text{ and } \; 
c_{0}=\inf_{\Lambda}D_{0}.
$$ 

More importantly the following holds:
\begin{proposition}\label{prop_6.5}
If $D_{0}$ is bounded from below on $\Lambda$
then $\rho([\beta])=4\pi(\mathfrak{g}-1)$. 
\end{proposition}	
\begin{proof}
Argue by contradiction and assume that
$0\leq \rho([\beta])<4\pi(\mathfrak{g}-1)$. Hence
\begin{equation*}
\lim_{t \to 0^{+} }
t \int_{X} e^{u_{t}}dA
=
4\pi (\mathfrak{g}-1)-\rho([\beta])
:=
\mu
>0.
\end{equation*}
On the other hand (by de L'Hopital rule)
\begin{equation*}
\lim_{t \to 0^{+} } \frac{c_{t}}{\ln t}
=
\lim_{t \to 0^{+} } t\dot c_{t}
=
\lim_{t \to 0^{+} }
t \int_{X} e^{u_{t}}dA=\mu>0
\end{equation*}
and this implies that, $c_{t}\longrightarrow -\infty$
as $t\longrightarrow 0^{+}$, a contradiction. 
\end{proof}

\noindent
Now the delicate questions to ask are the following:
\begin{align} 
&\text{ if } \;  \rho([\beta])=4\pi(\mathfrak{g}-1) 
\; \text{ is it true that } \; 
D_{0}
\; \text{  is bounded from below? } 
\label{Question_2_i}
\\
&\text{ if } \;  
D_{0}
\; \text{  is bounded from below then is the infimum attained? } \; 
%\label{Question_2_ii}
\end{align}	
\noindent
In order to investigate the questions raised above, 
for $(u_{t},\eta_{t})$, the global minimum of $D_{t}$ in $\Lambda$ 
(given in Theorem \ref{thm_A_intro})
we let:
\begin{equation*}%\label{6.11_prime}
u_{t}
=
w_{t} 
+
d_{t}
,
\; \text{ with } \; 
\int_{X} w_{t}dA=0
\; \text{ and } \;  
d_{t}=\fint_{X}u_{t}dA
\end{equation*}
\begin{equation*}%\label{6.12_prime}
\beta_{t}=\beta_{0}+\bar{\partial}\eta_{t} \in A^{0,1}(X,E)
\; \text{ and } \; 
\alpha_{t}=e^{u_{t}}*_{E}\beta_{t}\in C_{\kappa}(X)
\end{equation*}
and set
\begin{equation*}%\label{6.13_prime}
s_{t}\in \R \; : \; e^{(\kappa-1)s_{t}}=\Vert \alpha_{t} \Vert_{L^{2}}^{2}.
\end{equation*}
With the notation above, we check the following easy properties:
\begin{lemma}\label{lem_6.4}
For given $[\beta] \in \mathcal{H}^{0,1}(X,E)$ with harmonic representative
$\beta_{0}\in [\beta]$ and $t>0$ there holds
\begin{eqnarray}
(i) \quad &
\; \forall \;   q  \in [ 1,2 )  \;
\; \exists \; 
C_{q}>0 \; : \; \Vert w_{t} \Vert_{W^{1,q}(X)}\leq C_{q}
\label{6.14_prime} \\
(ii) \quad &
w_{t}\leq C \; \text{ in  } \; X \label{6.16_prime} \\
(iii) \quad &
te^{u_{t}}\leq 1 \; \text{ in  } \; X \label{6.15_prime} \\
(iv) \quad &
s_{t}\leq d_{t}+C 
\; \text{ for a suitable constant  } C>0\; \label{6.17_prime} \\
(v) \quad &
\; \text{ if } \; [\beta]\neq 0,
\;\text{ there exists a constant } \; \gamma=\gamma(\kappa) >0
\notag \\
&\; \text{ (depending on $\kappa$ only) such that, } \; 
 \notag
\\ \quad &
\int_{X}e^{-u_{t}}dA 
%\geq \gamma \frac{\Vert \beta \Vert_{L^{2}}}{\rho([\beta])} 
\geq
\gamma \fint_{X} \Vert \beta_{0} \Vert^{2}dA
\label{property_v}
\end{eqnarray}	
\end{lemma}	
\begin{proof}
The inequality \eqref{6.14_prime} is a direct consequence of the fact 
that the right hand side 
of the first equation in \eqref{system_P_t} is uniformly bounded in $L^{1}(X)$ 
(see \eqref{6.17}),
while \eqref{6.15_prime} follows easily by the maximum principle.

Next, to get \eqref{6.16_prime} we use the Green's representation formula for $w_{t}$,
with  $G(p,q)$ the Green function of the Laplace-Beltrami operator on $X$
satisfying: $\int_{X} G(p,q)dA(q)=0$
and $G(p,q)\leq a $ in $(X\times X)$ (see \cite{Aubin_Book}). 
We have:
\begin{equation*}
\begin{split}
w_{t}(p)
= & \;
2 \int_{X} G(p,q)
( 
te^{u_{t}(q)}
+
4(\kappa-1)e^{(\kappa-1)u_{t}(q)}
\Vert \beta_{t} \Vert^{2}(q)
-
1
)
dA(q) \\
= & \;
2 \int_{X} G(p,q)
(
te^{u_{t}(q)}
+
4(\kappa-1)e^{(\kappa-1)u_{t}(q)}
\Vert \beta_{t} \Vert^{2}(q)
)
dA(q)\\
\leq & \;
2a \int_{X} te^{u_{t}}+4(\kappa-1)e^{(\kappa-1)u_{t}}\Vert \beta_{t} \Vert^{2}
dA
\leq
8\pi (\mathfrak{g}-1)a.
\end{split}
\end{equation*}
To establish (iv) we  estimate:
\begin{equation*}
\begin{split}
e^{(\kappa-1)s_{t}}
= \;&
\Vert \alpha_{t} \Vert_{L^{2}}^{2}
= 
\int_{X} \Vert \alpha_{t} \Vert^{2}dA
=
\int_{X} e^{2(\kappa-1)u_{t}}\Vert \beta_{t} \Vert^{2}dA
\\
\leq & \;
C e^{(\kappa-1)d_{t}}
\int_{X} e^{(\kappa-1)u_{t}}\Vert \beta_{t} \Vert^{2}dA
\leq  
Ce^{(\kappa-1)d_{t}}
\frac{\rho([\beta])}{4(\kappa-1)}
\leq 
Ce^{(\kappa-1)d_{t}}
\end{split}
\end{equation*}
with a suitable constant $C>0$, and \eqref{6.17_prime}  follows.  

Finally, to get \eqref{property_v}, we recall first that, if
$\beta_{0}\in A^{0,1}(X,E)$ is harmonic then we have the estimate:
$
\Vert \beta_{0} \Vert_{L^{\infty}}\leq C \Vert \beta_{0} \Vert_{L^{2}} 
$
for suitable $C>0$. Consequently,
\begin{equation*}
\begin{split}
\int_{X} \Vert \beta_{0} \Vert^{2}dA 
= & \; 
\int_{X}\langle \beta_{0},\beta_{0} + \bar{\partial}\eta_{t} \rangle dA
\leq 
\Vert \beta_{0} \Vert_{L^{\infty}}
\; \int_{X} \Vert \beta_{0} + \bar{\partial} \eta_{t}\Vert dA \\
\leq  & \; 
C \Vert \beta_{0} \Vert_{L^{2}} 
\left(  \int_{X}e^{-u_{t}}dA\right)^{\frac{1}{2}} 
\left( \int_{X}e^{u_{t}}\Vert \beta_{0} + \bar{\partial} \eta_{t} \Vert^{2} dA\right )^{\frac{1}{2}}
\\
\leq & \; 
C \Vert \beta_{0} \Vert_{L^{2}} 
\left(\frac{\rho([\beta_{0}])}{4(\kappa-1)}\right)^{\frac{1}{2}}
\left(  \int_{X}e^{-u_{t} }\right)^{\frac{1}{2}} dA.
\end{split}
\end{equation*}
Hence, if $[\beta]\neq 0$, then 
$0<\rho([\beta])\leq 4\pi(\mathfrak{g}-1)$ and we easily deduce  
\eqref{property_v} from the above estimate.
\end{proof}

\begin{remark}%\label{rem_3_1}
By \eqref{6.17} and Jensen's inequality, we have:
\begin{equation*}%\label{6.19a}
te^{d_{t}}\leq 1.
\end{equation*} 
\end{remark}

According to the fixed basis for $C_{\kappa}(X)$ in \eqref{basis_for_C_k_X}, 
we can decompose:
\begin{equation*}%\label{6.19b}  
\alpha_{t}=\sum_{j=1}^{\nu}a_{j,t}s_{j}
\; \text{ with } \; 
\Vert \alpha_{t} \Vert^{2}_{L^{2}}
=
\sum_{j=1}^{\nu}\vert a_{j,t} \vert^{2}=e^{(\kappa-1)s_{t}}.
\end{equation*}
Furthermore, along a sequence
$t_{k}\longrightarrow 0^{+}$, we may assume that, 
for $d_{k}:=d_{t_{k}}$ and $u_{k}=u_{t_{k}}$, there holds, as $k\longrightarrow +\infty$:
\begin{equation*}
w_{k}:=w_{t_{k}}\longrightarrow w_{0}
\; \text{ and } \; 
e^{w_{k}}\longrightarrow e^{w_{0}} 
\; \text{ pointwise and in   } \; 
L^{p }(X),\; p>1;
\end{equation*}
\begin{equation}\label{6.20_prime}
t_{k}e^{d_{k}} \longrightarrow \mu \geq 0
\; \text{ and } \;
t_{k}e^{u_{k}}
\longrightarrow 
\mu e^{w_{0}}
\; \text{ pointwise and in  } \; 
L^{p}(X). 
\end{equation}
In addition, in view of Remark \ref{vanishing_property_of_C_kappa_X},
also we may assume that, 
for suitable $1\leq N \leq 2\kappa(\mathfrak{g}-1)$ and $k$ large:
\begin{equation*}%\label{multiplicity}
\begin{split}
&
\alpha_{t_{k}}\in C_{\kappa}(X) \setminus \{ 0 \} 
 \; \text{ admits } \; N\text{-\underline{distinct} zeroes: } \; 
\{ z_{1,k},\ldots,z_{N,k} \} \; \text{ with} \\
& \text{corresponding multiplicity 
 } \; 
\{ n_{1},\ldots,n_{N} \} \subset \N
\text{ and  } \; 
\sum_{j=1}^{N} n_{j}=2\kappa(\mathfrak{g}-1).
\end{split}
\end{equation*}
Moreover,
\begin{equation}\label{6.28}
z_{j,k}\longrightarrow z_{j},
\; \text{ as } \; 
k\longrightarrow +\infty,\;
j \in \{ 1,\ldots,N \}.
\end{equation}
Next, we set:
\begin{equation}\label{hat_alpha}
\hat{\alpha}_{t}
=
\frac{\alpha_{t}}{\Vert \alpha_{k} \Vert_{L^{2}}}
=
e^{-\frac{\kappa-1}{2}s_{t}}\alpha_{t}
=
\sum_{j=1}^{\nu}\hat{a}_{j,t}s_{j}
\; \text{ with } \; 
\hat{a}_{j,t}=e^{-\frac{\kappa-1}{2}s_{t}}a_{j,t}.
\end{equation}
Since 
$
\vert \hat{a}_{j,t} \vert \leq 1, \; \forall \; 
j=1,\ldots,N$,
also we may suppose that, as $k\longrightarrow +\infty$,
\begin{equation}\label{6.21a}
\hat{a}_{j,t_{k}} \longrightarrow \hat{a}_{j}
\; \text{ and for } \;
\hat{\alpha}_{0}
:=
\sum_{j=1}^{\nu}\hat{a}_{j}s_{j} 
\in C_{\kappa}(X)
\; \text{ : } \;  
\hat{\alpha}_{t_{k}} \longrightarrow  \hat{\alpha}_{0},
\end{equation}
with 
$\Vert \hat{\alpha}_{0} \Vert_{L^{2}}=1$, and so 
$\hat{\alpha}_{0}\neq 0$. Therefore, $\hat{\alpha}_{0}$
vanishes \underline{exactly} at the set: 
\begin{equation*}%\label{6.21b}
Z := \{ z_{1},\ldots,z_{N}\}
\end{equation*}
with $z_{j}$ given in \eqref{6.28}. 
Since the total multiplicity of each $z_{j}$ in $Z$ adds up to
the value: $2\kappa(\mathfrak{g}-1)$,
by Remark \ref{vanishing_property_of_C_kappa_X} we know that,  
$\hat{\alpha}_{0}$ cannot vanish anywhere else.
Observe that the points in $Z$ may not we distinct. 

\

\noindent 
We define:
\begin{equation}\label{6.21}
\xi_{k}=-(\kappa-1)(u_{t_{k}}-s_{t_{k}})
\end{equation}
and
\begin{equation*}%\label{6.22}
R_{k} = 8(\kappa-1)^{2}\Vert \hat{\alpha}_{t_{k}} \Vert^{2}
\end{equation*}
such that
\begin{equation}\label{6.23}
-\Delta \xi_{k}
=
R_{k}e^{\xi_{k}}-f_{k}
\; \text{ in  } \; 
X
\end{equation}
with 
$f_{k}
: =
2(\kappa-1)(1-t_{k}e^{u_{t_{k}}})
$ 
satisfying:
$\Vert f_{k} \Vert_{L^{\infty}(X)}\leq C$ and 
\begin{equation}\label{6.24}
\begin{split}
&
f_{k}
\longrightarrow
f_{0}
=:
2(\kappa-1)(1-\mu e^{w_{0}})
\, \text{ in } \, 
L^{p}(X),\; p>1,
\\
& 
\int_{X}f_{0}=2(\kappa-1)\rho([\beta])>0, 
\; \text{ for } \; [\beta]\neq 0
\end{split} 
\end{equation} 
see \eqref{6.15_prime}-\eqref{6.20_prime} and 
\eqref{6.17}-\eqref{6.19}. 
Notice that,
\begin{equation}%\label{form_of_R_k}
R_{k}(z)
=
8(\kappa-1)^{2}
\Pi_{j=1}^{N}(d_{g_{X}}(z,z_{j,k}))^{2 n_{j}}G_{k}(z)
,\;
z \in X,
 \end{equation}
where 
\begin{align}
&
z_{j,k}\neq z_{l,k} 
\; \text{ for } \; 
j\neq l\in \{1,\ldots,N \}
,\;
n_{j}\in \N
\; : \; \sum_{j=1}^{N}n_{j}=2\kappa(\mathfrak{g}-1)
%\label{3.48a}  needed for referencing multiple equations from to
\\ &
G_{k}\in C^{1}(X)
\; : \; 
0 < a \leq G_{k} \leq b	
\; \text{ and } \; 
\vert \nabla G_{k} \vert \leq A
\; \text{ in   } \; X,
%\label{3.48b} same
\end{align}
 with suitable positive constants $a,b$ and $A$.
Hence (by taking a subsequence if necessary) we may assume that, 
\begin{equation}\label{6.27}
G_{k}\longrightarrow G_{0}
\; \text{ in   } \; 
C^{0}(X) 
\; \text{ and so } \; 
R_{k}\longrightarrow R_{0}
\; \text{ in } \; 
C^{0}(X),
\; \text{ as } \; 
k\longrightarrow +\infty,
\end{equation}
with
\begin{equation*}
R_{0}(z)
=
8(\kappa-1)^{2}
\Pi_{j=1}^{N} (d_{g_{X}}(z,z_{j}))^{n_{j}}G_{0}(z)
=
8(\kappa-1)^{2}\Vert \hat{\alpha}_{0} \Vert^{2}.
\end{equation*}
Next, in \eqref{definition_Z_0_intro} we have identified the subset 
$Z_{0}\subseteq Z$ (possibly empty) given by "collapsing" zeroes, 
namely by those zeroes of $R_{0}$ 
(or equivalently $\hat{\alpha}_{0}$) 
corresponding to the limit points of 
\underline{different} zeroes of $R_{k}$
(or equivalently $\hat{\alpha}_{k}$), 
namely: 
\begin{equation*}%\label{definition_Z_0}
\begin{split}
Z_{0}
=
\{ 
z \in Z
\; : \;
&
 \exists \; s\geq 2,\;
1\leq j_{1}<\ldots<j_{s}\leq N
\; \text{ such that} \; \\ 
& 
z=z_{j_{1}}=\ldots=z_{j_{s}}
\; \text{ and } \;
z \not \in Z \setminus \{  z_{j_{1}},\ldots,z_{j_{s}} \}  
\} . 
\end{split}
\end{equation*}  
With the information above, we see that Theorem 3 in \cite{Tar_1} 
applies and yields to the following alternatives
about the asymptotic behavior of $\xi_{k}$:
\begin{thmx}[\cite{Tar_1}]\label{thm_blow_up_global_from_part_1}
Let $\xi_{k}$ satisfy \eqref{6.23} and  assume
\eqref{6.24}-\eqref{6.27}. Then with the above notation
one of the following alternatives holds (along a subsequence):

\

\noindent
(i) \quad (compactness)\; : \; 
$\xi_{k}\longrightarrow \xi_{0}$ in $C^{2}(X)$ with
\begin{equation}%\label{thm_referenced_compactness}
-\Delta \xi_{0}
=
R_{0}e^{\xi_{0}}-f_{0}
,\;
\; \text{ in   } \; 
X
\end{equation}

\noindent
(ii)  \quad (blow-up)\; : \;There exists a \underline{finite} blow-up set 
$$
\mathcal{S}
=
\{ q\in X
\; : \; 
\; \exists \; 
q_{k}\longrightarrow q
\; \text{ and } \; 
\xi_{k}(q_{k})
\longrightarrow 
+ \infty,
\; \text{ as } \; 
k\longrightarrow +\infty
 \} 
$$

\quad \,
such that, 
$\xi_{k}$ 
is uniformly bounded from above on compact sets of $X\setminus \mathcal{S}$

\quad \;
and,
as $k\longrightarrow +\infty$, 
\begin{enumerate}[label=(\roman*)]
\item[(a)] either (blow-up with concentration)\;:\;
\begin{equation*}
\begin{split}
& 
\xi_{k}\longrightarrow -\infty
\; \text{ uniformly on compact sets of  } \;
X\setminus \mathcal{S},
\\
&
R_{k}e^{\xi_{k}}
\rightharpoonup
\sum_{q\in \mathcal{S}}\sigma(q)\delta_{q}
\; \text{ weakly in the sense of measures, with}
\; 
\sigma(q)\in 8\pi \N.\\
& 
\text{In particular,} \; \int_{X} f_{0}\,dA\in 8\pi \N, 
\; \text{ and } \; 
\\ & 
\sigma(q)=8\pi \; \text{ if } \; q\not \in Z 
\; \text{ and } \; 
\sigma(q) = 8\pi (1+n_{i}) \; \text{ if } \; q=  z_{i} \in Z\setminus Z_{0}. 
\end{split}
\end{equation*}  
Moreover, such an alternative always holds
when $\mathcal{S} \setminus Z_{0} \neq \emptyset$. 

\item[(b)] or (blow-up without concentration)\;:\; 
\begin{align}
&
\xi_{k}\longrightarrow \xi_{0}
\; \text{ in } \; C^{2}_{loc}(X\setminus \mathcal{S}),
\notag
\\
&
R_{k}e^{\xi_{k}}
\rightharpoonup
R_{0}e^{\xi_{0}}
+
\sum_{q\in \mathcal{S}}\sigma(q)\delta_{q}
\; \text{ weakly in the sense of measures, } \;
\notag
\\
&
\sigma(q)\in 8\pi \N; 
\notag
\end{align}	
and 
\begin{equation*}
\quad\quad\quad
-\Delta \xi_{0}
=
R_{0}e^{\xi_{0}}
+
\sum_{q\in \mathcal{S}}\sigma(q)\delta_{q}-f_{0}
\; \text{ in  } \; 
X.
\end{equation*}
Furthermore, if alternative (b) of (ii) holds, then
$
\mathcal{S} \subset Z_{0}
$
 and so, in this case,  blow-up occurs at points of "collapsing" zeroes. 
\end{enumerate} 
\end{thmx}

By virtue of Theorem \ref{thm_blow_up_global_from_part_1} we start to derive the following consequences: 
\begin{lemma}\label{lem_6.5}
If $\xi_{k}$ in \eqref{6.21} satisfies alternative (i),
then  $D_{0}$ is bounded from below in $\Lambda$ and
$(u_{t},\eta_{t})\longrightarrow (u_{0},\eta_{0})$,
as $t\longrightarrow 0^{+}$, in $\mathcal{V}_{p}$, $p>2$
(and in any other relevant norm)
with $(u_{0},\eta_{0})$ the only critical point of $D_{0}$ 
corresponding to its global minimum. 
In particular, 
$\rho([\beta])=4\pi(\mathfrak{g}-1)$ in this case. 
\end{lemma}	

\begin{proof}
Recall that we have set, $u_{t_{k}}=u_{k}=d_{k}+w_{k}$. 
By hypothesis, $\xi_{k}\leq C$ in $X$ and therefore, 
by elliptic estimates we derive that
$w_{k}$ is uniformly bounded and actually (along a subsequence) 
converges strongly in $C^{2,\alpha}(X)$.  
Thus we can estimate:  
\begin{equation*}
e^{d_{k}}
\int_{X} \Vert \beta_{0}+\bar{\partial}\eta_{t_{k}} \Vert^{2}dA
\leq 
C \int_{X} e^{u_{k}}
\Vert \beta_{0}+\bar{\partial}\eta_{t_{k}} \Vert^{2}dA
\leq C
\end{equation*}
and
$
\int_{X} \Vert \beta_{0}+ \bar{\partial}\eta_{t_{k}} \Vert^{2}dA
\geq 
\int_{X} \Vert \beta_{0} \Vert^{2}dA>0
$.
As a consequence, 
$d_{k}$ is uniformly bounded from above. On the other hand, since 
$
c_{t}=D_{t}(u_{t},\eta_{t})\leq c_{1}$, 
for all $t\in (0,1)$,
we obtain:
\begin{equation*}
\frac{1}{4}\int_{X} \vert \nabla w_{k} \vert^{2}
-
4\pi (\mathfrak{g}-1)d_{k}
=
c_{t_{k}}+O(1) \leq C,
\end{equation*}
showing that $d_{k}$ is also uniformly bounded from below. 
Hence, along a subsequence, we find that, 
$u_{k}\longrightarrow u_{0}$
strongly in
$C^{2,\alpha}(X)$, 
as
$k\longrightarrow +\infty$,
and, by virtue of Proposition \ref{prop_6.1}, we have:
\begin{equation*}
\eta_{t_{k}}
=
\eta(u_{k})
\longrightarrow 
\eta(u_{0})=\eta_{0}
\; \text{ in   } \; 
W^{1,p}(X,E),
\; \text{ as } \; 
k\longrightarrow +\infty,\; 
p>1. 
\end{equation*}
In conclusion,
$
D_{0}(u_{0},\eta_{0})
=
\lim_{k \to +\infty }
D_{t_{k}}(u_{t_{k}},\eta_{t_{k}})
=
\min_{\Lambda}D_{0}
$.

So $D_{0}$ is bounded from below and attains its minimum at $(u_{0},\eta_{0})$.
At this point the desired conclusion follows by
Theorem \ref{thm_primo} and Proposition \ref{prop_6.5}. 
\end{proof}
 
As an immediate consequence of Lemma \ref{lem_6.5} we have:

\begin{proposition}\label{proposition_lower_bound_on_rho}
For every  
$
[\beta] 
\in 
\mathcal{H}^{0,1}(X)\setminus \{  0 \}  $
we have:
$
\rho([\beta])\geq \frac{4\pi}{(\kappa-1)}
$.
\end{proposition}	
\begin{proof}
If by contradiction we assume that: 
$
\rho([\beta])
<
\frac{4\pi}{\kappa-1}
$, 
then:
\begin{equation*}
\begin{split}
\lim_{k\to +\infty}
\int_{X} R_{k}e^{\xi_{k}}
= &
8(\kappa-1)^{2}
\lim_{k\to +\infty}
\int_{X} e^{(\kappa-1)u_{t_{k}}}
\Vert \beta_{0}+\bar{\partial}\eta_{t_{k}} \Vert^{2}dA \\
= &
2(\kappa-1)\rho([\beta])<8\pi.
\end{split}
\end{equation*}
Therefore,  we  see that necessarily 
$\xi_{k}$ must satisfy the "compactness" alternative (i) 
in Theorem \ref{thm_blow_up_global_from_part_1}.  
Thus, by Lemma \ref{lem_6.5}, we find
$4\pi(\mathfrak{g}-1)=\rho([\beta])<\frac{4\pi}{\kappa-1}$, a contradiction. 
\end{proof}

\begin{remark}
By combining part (i) of Proposition \ref{proposition_on_rho} 
and Proposition \eqref{proposition_lower_bound_on_rho} we conclude that, 
\begin{equation}\label{least_possible_blow_up_energy_for_k_2_and_g_2}
\begin{split}
\; \text{ if } \; 
\kappa=2
\; \text{ and } \; 
\mathfrak{g}=2
\; \text{ then } \; 
\rho[\beta]=4\pi,
\; \forall \; 
[\beta] \in \mathcal{H}^{0,1}(X,E)\setminus \{  0 \}. 
\end{split}
\end{equation}
\end{remark}	
\noindent
So far we have established Theorem \ref{thm_1} and 
Proposition \ref{proposition_on_rho_intro}. 

To proceed further, in the following section we are going to analyze 
what happens when $\xi_{k}$ blows-up, 
in the sense of alternative (ii) of Theorem \ref{thm_blow_up_global_from_part_1}.

\subsection{Blow-up Analysis for Minimizers}\label{blow_up_minimizers}

In order to simplify technicalities,
from now on we shall focus to the case:
\begin{equation}\label{k_is_2_condition}
\kappa=2.
\end{equation}
As above, we write  $u_{t_{k}}=u_{k}=d_{k}+w_{k}$,  
$s_{k}=s_{t_{k}}$, $d_{k}=d_{t_{k}}$, and we consider the sequence
\begin{equation}\label{6.33}  
\xi_{k}=-(u_{k}-s_{k}),
\end{equation}
satisfying: 
\begin{equation}%\label{xi_k_equation}
-\Delta \xi_{k}
=
8\Vert \hat{\alpha}_{k} \Vert^{2}e^{\xi_{k}}
-
f_{k}
\; \text{ in   } \; 
X 
\end{equation}
(see \eqref{hat_alpha} and \eqref{6.21a})
with 
$\hat{\alpha}_{k}=e^{-\frac{s_{k}}{2}}\alpha_{k}$
and, as $k\longrightarrow +\infty$, 
\begin{equation}\label{f_k_convergence}
f_{k}
=
2(1-t_{k}e^{u_{k}})
\longrightarrow 
f_{0}
=
2(1-\mu e^{w_{0}})
\; \text{ in   } \; 
L^{p }(X),\;
p>1, 
\end{equation}
We suppose that $\xi_{k}$ blows up,
in the sense of (ii) in Theorem \ref{thm_blow_up_global_from_part_1}, 
and that   
\begin{equation}\label{the_blow_up_set_S}
\mathcal{S}=\{ q_{1},\ldots,q_{m} \} 
,\;
1\leq m \leq \mathfrak{g}-1
\end{equation}
is the  corresponding (non empty) blow-up set. 
By recalling that, $\Vert w_{k} \Vert_{L^{2}(X)}\leq C$ (see \eqref{6.14_prime})
and by using elliptic estimates, 
we easily derive that the sequence $w_{k}$  is uniformly bounded 
away from the blow-up set $\mathcal{S}$, and therefore, 
\begin{equation}\label{xi_k_versus_d_k_minus_s_k}
\xi_{k}
=
-(d_{k}-s_{k})+O(1)
\; \text{ on compact sets of  } \;
X\setminus \mathcal{S}.
\end{equation}
\begin{remark} 
Since
$$
c_{k}=D_{t_{k}}(u_{k},\eta_{k})
=
\frac{1}{4}\int_{X}\vert \nabla w_{k} \vert^{2}dA
-
4\pi(\mathfrak{g}-1)d_{k}+O(1),
$$
when blow-up occurs then, $d_{k}\longrightarrow +\infty$, as $k\longrightarrow +\infty$. 

From \eqref{xi_k_versus_d_k_minus_s_k} we see that, 
"blow-up with concentration" in Theorem \ref{thm_blow_up_global_from_part_1} 
occurs if and only if $d_{k}-s_{k}\longrightarrow +\infty$.
\end{remark}	

We start our investigation with the case where $\hat{\alpha}_{0}$ in \eqref{6.21a}
does not vanish on $\mathcal{S}$, namely: 
$
\hat{\alpha}_{0}(q_{l})\neq 0,
\; \forall \; l\in \{ 1,\ldots,m \}
$.
In this case we can use the blow-up analysis available in  
\cite{Brezis_Merle},\cite{Li_Shafrir},\cite{Chen_Lin_1}
to show Theorem \ref{thm_main_1_intro}  which for convenience we 
restate as follows:
\begin{thm}\label{main_thm_1}
Assume \eqref{k_is_2_condition}-\eqref{f_k_convergence}
and suppose that the blow-up set 
$\mathcal{S}\neq \emptyset$ of $\xi_{k}$ in \eqref{the_blow_up_set_S} satisfies:
\begin{equation}\label{S_intersected_Z_is_empty_condition}
\mathcal{S}\cap Z=\emptyset.
\end{equation}
Then (along a subsequence), as $k\longrightarrow +\infty$:
\begin{align}
&
\label{alpha_k_convergence}
\alpha_{k}\longrightarrow \alpha_{0} \in C_{2}(X)
\; \text{ with } \; 
\alpha_{0}\neq 0
\; \text{  vanishing exactly at  } \; 
Z,
\\ &
\label{weak_convergence_of_exp_minus_u_k}
e^{-u_{k}}
\rightharpoonup
\pi \sum_{q\in \mathcal{S}}\frac{1}{\Vert \alpha_{0} \Vert^{2}(q)}\delta_{q},
\; \text{  weakly in the sense of measures, } \; 
\\ &
\label{c_k_esitimate_} 
c_{k}=D_{t_{k}}(u_{k},\eta_{k})
=
-4\pi(\mathfrak{g}-1-m)d_{k}+O(1),
\text{ with } 
d_{k}=\fint_{X}u_{k}\,dA \longrightarrow +\infty,
\\ &
\label{beta_0_wedge_alpha_integral_is_zero}
\int_{X} \beta_{0} \wedge \alpha =0
\; \text{ for } \; 
\alpha \in Q_{2}[q_{1},\ldots,q_{m}]
\end{align}	
(recall \eqref{Q_kappa_definition}). 
Furthermore, 
$\rho([\beta])=4\int_{X} \beta_{0} \wedge \alpha_{0} = 4\pi m$. 
\end{thm}	 
\begin{remark}\label{rem_only_one_blow_up_point}
Since
$
\dim_{\mathbb{C}}Q_{2}[q_{1},\ldots,q_{m}]=3(\mathfrak{g}-1)-m
$
(see \eqref{dimension_of_Q_kappa}),
then the orthogonality condition \eqref{beta_0_wedge_alpha_integral_is_zero} 
together with the estimate \eqref{c_k_esitimate_} for the global minimum of $D_{t_{k}}$ 
seem to indicate that $\xi_{k}$ should admit only \underline{one} blow-up point ($m=1$),  
where the  holomorphic quadratic differential
$*_{E} \beta_{0}$
does not vanish.
\end{remark}

\begin{proof}
From \eqref{S_intersected_Z_is_empty_condition} we see that 
$\xi_{k}$ satisfies alternative (a) of (ii) 
in Theorem \ref{thm_blow_up_global_from_part_1}, i.e.
blow-up occurs with the "concentration" property, and furthermore 
$\Vert \hat{\alpha}_{0} \Vert(q)>0, \; \forall \; q \in \mathcal{S}$.  
More precisely, (along a subsequence) as $k\longrightarrow +\infty$, we have: 
\begin{align}
& 
8e^{\xi_{k}}\Vert \hat \alpha_{k} \Vert^{2}
\rightharpoonup
8\pi \sum_{l=1}^{m}\delta_{q_{l}}
,
\; \text{ weakly in the sense of measures, } \; 
\notag
%\label{6.38} 
\\
& 
w_{k}\longrightarrow w_{0}
\; \text{ in } \; 
C^{2,\alpha}(X\setminus S) 
,\; 
-\Delta w_{0}
=
8\pi \sum_{l=1}^{m}\delta_{q_{l}}
+
2\mu e^{w_{0}}-2
\; \text{ in } \; 
X. 
\notag
% \label{6.39}
\end{align}
Furthermore (by \eqref{xi_k_versus_d_k_minus_s_k}) we know that, 
$d_{k}-s_{k}\longrightarrow +\infty$, 
and from \eqref{S_intersected_Z_is_empty_condition}, 
we derive also that, 
\begin{equation}\label{6.41}
\int_{X}e^{\xi_{k}}\leq C.
\end{equation}
Hence
\begin{equation}\label{6.42}
e^{\xi_{k}}
\rightharpoonup
\pi
\sum_{l=1}^{m}
\frac{1}{\Vert \hat \alpha_{0} \Vert^{2}(q_{l})}
\delta_{q_{l}},
\; \text{ as } \; 
k\longrightarrow +\infty,
\end{equation}
 weakly in the sense of measures. 
\begin{equation}\label{6.43}
\; \text{ \textbf{\underline{Claim:} } } \;
\quad 
s_{k}=O(1).  
\quad\quad\quad\quad\quad\quad\quad
\quad\quad\quad\quad\quad\quad\quad
\end{equation}
To establish \eqref{6.43}, we observe that,
\begin{equation}\label{beta_0_squared_integral}
\begin{split}
\int_{X}\Vert \beta_{0} \Vert^{2}dA
= \; &
\int_{X}\langle \beta_{0},\beta_{0} + \bar{\partial}\eta_{k} \rangle dA
=
\int_{X}(e^{-u_{k}}\beta_{0})\wedge \alpha_{k} \\
\leq \; &     
\Vert \beta_{0} \Vert_{L^{\infty}}
\Vert \alpha_{k} \Vert_{L^{\infty}}
\int_{X} e^{-u_{k}}dA \\
\leq \; & Ce^{\frac{s_{k}}{2}}\int_{X}e^{-u_{k}}dA  
\leq 
Ce^{-\frac{s_{k}}{2}}\int_{X}e^{\xi_{k}}dA
\leq 
Ce^{-\frac{s_{k}}{2}},
\end{split}
\end{equation}
where in the last inequality we have used \eqref{6.41}. 
Since $\beta_{0}\neq 0$, then \eqref{beta_0_squared_integral} implies that,
$ s_{k} \leq  C $ 
for suitable $C>0$. In order to obtain a lower bound for 
$s_{k}$,
we take $\alpha \in C_{2}(X)$ and compute:
\begin{equation}\label{beta_0_wege_alpha_integral}
\begin{split}
\int_{X} \beta_{0}\wedge \alpha
= &
\int_{X} (\beta_{0} + \bar{\partial} \eta_{k}) \wedge \alpha 
= 
\int_{X} \langle \beta_{0}+\bar{\partial}\eta_{k},*_{E}^{-1}\alpha \rangle dA 
\\
= & 
e^{-\frac{s_{k}}{2}}
\int_{X} e^{\xi_{k}} \langle *_{E}^{-1}\hat \alpha_{k},*_{E}^{-1}\alpha \rangle dA .
\end{split}
\end{equation}
For $r>0$ sufficiently small, clearly we have
\begin{equation*}
\vert 
\int_{X \setminus \cup_{l=1}^{m}B_{r}(q_{l})}
e^{\xi_{k}}\langle *_{E}^{-1}\hat \alpha_{k},*_{E}^{-1}\alpha \rangle dA
\vert 
\leq C \int_{X \setminus \cup_{l=1}^{m}B_{r}(q_{l})}
e^{\xi_{k}}dA
\longrightarrow 
0,
\end{equation*}
as $k\longrightarrow +\infty$.
Next, for $l\in \{ 1,\ldots,m \}$, around $q_{l}$ we introduce holomorphic coordinates 
$\{ z \}$, centered at the origin, so that:
\begin{equation*}
\hat \alpha_{k}
=
h_{k}^{(l)}(dz)^{2}
\; \text{ and } \; 
\alpha =\varphi^{(l)}(dz)^{2}
\; \text{ in } \; 
B_{r},
\end{equation*}
where $h_{k}^{(l)}$ and $\varphi^{(l)}$ are holomorphic in $B_{r}$. 
Furthermore,  
$h_{k}^{(l)}\longrightarrow h^{(l)}$ uniformly in $B_{r}(0)$
and  
$\hat \alpha_{0}=h^{(l)}(dz)^{2}$. 
Since 
$\Vert \hat{\alpha}_{0} \Vert(q_{l})>0$, then (for $r>0$ sufficiently small)
we have that $\hat{\alpha}_{0}$ never vanishes in a neighbourhood of $q_{l}$,
that is the holomorphic function:
$h^{(l)}(z)\neq 0$ for all $z\in B_{r}$.
Therefore, in view of \eqref{6.42}, for
$r>0$ sufficiently small, we find, as $k\longrightarrow +\infty$, 
\begin{equation}\label{H_l_varphi_k}
\int_{B_{r}(q_{l}) } e^{\xi_{k}}
\langle *_{E}^{-1} \hat \alpha_{k}, *_{E}^{-1} \alpha\rangle dA
\longrightarrow 
\pi 
\frac{h^{(l)}(0)}{\vert h^{(l)}(0) \vert^{2}}
\bar\varphi^{(l)}(0)
:=
H_{l}\bar{\varphi}_{l} \in \mathbb{C}
,
\end{equation}
with 
$H_{l} \neq 0 $ and $\vert \varphi_{l} \vert=\Vert \alpha \Vert(q_{l})$,
$l=1,\ldots, m$.

\

By Remark \ref{rem_vanishing_at_all_but_one}, we know that there exists a quadratic holomorphic differential 
$\alpha \in C_{\kappa}(X)$, which vanishes at all but one point in $\mathcal{S}$. 
For example, we may choose 
\begin{equation*}
\alpha \in C_{2}(X)
\; : \; 
\Vert \alpha \Vert(q_{1})=1
\; \text{ while } \; 
\Vert \alpha \Vert(q)=0
,
\; \forall \; q\in \mathcal{S}\setminus \{  q_{1} \}, 
\end{equation*}
and by this choice, according to the estimates above, we find:
\begin{equation}\label{6.44}
\int_{X}\beta_{0}\wedge \alpha
=
e^{-\frac{s_{k}}{2}}(H_{1}+o(1))
,
\; \text{ as } \; 
k\longrightarrow +\infty,
\end{equation}
with $H_{1}\neq 0$. 
Since we have already shown that $s_{k}$ is uniformly bounded from above, 
then \eqref{6.44} implies that,
\begin{equation*}
\vert \int_{X}\beta_{0}\wedge \alpha \vert 
\geq 
L
>
0,
\end{equation*}
and in turn we deduce that,  
$s_{k}\geq -c$
and \eqref{6.43} is established. 

\

As a consequence of Claim \eqref{6.43}, we have that,  
$C^{-1}\leq \Vert \alpha_{k} \Vert_{L^{2}}\leq C$
for suitable $C>0$, and then (along a subsequence):
$\alpha_{k}\longrightarrow \alpha_{0} \in C_{2}(X)$, 
as  $k\longrightarrow +\infty$, 
(in any relevant norm), 
and $\alpha_{0}\neq 0$. 
So, $\alpha_{0}$ vanishes \underline{exactly} at
$Z$ and moreover, by \eqref{S_intersected_Z_is_empty_condition}, 
$
\Vert \alpha_{0} \Vert(q_{l})
\neq 0
$, 
$\; \forall \; l=1,\ldots,m$.
So we can reformulate \eqref{6.42} as follows:
\begin{equation*}
e^{-u_{k}}
\rightharpoonup
\pi \sum_{l=1}^{m}
(\frac{1}{\Vert \alpha_{0} \Vert^{2}(q_{l})})\delta_{q_{l}}
,
\; \text{ as } \; 
k\longrightarrow +\infty
,
\; \text{ weakly in the sense of measures,} 	
\end{equation*} 
and \eqref{alpha_k_convergence} and \eqref{weak_convergence_of_exp_minus_u_k} are established. 
Moreover, for $r>0$ sufficiently small, we can use the first identity in
\eqref{beta_0_wege_alpha_integral} and argue as above to obtain
\begin{equation*}
\begin{split}
\int_{X} \beta_{0} \wedge \alpha_{0}
= \; &
\lim_{k\to \infty}
(
\sum_{l=1}^{m}
\int_{B_{r}(q_{l})} e^{-u_{k}}
\langle *_{E}^{-1}\alpha_{k},*_{E}^{-1}\alpha_{0}\rangle dA
) \\
= \; &
\lim_{k\to \infty} \int_{X}e^{-u_{k}}\Vert \alpha_{k} \Vert^{2}dA
=\pi m.
\end{split}
\end{equation*}
Since 
$\int_{X}e^{-u_{k}}\Vert \alpha_{k} \Vert^{2}dA
=
\int_{X} e^{u_{k}}\Vert \beta_{0}+\bar{\partial}\eta_{0} \Vert^{2}dA
$, 
we obtain in particular, that
$\rho([\beta])=4\pi m$. 

Similarly  if $\alpha\in C_{2}(X)$ vanishes 
at each point of $\mathcal{S}$, then from 
\eqref{beta_0_wege_alpha_integral} and \eqref{H_l_varphi_k} we have:
\begin{equation*}
\int_{X}
e^{-u_{k}}
\langle *_{E}^{-1}\alpha_{k},*_{E}^{-1}\alpha \rangle dA
\longrightarrow 0
,
\; \text{ as } \; 
k\longrightarrow +\infty,
\end{equation*}
and so, 
$\int_{X}\beta_{0} \wedge \alpha  =0$ as claimed in \eqref{beta_0_wedge_alpha_integral_is_zero}. 

Finally, under the assumption \eqref{S_intersected_Z_is_empty_condition}, we can
use well known \textit{sup+inf} estimates and gradient estimates for $\xi_{k}$ around 
$q_{l}\in \mathcal{S}$ (see \cite{Chen_Lin_1}) and obtain
(for $r>0$ sufficiently small)
\begin{equation*}
\lambda_{k}:=\max_{\bar B _{r}(q_{l})}\xi_{k}
=
-\min_{\partial \bar B _{r}(q_{l})}\xi_{k}
+
O(1)
,
\; \forall \; l=1,\ldots,m
\end{equation*}
 
\begin{equation*}
\int_{B_{r}(q_{l})} \vert \nabla w_{k} \vert^{2}
=
\int_{B_{r}(q_{l})}\vert \nabla \xi_{k} \vert^{2}
=
16 \pi \lambda_{k}+O(1).
\end{equation*}
But, from \eqref{xi_k_versus_d_k_minus_s_k} and \eqref{6.43},  we see that
$\lambda_{k}=d_{k}+O(1)$, and we deduce that
\begin{equation*}
\int_{X} \vert \nabla w_{k} \vert^{2}dA 
=
(16 \pi m)d_{k}+O(1). 
\end{equation*}
Consequently,
\begin{equation*}
\begin{split}
c_{k}
= &
D_{t_{k}}(u_{k},\eta_{k})
=
\int_{X}
(
\frac{\vert \nabla w_{k} \vert^{2}}{4}
-
u_{k}
+
t_{k}e^{u_{k}}
+
4e^{u_{k}}
\Vert \beta_{0}+\bar{\partial}\eta_{k} \Vert^{2}
)
dA \\
= &
\int_{X}
(
\frac{\vert \nabla w_{k} \vert^{2}}{4}
-
4\pi(\mathfrak{g}-1)d_{k}
+
O(1)
=
-4\pi(\mathfrak{g}-1-m)d_{k}+O(1)
\end{split}
\end{equation*}
and also \eqref{c_k_esitimate_} is established. 
\end{proof}	

\begin{remark}
Under the assumption \eqref{S_intersected_Z_is_empty_condition}, 
the estimate \eqref{c_k_esitimate_} allows us to give 
a positive answer to \eqref{Question_2_i} posed above. 
Indeed if 
$\rho([\beta])=4\pi(\mathfrak{g}-1)$, 
then  necessarily   
$m=\mathfrak{g}-1$ 
and according to \eqref{c_k_esitimate_} we deduce that 
$D_{0}$ is bounded from below in 
$\Lambda$. 

\end{remark}	
When \eqref{S_intersected_Z_is_empty_condition} holds
then Theorem \ref{main_thm_1} gives a reasonable description about 
the blow-up behaviour of minimizers for the Donaldson functional, 
as $t\longrightarrow 0^{+}$.
Although we cannot yet guarantee that \eqref{S_intersected_Z_is_empty_condition}
always holds, 
by the minimizing property of the sequence $(u_{k},\eta_{k})$ 
for $D_{t}$ we expect,  
that blow-up for $(u_{k},\eta_{k})$ should occur
with the least possible blow-up mass $8\pi$. 
So, next, we shall focus to this situation and 
assume that:
\begin{equation}\label{6.45}
\lim_{r \to 0^{+} }
\lim_{k \to +\infty }
8 \int_{B_{r}(q)} e^{u_{k}}
\Vert \beta_{0} + \bar{\partial}\eta_{k} \Vert^{2}dA
=
8\pi
,
\; \forall \; 
q\in S. 
\end{equation}
Indeed, \eqref{6.45} is always ensured 
under the assumption \eqref{S_intersected_Z_is_empty_condition}
or when $\mathfrak{g}=2$, see \eqref{least_possible_blow_up_energy_for_k_2_and_g_2}.
On the other hand, when $q\in \mathcal{S} \cap Z$ and we assume \eqref{6.45}, 
then (by Theorem \ref{thm_blow_up_global_from_part_1}), necessarily $q\in Z_{0}$, 
that is blow-up must occurs at a point of "collapsing" zeroes of 
$\hat{\alpha}_{k}$. In other words, if \eqref{6.45} holds then  
$\mathcal{S}\cap Z=\mathcal{S}\cap Z_{0}$. 
Therefore,

\begin{corollary}\label{cor_6.8}
Assume \eqref{6.45} and that
$\mathcal{S}\cap Z_{0}=\emptyset$. 
Then the conclusion of Theorem \ref{main_thm_1} holds.
\end{corollary}	
To deal with the case where 
$\mathcal{S}\cap Z_{0}\neq \emptyset$, 
we introduce the following notation.
For $q_{l}\in \mathcal{S}$ and $r>0$ sufficiently small, we let
\begin{equation}\label{6.46}
x_{k,l}\in B_{r}(q_{l})
\; : \; 
\xi_{k}(x_{k,l})=\max_{B_{r}(q_{l})}\xi_{k}
\longrightarrow 
+\infty
\; \text{ and } \; 
x_{k,l}\longrightarrow q_{l}, 
\end{equation} 
as $k\longrightarrow +\infty$, and set
\begin{equation}%\label{mu_k_l_definition}
\mu_{k,l}=\Vert \alpha_{k} \Vert^{2}(x_{k,l}),
\end{equation}
for $l=1,\ldots,m$. Since "locally" in holomorphic $z$-coordinates $\Vert \alpha_{k} \Vert^{2}(z)$ coincides (essentially) with the norm of a holomorphic function, 
we see that 
\begin{equation}%\label{alpha_0_vanishing_implication}
\; \text{ if  } \; 
\Vert \hat \alpha _{0} \Vert(q_{l})\neq 0
\; \text{ then  } \; 
\Vert \alpha_{k} \Vert^{2}(x_{k,l})
=
O(e^{s_{k}}). 
\end{equation}
We have:
\begin{thm}\label{main_thm_2}
Assume \eqref{6.45} and suppose that the blow-up set
$\mathcal{S}$ of $\xi_{k}$ in \eqref{the_blow_up_set_S} satisfies:
$
\mathcal{S} \cap  Z_{0}
\neq  
\emptyset
$. 
Then (along a subsequence):
%there holds 
%\begin{equation}\label{6.48}
$$s_{k}\longrightarrow +\infty
,
\; \text{ as } \;
k\longrightarrow +\infty.$$
%\end{equation}
Moreover,  there exists a set of indices
$J\subseteq \{ 1,\ldots,m \} $ such that,
as $k\longrightarrow +\infty$,
\begin{enumerate}[label=(\roman*)]
\item 	
$\; \forall \; l\in J
\; : \;  
q_{l}\in \mathcal{S} \cap Z_{0}
\; \text{ and } \; 
\mu_{k,l}\longrightarrow \mu_{l}>0
$:
\begin{equation*}
\quad\quad\;\, 
e^{-u_{k}}
\rightharpoonup
\pi \sum_{l\in J}\frac{1}{\mu_{l}}\delta_{q_{l}}
\; \text{ weakly in the sense of measures, } \; 
\end{equation*}
\item

\quad
\vspace{-26pt}
\begin{equation}\label{orthonality_on_S_0}
\begin{split}
&
\int_{X} \beta_{0} \wedge \alpha = 0, 
\; \forall \;   \alpha \in C_{2}(X) 
\; \text{ vanishing at any point of} \; 
\\
&  
\mathcal{S}_{0}=\{ q_{l}\in \mathcal{S} \; : \; l\in J \}\subset Z_{0}.
\; \text{ In particular, } \;   
\int_{X} \beta_{0} \wedge \hat \alpha _{0} = 0.
\end{split}
\end{equation}

\item
$ \mu_{k,l}\longrightarrow +\infty$, 
as 
$ k\longrightarrow +\infty$, 
$\; \forall \; l\in \{ 1,\ldots,m \}\setminus J$ (if not empty)
\begin{equation}\label{c_k_asymptotic}
c_{k}
=
D_{t_{k}}(u_{k},\eta_{k})
=
-4\pi 
\left(
(\mathfrak{g}-1-m)d_{k}
+
\sum_{l  \in \{ 1,\ldots, m\} \setminus  J}\log(\mu_{k,l})
\right)
+
O(1).
\end{equation}
with $d_{k}=\fint_{X}u_{k}\longrightarrow +\infty$. 
\end{enumerate}
\end{thm}	
\begin{proof}
It is understood that the summation in \eqref{c_k_asymptotic} is dropped
when $J=\{ 1,\ldots,m \}$. 
As above, for given $\alpha \in C_{2}(X)$ and $r>0$ small, 
from \eqref{xi_k_versus_d_k_minus_s_k} we have:
\begin{equation*}%\label{6.53_prime}
\begin{split}
\int_{X} \beta_{0} \wedge \alpha 
= \;&
\int_{X} e^{-u_{t}}
\langle *_{E}^{-1} \alpha_{t},*_{E}^{-1}\alpha \rangle dA 
= 
e^{-\frac{s_{k}}{2}}
(
\int_{X}e^{\xi_{k}}\langle *_{E}^{-1}\hat{\alpha}_{t},*_{E}^{-1}\alpha \rangle dA
)
\\
= \;&
e^{-\frac{s_{k}}{2}}
(
\sum_{l=1}^{m} \int_{B_{r}(x_{k,l})}e^{\xi_{k}}
\langle *_{E}^{-1}\hat \alpha _{t} , *_{E}^{-1}\alpha \rangle dA
+
O(e^{-d_{k}+s_{k}}) 
)
\end{split}
\end{equation*}
with $x_{k,l}$ as given in \eqref{6.46}.
Since $s_{k}\leq d_{k} + O(1)$, see \eqref{6.17_prime}, then

\noindent
$-d_{k}+\frac{s_{k}}{2}\longrightarrow -\infty$, and we may
still conclude (as before) that,
\begin{equation}\label{wedge_product_asymptotic}
\int_{X}\beta_{0} \wedge \alpha
=
e^{-\frac{s_{k}}{2}}
(
\sum_{l=1}^{m} \int_{B_{r}(x_{k,l})}
e^{\xi_{k}} \langle *_{E}^{-1}\hat{\alpha}_{t},*_{E}^{-1}\alpha\rangle dA
)
+
o(1)
\end{equation}

In case $x_{k,l} \longrightarrow  q_{l} \in \mathcal{S}\setminus Z_{0}$, 
as $k\longrightarrow +\infty$, 
then $q_{l} \not \in Z$ (as $\mathcal{S}\cap Z_{0}=\mathcal{S}\cap Z$) 

and so,
$$
0<a<\Vert \alpha_{k} \Vert(x_{k,l}) \leq b 
\; \text{ in  } \;
B_{r}(x_{k,l}) 
$$

So we can argue exactly as in \eqref{H_l_varphi_k} to obtain:
\begin{equation}\label{some_integral_..}
e^{-\frac{s_{k}}{2}}\int_{B_{r}(x_{k,l})}e^{\xi_{k}}
\langle *_{E}^{-1} \hat \alpha _{k} , *_{E}^{-1} \alpha \rangle dA
=
\frac{1}{\Vert \alpha_{k} \Vert(x_{k,l})}
(
H_{l} \bar \varphi _{l} + o(1))
, 
\end{equation}

as $k\longrightarrow +\infty$, 
for suitable $H_{l},\varphi_{l} \in \mathbb{C}$, with $H_{l}\neq 0$ and 
$ 
\vert  \varphi _{l} \vert 
= 
\Vert \alpha \Vert(q_{l}). 
$
\\
\noindent
Hence suppose that, $x_{k,l}\longrightarrow q_{l}\in \mathcal{S} \cap Z_{0}$, 
as $k\longrightarrow +\infty$. 
In this situation we are going to apply the blow-up analysis developed in 
\cite{Tar_1}. 
For this purpose it is convenient to introduce local holomorphic $z$-coordinates
around $x_{k,l}$ centered at the origin. 
So, for $r>0$ sufficiently small, we can always assume that,

\begin{align}
& 
\hat \alpha _{k}
= 
\left(\Pi_{j=1}^{s}
(z-p_{j,k})^{n_{j}}
\right)
\psi_{k}(z)(dz)^{2}
\; \text{ in } \; B_{r}
,\; s\geq 2,\; n_{j}\in \N; 
\notag
%\label{6.54}
\\
&  
0
\leq 
\vert p_{1,k} \vert 
\leq
\vert p_{2,k} \vert 
\leq \ldots \leq
\vert p_{s,k} \vert \longrightarrow 0
,
\; \text{ as } \; 
k\longrightarrow \infty,
\;
p_{i,k}\neq p_{j,k},\;
i\neq j; 
\label{v_4} \\
& 
\hat{\alpha}_{0}(z)=z^{n}\psi(z)(dz)^{2} 
\; \text{ in } \;
B_{r},\;
n=\sum_{j=1}^{s}n_{j};
\notag
\\
&
\psi_{k} \; \text{ and } \; \psi
\; \text{ holomorphic and \underline{never vanishing} in  } \; B_{r},\;
\notag 
%\label{psi_k_and_psi_never_vanish}
\\
&
\psi_{k}\longrightarrow \psi 
\; \text{ uniformly in } \; B_{r}, \; \text{ as } \; k\longrightarrow \infty;
\label{psi_k_towards_psi}
\\
& 
\alpha=\varphi (dz)^{2} \; \text{ in  } \; B_{r}
\; \text{ and } \; 
\varphi \; \text{ is holomorphic in } \; B_{r}. 
\notag
\end{align}

Thus, if (by abusing notation) we identify a function with its local expression 
in local $z$-coordinates, we see that 
$\xi_{k}=\xi_{k}(z)$ satisfies:
\begin{DispWithArrows}<>
& -\Delta \xi_{k}(z)
= 
(\Pi_{j=1}^{s}\vert z-p_{j,k} \vert^{2 n_{j}})
h_{k}(z)e^{\xi_{k}(z)}
+
g_{k}(z)
\; \text{ in } \; 
B_{r}
\label{v_1}
\\
& \max_{\partial B_{r}(0)} \xi_{k} -  \min_{\partial B_{r}} \xi_{k}\leq  C 
\label{v_2}
\\
& \xi_{k}(0)=  \max_{B_{r}(0)}\xi_{k} \longrightarrow \infty, \; \text{ as } \;
k\longrightarrow \infty,
\notag
\label{v_3} 
\end{DispWithArrows}
where 
\begin{equation*}%\label{h_k_convergence_part_2}
\begin{split}
&
h_{k}(z)=8\vert \psi_{k} \vert^{2}(z) \lambda^{-1}(z) 
\longrightarrow h(z):=8\vert \psi \vert^{2}(z) \lambda^{-1}(z)
\; \text{ uniformly in } \; B_{r};
\\
&
g_{k}(z)
= 
-f_{k}(z)\lambda(z)
\; \text{ is convergent in  } \; 
L^{p}(B_{r}), \; p>1,
\end{split}
\end{equation*}
with
$\lambda(z)=e^{2u_{0}(z)}>0$ the conformal factor of the metric $g_{X}$
in the $z$-coordinates, taken with the normalization $\lambda(0)=1$. 

We recall that \eqref{v_2} is by now a well known consequence of 
the Green representation formula for the function $\xi_{k}$ on $X$
(see \cite{Bartolucci_Chen_Lin_Tarantello}). 
Also it follows from \eqref{6.17} that,
\begin{equation*}%\label{W_k_definition}
\int_{B_{r}} 
W_{k}
e^{\xi_{k}}
\leq C
,\quad
W_{k}(z)
: =
(\Pi_{j=1}^{s}\vert z-p_{j,k} \vert^{2 n_{j}})
h_{k}(z). 
\end{equation*}

By taking $r>0$ smaller if necessary, we can assume in addition that the origin 
it the only blow-up point for $\xi_{k}$ in $B_{r}$, namely
\begin{equation}\label{unique_blow_up_in_zero}
\; \forall \; \delta \in (0,r) \; \exists \; C_{\delta}>0
\; : \; 
\max_{\bar{B}_{r}\setminus B_{\delta}}\xi_{k} \leq C_{\delta}.
\end{equation}

Finally, around a point $q\in \mathcal{S} \cap Z_{0}$, 
the assumption \eqref{6.45} can be stated as follows:  
\begin{equation}\label{3.0_part_2}
\mathfrak{m}:=
\lim_{\delta \searrow 0  }
\lim_{k \to +\infty }
\frac{1}{2\pi}\int_{B_{\delta}}W_{k}e^{\xi_{k}} =4. 
\end{equation}
At this point the blow-up analysis developed in \cite{Tar_1} 
can be applied and
we deduce the following: 
\begin{thmx}[\cite{Tar_1}]\label{theorem_from_part_1}
Let $\xi_{k}$ be a solution of \eqref{v_1} and assume  
\eqref{v_4}-\eqref{psi_k_towards_psi} and 
\eqref{v_2}-\eqref{unique_blow_up_in_zero}. 
If \eqref{3.0_part_2} holds then
$p_{j,k}\neq 0,\; j=1,\ldots,s$ and 
there exists 
$s_{1}\in \{ 2,\ldots,s \} $ such that (along a subsequence)
\begin{equation*} %\label{z_j_k_to_z_j}
z_{j,k}:=\frac{p_{j,k}}{\vert p_{s_{1},k} \vert }
\longrightarrow 
z_{j}\neq 0, 
\; \forall \;
j=1,\ldots,s_{1}, 
\end{equation*}
\begin{equation*}%\label{convergence_for_s_1_smaller_than_s}
\; \text{ if } \; 
s_{1}<s 
\; \text{ then } \; 
\frac{p_{j,k}}{\vert p_{s,k} \vert }\longrightarrow q_{j} \neq 0
\; \text{ and } \; 
\frac
{\vert p_{j,k} \vert }
{\vert p_{s_{1},k} \vert }
\longrightarrow 
+\infty,
\; \forall \; 
 j=s_{1}+1,\ldots,s. 
\end{equation*}
Moreover, 
\begin{align}
& \xi_{k}(0) \; 
+ 
2\ln \vert p_{s_{1},k}\vert + \ln (W_{k}(0)) \longrightarrow +\infty, 
\; \text{ as } \; k\longrightarrow +\infty,  \notag
\\
& \xi_{k}(0) 
=  
-
\left( \min_{\partial B_{r}} \xi_{k} + 2\ln (W_{k}(0))\right) + O(1)
\label{xi_k_0}
\\
&
\xi_{k}(x)
=  
\ln  
\left(
\frac
{e^{\xi_{k}(0)}}
{(1+ \frac{W_{k}(0)}{8}e^{\xi_{k}(0)}\vert x \vert^{2})^{2} }
\right)
+
O(1) 
%\label{bubble_shape}
\\
&
\int_{B_{r}}\vert \nabla  \xi_{k} \vert^{2}dx
=  
16\pi 
\left(
\xi_{k}(0)+\ln (W_{k}(0))
\right)
+ O(1).
\label{gradient_estimate}
\end{align}
\end{thmx}
\begin{proof}
See Theorem 5 in \cite{Tar_1}.  
\end{proof}

\begin{remark}%\label{rem_blow_up_comparablility}
Incidentally, for the original sequence $u_{k}$ (see \eqref{6.33}), 
from Theorem \ref{theorem_from_part_1}  we obtain that,
\begin{equation}\label{formula_remark}
-u_{k}(z)
=
\ln
\frac
{e^{-u_{k}(0)}}
{(1+\Vert \alpha_{k} \Vert^{2}(x_{k,l})e^{-u_{k}(0)}\vert z \vert^{2} )^{2}}
+
O(1)
\; \text{ in } \; 
B_{r}.
\end{equation}
In particular, such an estimate allows us to compare the blow-up rates at different blow-up points as follows: 
\begin{equation*}
\frac{1}{C}
\leq 
\frac
{e^{-u_{k}(x_{k,h})}\Vert \alpha_{k} \Vert^{4}(x_{k,h})}
{e^{-u_{k}(x_{k,l})}\Vert \alpha_{k} \Vert^{4}(x_{k,l})}
\leq C
\; \text{ for all } \;
h,l=1,\ldots,m 
\end{equation*}
for suitable $C\geq 1$. 
\end{remark}

\

\noindent
At this point we may continue the proof of Theorem \ref{main_thm_2}
and we can  take advantage of the information about $\xi_{k}$ 
provided by Theorem \ref{theorem_from_part_1} to estimate the integral:
\begin{equation*}
\begin{split}
\int_{B_{r}(x_{k,l})}e^{\xi_{k}} 
&
\langle *_{E}^{-1} \hat \alpha _{k} , *_{E}^{-1} \alpha \rangle dA \\
= \; &
\int_{B_{r}}e^{\xi_{k}(z)}
(\Pi_{j=1}^{s}(z-p_{j,k} )^{n_{j}})
\psi_{k}(z)\bar{\varphi}(z) \lambda^{-1}(z)dz d\bar{z}.
\end{split}
\end{equation*}
To this purpose we see that, 
as $k\longrightarrow \infty$,
\begin{equation}\label{epsilon_s_k_definition}
\begin{split}
&
\varepsilon_{k}:=\vert p_{s_{1},k} \vert \longrightarrow 0
,\;
\lambda_{k}^{2}:=e^{\xi_{k}(0)+2\ln \varepsilon_{k}+\ln(W_{k}(0))}
\longrightarrow \infty \\
& 
\varepsilon_{j,k}:=
\frac{\varepsilon_{k}}{\vert p_{j,k} \vert }
\longrightarrow L_{j}\geq 0 
\; \text{ for all } \; 
j=1,\ldots,s_{1}.
\end{split}	
\end{equation}
Hence 
$
t_{k}^{2}=
\frac{W_{k}(0)}{8}e^{\xi_{k}(0)}
=
\frac{1}{8}(\frac{\lambda_{k}}{\varepsilon_{k}})^{2}
\longrightarrow \infty$,
as
$k\longrightarrow \infty$,
with 
\begin{equation}\label{W_k_0_property}
\begin{split}
&
\frac{W_{k}(0)}{8}
= 
(\Pi_{j=1}^{s}\vert p_{j,k} \vert^{2n_{j}})\vert \psi_{k}(0) \vert^{2} 
=
\Vert \hat{\alpha}_{k} \Vert^{2}(x_{k,l})
=
e^{-s_{k}}\Vert \alpha_{k} \Vert^{2}(x_{k,l})
>0. 
\end{split}
\end{equation}
Therefore, for $R>0$ large, we have 
\begin{equation*}
\begin{split}
\vert 
\int_{\{ \frac{R}{t_{k}} \leq \vert z \vert \leq r  \} } 
&
(\Pi_{j=1}^{s}( z-p_{j,k} )^{n_{j}})
\psi_{k}(z)\bar{\varphi}(z) \lambda^{-1}(z) \frac{i}{2}dz \wedge d\bar{z}
\vert
\\
\leq \; &
C
\left(
\int_{\{ \frac{R}{t_{k}} \leq \vert z \vert \leq r  \} } 
e^{\xi_{k}(z)}
\right)^{\frac{1}{2}}
\left(
\int_{\{ \frac{R}{t_{k}} \leq \vert z \vert \leq r  \} } 
e^{\xi_{k}(z)}W_{k}(z)
\right)^{\frac{1}{2}}
\\
\leq \; &
C
\left(
\int_{\{ \frac{R}{t_{k}} \leq \vert z \vert \leq r  \} } 
e^{\xi_{k}(z)}
\right)^{\frac{1}{2}}
\leq 
\frac{C}{(W_{k}(0))^{\frac{1}{2}} R}
\leq
\frac{1}{\Vert \hat{\alpha}_{k} \Vert(x_{k,j})}\frac{C}{R}.
\end{split}
\end{equation*}
While, for 
$\hat p _{j,k}
:=
\frac{p_{j,k}}{\vert p_{j,k} \vert }
\longrightarrow 
\hat p _{j}$ 
with
$\vert \hat{p_{j}} \vert=1 $, $j=1,\ldots,s$; 
in view of 
\eqref{epsilon_s_k_definition}-\eqref{W_k_0_property} 
we have:
\begin{equation*}
\begin{split}
&
\int_{\{ \vert z \vert \leq \frac{R}{t_{k}}  \} }  
e^{\xi_{k}(z)}
(\Pi_{j=1}^{s}( z-p_{j,k} )^{n_{j}})
\psi_{k}(z)\bar{\varphi}(z) \lambda^{-1}(z)\frac{i}{2}dz \wedge d\bar{z} \\
= \; &
\int_{\{ \vert z \vert \leq \frac{R}{t_{k}}  \}}
\frac
{e^{\xi_{k}(0)}(\Pi_{j=1}^{s}( z-p_{j,k} )^{n_{j}})}
{(1+\frac{W_{k}(0)}{8})e^{\xi_{k}(0)}\vert z \vert^{2})^{2} }
\psi_{k}(z)\bar{\varphi}(z) \lambda^{-1}(z)\frac{i}{2}dz \wedge d\bar{z}
+
o(1) \\
= \; &
\frac{8}{W_{k}(0)}
\int_{\{ \vert w \vert \leq R  \} }
\frac
{(\Pi_{j=1}^{s}( \frac{w}{t_{k}}-p_{j,k} )^{n_{j}})
\psi_{k}(\frac{w}{t_{k}})\bar{\varphi}(\frac{w}{t_{k}}) \lambda^{-1}(\frac{w}{t_{k}})}
{(1+\vert w \vert^{2})^{2}}
\frac{i}{2}dw \wedge d\bar{w}
+
o(1) 
\\
= \; &
\frac{1}{\Vert \hat{\alpha}_{k} \Vert(x_{k,l})}
\underset{\{ \vert w \vert \leq R  \} }{\int}
\hspace{-10pt}
\frac
{(\Pi_{j=1}^{s}( \varepsilon_{j,k}\frac{w}{\lambda_{k}}-\hat{p}_{j,k})^{n_{j}})
\psi_{k}(\frac{w}{t_{k}})\bar{\varphi}(\frac{w}{t_{k}}) \lambda^{-1}(\frac{w}{t_{k}})}
{\vert \psi_{k}(0) \vert (1+\vert w \vert^{2})^{2}}
\frac{i}{2}dw \wedge d\bar{w}
+
o(1)
\\
= \; &
\frac{1}{\Vert \hat{\alpha}_{k} \Vert(x_{k,l})}
\left[
\pi \, 
\Pi_{j=1}^{s}(-\hat{p}_{j})^{n_{j}}
\frac{\psi(0)}{\vert \psi(0) \vert }\bar{\varphi}(0)
+
O(\frac{1}{R^{2}})
+
o(1)
\right]
, 
\end{split}
\end{equation*}
where $o(1)\longrightarrow 0$ uniformly in $R$, as $k\longrightarrow \infty$. 
So, by letting $R\longrightarrow +\infty$, and combining the above estimates,
we conclude that the analogous of \eqref{some_integral_..} holds also when
$q_{l} \in \mathcal{S} \cap Z_{0}$. In other words, in terms of $u_{k}$, 
we have established that,
\begin{equation*}
\begin{split}
\int_{B_{r}(x_{k,l})}
e^{-u_{k}}
\langle *_{E}^{-1}\alpha_{k},*_{E}^{-1}\alpha \rangle dA
= \; &
e^{-\frac{s_{k}}{2}}\int_{B_{r}(x_{k,l})}e^{\xi_{k}} 
\langle *_{E}^{-1} \hat \alpha _{k} , *_{E}^{-1} \alpha \rangle dA  \\
= \; &
\frac{1}{\Vert \alpha_{k} \Vert(x_{k,l})}
(H_{l}\bar{\varphi}_{l}+o(1))
,
\; \forall \; l=1,\ldots,m
\end{split}
\end{equation*}
for suitable $H_{l},\varphi_{l} \in \mathbb{C}$, $H_{l}\neq 0$ and 
$ 
\vert  \varphi _{l} \vert 
= 
\Vert \alpha \Vert(q_{l})
$.
Consequently, from \eqref{wedge_product_asymptotic} we conclude that,  
\begin{equation}\label{6.55a}
\begin{split}
\int_{X} \beta_{0} \wedge \alpha 
= &
\sum_{l=1}^{m}
\frac{1}{\Vert \alpha_{k} \Vert(x_{k,l})}
(H_{l}\bar \varphi _{l}+o(1))
+
o(1),
\; \text{ as } \;  
k\longrightarrow +\infty.
\end{split}
\end{equation}
In particular, as above, by taking $q_{l}\in \mathcal{S} \cap Z_{0}$ 
(so $\hat{\alpha}_{0}(q_{l})=0$)
and a holomorphic quadratic differential $\alpha \in C_{2}(X)$ 
such that $\Vert \alpha \Vert(q_{l})\neq 0$, but
$\Vert \alpha \Vert(q)=0$ for all $q\in \mathcal{S}\setminus \{ q_{l} \} $,
from \eqref{6.55a} we find:
$$
\int_{X} \beta_{0} \wedge \alpha 
=
\frac{1}{\Vert \alpha_{k} \Vert(x_{k,l})}
(H_{l} \bar{\varphi}_{l} + o(1)) +o(1)
\; \text{ with } \; 
\vert H_{l} \bar{\varphi}_{l}\vert=\vert H_{l} \vert \Vert \alpha \Vert(q_{l}) >0.
$$
As a consequence, we deduce that
$\frac{1}{\Vert \alpha_{k} \Vert(x_{k,l})} \leq C$, 
and since  
$$
\Vert \hat{\alpha}_{k} \Vert(x_{k,l})\longrightarrow \Vert \hat{\alpha}_{0} \Vert(q_{l})=0, \; \text{ as } \; k\longrightarrow \infty,$$ 
we find that,
$$
e^{-s_{k}} \leq C \Vert \hat{\alpha}_{k} \Vert(x_{k,l}) \longrightarrow 0,
\; \text{ i.e. } \; 
s_{k}\longrightarrow \infty, 
\; \text{ as } \; 
k\longrightarrow \infty,
$$ 
as claimed. 

\begin{equation}\label{6.56}
\; \text{ \textbf{\underline{Claim:} } } \;
\quad
\min
\{ 
\Vert \alpha_{k} \Vert(x_{k,l})\; : \; l=1,\ldots,m
\} 
\leq C
\quad\quad \quad \quad \quad   
\end{equation}

\noindent
Argue by contradiction, and assume that, 
(along a subsequence)  
$$\Vert \alpha_{k} \Vert(x_{k,l})\longrightarrow +\infty,$$
as $k\longrightarrow +\infty$, for every 
$l=1,\ldots,m$. 
In case  $q_{l}\in \mathcal{S} \cap Z_{0}$ then, 
\begin{equation}\label{integrals_in_balls}
\begin{split}
\int_{B_{r}(x_{k,l})}e^{-u_{k}}
= &
e^{-s_{k}}\int_{B_{r}(x_{k,l})}e^{\xi_{k}} \\
= \; &
\frac{1}{\Vert \alpha_{k} \Vert^{2}(x_{k,l})}
\int_{\{\vert z \vert \leq \frac{r}{\varepsilon_{k}}\}} 
\frac{1}{(1+\vert z \vert^{2})^{2}}dz d\bar{z}
\\
\leq &
\frac{C}{\Vert \alpha_{k} \Vert^{2}(x_{k,l})}
\longrightarrow 
0
,
\; \text{ as } \; 
k\longrightarrow +\infty. 
\end{split}
\end{equation}
On the other hand, in case 
$q_{l} \in \mathcal{S} \setminus Z_{0}\neq \emptyset$ then we know that, 
$q_{l} \not \in Z$ and so $\Vert \hat{\alpha}_{k} \Vert$ is bounded below away from zero
in a small neighbourhood of $q_{l}$. As a consequence, 
for 
$r>0$ sufficiently small, 
we find 
$
\int_{B_{r}(x_{k,l})}e^{\xi_{k}}\leq C
$ 
and we obtain:
\begin{equation*}%\label{integral_on_balls_to_zero}
\int_{B_{r}(x_{k,l})}e^{-u_{k}}
=
e^{-s_{k}}
\int_{B_{r}(x_{k,l})}e^{\xi_{k}}
\leq 
Ce^{-s_{k}}
\longrightarrow 
0,
\; \text{ as } \; 
k\longrightarrow \infty.  
\end{equation*}
Since we know also that, 
\begin{equation}\label{estimates_outside_balls}
e^{-u_{k}} =
e^{\xi_{k} -  s_{k}} 
\leq 
Ce^{- s_{k}}
\; \text{ in } \; 
X\setminus   \cup_{l=1}^{m}B_{r}(x_{k,l}),
\end{equation}
then we may conclude that,
$
\int_{X}e^{-u_{k}} 
\longrightarrow 0
$, as
$k\longrightarrow +\infty$
in contradiction to part (v) of Lemma \ref{lem_6.4},   
and \eqref{6.56}  is established. 

\

Since
for $q_{l}\in \mathcal{S} \setminus Z_{0} $
(i.e. $\Vert \hat{\alpha}_{0} \Vert(q_{l} )>0$)
we have:
$\Vert \alpha_{k} \Vert^{2}(x_{k,l})=O(e^{s_{k}}) \longrightarrow +\infty$,
then the condition
$\Vert \alpha_{k} \Vert(x_{k,l})\leq C$
implies that necessarily, $q_{l} \in \mathcal{S}\cap Z_{0}$.

Without loss of generality and after relabelling
(along a subsequence),
we can assume that
\begin{equation}\label{q_1_in_S_cap_Z_0}
q_{1} \in \mathcal{S}\cap Z_{0}
\; \text{ and } \; 
\mu_{k,1}
=
\Vert \alpha_{k} \Vert^{2}(x_{k,1})
=
\min_{l=1,\ldots, m}
\Vert \alpha_{k} \Vert^{2}(x_{k,l})
\leq C.
\end{equation}
As above, we can use such information in 
\eqref{6.55a} with $\alpha \in C_{2}(X)$ vanishing in  
$\mathcal{S}\setminus \{ q_{1} \} $, but not in $q_{1}$, to obtain:
\begin{equation*}
\int_{X} \beta_{0} \wedge \alpha 
=
\frac{1}{\sqrt{\mu_{k,1}}}
(H_{1}\bar \varphi_{1}+o(1))
,
\; \text{ as } \; 
k\longrightarrow +\infty,
\end{equation*} 
with $H_{1}\neq 0$ and $\vert \varphi_{1} \vert = \Vert \alpha \Vert(q_{1}) >0.$
So, by \eqref{q_1_in_S_cap_Z_0}, we derive first that,  
$\int_{X} \beta_{0} \wedge \alpha \neq 0$,
and then  we obtain also that,  
$\mu_{k,1}$ us uniformly bounded below away from zero. 
In fact, more generally, we conclude:
$$
\mu_{k,l} = \Vert \alpha_{k} \Vert^{2}(x_{k,l}) \geq  c >0
,
\; \forall \; k\in \N
\; \text{ and } \; 
\; \forall \; 
l\in \{ 1,\ldots,m \}. 
$$
Next, we define the set of indices:
\begin{equation*}
J
=
\{  
l\in \{ 1,\ldots, m \} 
\; : \; 
\limsup_{k \to +\infty}
\mu_{k,l} 
 < 
 \infty
\}
,
\end{equation*}
so that $1\in J$. Furthermore, if $l\in J$  then necessarily
$q_{l}\in \mathcal{S}\cap Z_{0}$ and (along a subsequence)
we can summarize the properties of the elements of $J$ as follows:
\begin{equation*}%\label{mu_k_l_limits}
\mu_{k,l}
\longrightarrow \mu_{l} >0
\; \text{ and } \;
q_{l} \in \mathcal{S}\cap Z_{0}
,
\; \forall \; l\in J\neq \emptyset.  
\end{equation*}
On the contrary,
$$
\; \text{ if } \;
\{ 1,\ldots, m \} \setminus J \neq \emptyset,
\; \text{ then } \; 
\mu_{k,l} \longrightarrow  +\infty
\; \text{ for } \; 
l\in \{ 1,\ldots, m \} \setminus J.
$$
Therefore, 
\eqref{6.55a}
can be expressed as follows:
\begin{equation*}%\label{with_J}
\int_{X}  \beta_{0} \wedge \alpha 
= 
(
\sum_{l \in J}
\frac{1}{\Vert \alpha_{k} \Vert(x_{k,l})}
H_{l}\bar \varphi _{l})
+
o(1),
\; \text{ as } \; 
k\longrightarrow +\infty,
\end{equation*}
and we easily derive that, 
\begin{equation*}
\; \text{ if } \; 
\alpha \in C_{2}(X)
\; \text{ and } \;
\Vert \alpha \Vert(q_{l})=0,  
\; \forall \; l\in J
\; \text{ then } \;  
\int_{X} \beta_{0} \wedge \alpha =0,
\end{equation*} 
and also (ii) is established. 

At this point we can combine the estimate \eqref{estimates_outside_balls} 
together with the integral estimate \eqref{integrals_in_balls} over 
the ball $B_{r}(x_{k,l})$ when $l\in \{ 1,\ldots,m \} \setminus J$,
and the estimate \eqref{formula_remark} around $x_{k,l}$ when $l\in J$, 
to deduce that,
$e^{-u_{k}}
\rightharpoonup
\pi \sum_{j\in J} \frac{1}{\mu_{l}}\delta_{q_{l}}
$,
weakly in the sense of measure, as claimed in (i).

Finally, for $r>0$ sufficiently small, we have:
\begin{equation*}
\int_{X}\vert \nabla w_{k} \vert^{2}
=
\int_{X} \vert \nabla \xi_{k} \vert^{2}
=
\sum_{l=1}^{m} \int_{B_{r}(q_{l})}
\vert \nabla \xi_{k} \vert^{2}
+
O(1),
\end{equation*}
and by means of the gradient estimate \eqref{gradient_estimate}  
in Theorem \ref{theorem_from_part_1},
together with  \eqref{xi_k_versus_d_k_minus_s_k},\eqref{xi_k_0} and 
\eqref{W_k_0_property}
we find:
\begin{equation*}
\int_{B_{r}(x_{k,l})}\vert \nabla \xi_{k} \vert^{2}
=
16 \pi 
(d_{k}-\log\Vert \alpha_{k} \Vert^{2}(x_{k,l}))
+
O(1).
\end{equation*}
On the other hand, for $l\in J$ we know that
$\Vert \alpha_{k} \Vert^{2}(x_{k,l})=O(1)$,
and we conclude:
\begin{equation*}
\begin{split}
c_{k}
= &
D_{t_{k}}(u_{k},\eta_{k})
=
\frac{1}{4}\int_{X}\vert \nabla w_{k} \vert^{2}
-
4\pi(\mathfrak{g}-1)d_{k}+O(1)
\\
= &
-
4\pi
\left(
(\mathfrak{g}-1-m)d_{k}
+
\sum_{l\in \{ 1,\ldots,m \} \setminus J}^{m}
\log(\Vert \alpha_{k} \Vert^{2}(x_{k,l}))
\right)
+
O(1)
,
\end{split}
\end{equation*}
and also \eqref{c_k_asymptotic} is established with $d_{k}\longrightarrow \infty$, as $k\longrightarrow \infty$.  
\end{proof}	 

It is reasonable to expect that the set of indices $J$ in Theorem \ref{main_thm_2}
either is a "singleton" or it covers the full set of indices.
In either cases we get a "cleaner" version of Theorem \ref{main_thm_2}, 
as already expressed by Theorem \ref{thm_last_intro} of the introduction, 
in case $\mathcal{S}$ is a "singleton". On the other hand we have:

\begin{corollary}\label{cor_6.10}
Let $\mathcal{S}$ be the (non-empty) blow-up set 
of $\xi_{k}$ in \eqref{the_blow_up_set_S}
and assume  $J=\{ 1,\ldots,m \}$ in Theorem \ref{main_thm_2}.  
Then   
$\mathcal{S}\subset Z_{0}$ and there holds:
\begin{enumerate}[label=(\roman*)]
\item $\mu_{k,l}\xrightarrow{k\to \infty} \mu_{l}>0$, 
$\; \forall \; l=1,\ldots,m$;
\item 
$
e^{-u_{k}}
\longrightarrow 
\pi \sum_{l=1}^{m}
\frac{1}{\mu_{l}}\delta_{q_{l}}$,	
as 
$k\longrightarrow +\infty$,
weakly in the sense of measures;
\item
$
\int_{X} \beta_{0}\wedge \alpha =0$,
$\alpha \in Q_{2}[q_{1},\ldots,q_{m}]$;
\item
$c_{k}=-4\pi(\mathfrak{g}-1-m)d_{k}+O(1)$,
with
$d_{k}=\fint_{X}u_{k}\longrightarrow +\infty$, 
as $k\longrightarrow +\infty$. 
\end{enumerate}
\end{corollary}

\
	
\noindent	
Thus, under the assumptions of Corollary \ref{cor_6.10}  we can still conclude, 
as before, 

\noindent
that, when
$\rho([\beta])=4\pi(\mathfrak{g}-1)$ (i.e. $m=\mathfrak{g}-1$)
then $D_{0}$ is bounded from below. 

\

At this point, 
Theorem \ref{thm_last_intro}, Theorem \ref{main_thm_3}
and Theorem \ref{thm_infimum_in_lambda_outside_tau_X} 
can be easily derived.

\

\

\textbf{Acknowledgments}

\noindent
We thank Dr. Martin Mayer for his precious assistance and 
useful comments during the preparation of the manuscript. 

Also, we have profited enormously from the illuminating comments
and suggestions of a dear colleague, Prof. Stefano Trapani. 

This work was partially supported by:
MIUR excellence project:   
Department of Mathematics, University of Rome "Tor Vergata"
CUP E83C18000100006,
by "Beyond Borders" research project: 
"Variational approaches to PDE" 
and Fondi di Ricerca Scientifica d'Ateneo, 2021:
Research Project 
"Geometric Analysis in Nonlinear PDEs",
CUP E83C22001810005.


\begin{thebibliography}{99} 
 
 
\bibitem{Alessandrini_Li_Sanders} 
{
Alessandrini D., Li Q., Sanders A.,
\emph{Nilpotent Higgs bundles and 
minimal surfaces in hyperbolic three-space},
preprint 2021
}

\bibitem{Ambrosetti_Rabinowitz}
{
Ambrosetti A., Rabinowitz P.H.,
Dual variational methods in critical point theory and applications.
\emph{J. Functional Analysis}  
\textbf{14} (1973), 349–381. 
}

\bibitem{Aubin_Book}
{
Aubin, T., 
\emph{Some Nonlinear Problems in Riemannian Geometry}.
Springer-Verlag Berlin Heidelberg, 1998
}



\bibitem{Bartolucci_Chen_Lin_Tarantello}
{
Bartolucci D.,  Chen C.C.,  Lin C.S.,  Tarantello G.,
Profile of blow-up solutions to mean field equations with singular data. 
\emph{Comm. Partial Differential Equations}
\textbf{29} (2004), no. 7-8, 1241–1265
}




\bibitem{Bartolucci_Tarantello} 
{
Bartolucci D.,  Tarantello G.,
Liouville type equations with singular data and their applications to periodic multivortices for the electroweak theory. 
\emph{Comm. Math. Phys.}
\textbf{229} (2002), no. 1, 3–47
} 

\bibitem{Bartolucci_Tarantello_JDE}
{
Bartolucci D., Tarantello G.,
Asymptotic blow-up analysis for singular Liouville type equations with applications.
\emph{J. Diff. Eq.} 
\textbf{262} (2017), no. 7, 3887-3931
}  


%\bibitem{Battaglia_Jevnikar_Malchiodi_Ruiz}
%Battaglia, L., Jevnikar A., Malchiodi A., Ruiz D.,
%A general existence result for the Toda system on compact surfaces. 
%\emph{Adv. Math.} 
%\textbf{285} (2015), 937-979. 

%\bibitem{Brezis_Li_Shafrir}
%{
%Brezis H., Li Y.Y., Shafrir I., 
%A $sup+inf$ inequality for some nonlinear elliptic equations involving exponential %nonlinearities.
%\emph{J. Funct. Anal.} \textbf{115} (1993), no. 2, 344–358. 
%} 

\bibitem{Brezis_Merle} 
{
Brezis H., Merle F., Uniform estimates and blow-up behavior for solutions of 
$-\Delta u = V(x)e^{u}$
in two dimensions. 
\emph{Comm. Partial Differential Equations} 
\textbf{16} (1991), no. 8-9, 1223-1253.
}

\bibitem{Bryant}
{
Bryant R., 
Surfaces of mean curvature one in hyperbolic space,
Th\'eorie des vari\'et\'es minimales et applications
\emph{Ast\'erisque} No. 154-155 (1987), 12, 321-347, 353 (1988)
}


%\bibitem{Chai_Lin_Wang}
%{
% Chai, C.L., Lin, C.S., Wang, C.L.,  
% Mean field equations, hyperelliptic curves and modular forms: I. 
% \emph{Camb. J. Math.} 
% \textbf{3} (2015), no. 1-2, 127-274.
%}
 

%\bibitem{Chen_Kuo_Lin_1}
%{
% Chen Z., Kuo T.J., Lin C.S., The geometry of generalized Lamé equation, II: %Existence of premodular forms and application.
% \emph{J. Math. Pures Appl.} 
% \textbf{132} (2019), 251-272.
%}

%\bibitem{Chen_Kuo_Lin_2}
%{
%Chen Z., Kuo T.J., Lin C.S., 
%Simple zero property of some holomorphic functions on the moduli space of tori. 
%\emph{Sci. China Math.} 
%\textbf{62} (2019), no. 11, 2089-2102
%}

%\bibitem{Chen_Kuo_Lin_3}
%{
% Chen Z., Kuo T.J., Lin C.S., The geometry of generalized Lam\'e equation, I. 
% \emph{J. Math. Pures Appl.} 
% \textbf{127} (2019), 89-120
%}


%\bibitem{Chen_Kuo_Lin_Takemura_1}
%{
%Chen Z., Kuo T.J., Lin C.S., Takemura K.,
%On reducible monodromy representations of some generalized Lam\'e equation. 
%\emph{Math. Z.} 
%\textbf{288} (2018), no. 3-4, 
%}

%\bibitem{Chen_Kuo_Lin_Takemura_2}
%{
%Chen Z., Kuo T.J., Lin C.S., Takemura K.,
%Real-root property of the spectral polynomial of the Treibich-Verdier potential and %related problems. 
%\emph{J. Differential Equations} 
%\textbf{264} (2018), no. 8, 5408–5431
%}

%\bibitem{Chen_Kuo_Lin_Wang}
%{
%Chen Z., Kuo T.J., Lin C.S., Wang C.L.,
%Green function, Painlev\'e VI equation, and Eisenstein series of weight one. 
%\emph{J. Diff. Geom.} 
%\textbf{108} 
%(2018), no. 2, 185-241.
%}
  


\bibitem{Chen_Lin_1}
{
Chen C.C., Lin C.S.,
Sharp estimates for solutions of multi-bubbles in compact Riemann surfaces.
\emph{Comm. Pure Appl. Math.} \textbf{55} (2002), no. 6, 728–771. 
} 
 

%\bibitem{Chen_Lin_2}
%{
% Chen C.C., Lin C.S.,
%Topological degree for a mean field equation on Riemann surfaces. 
%\emph{Comm. Pure Appl. Math.}
%\textbf{56} (2003), no. 12, 1667–1727
%}

%\bibitem{Chen_Lin_3}
%{
%Chen C.C., Lin C.S.,
%Mean field equation of Liouville type with singular data: topological degree.
%\emph{Comm. Pure Appl. Math.} \textbf{68}, 
%(2015), no. 6, 887–947. 
%}

%\bibitem{Chen_Lin_4}
%{
%Chen C.C., Lin C.S., Mean field equations of Liouville type with singular data: %sharper estimates. \emph{Discrete Contin. Dyn. Syst} 
%\textbf{28} (2010), no. 3, 1237–1272
%}


%\bibitem{Chen_Wang_Wu_Xu}
%{
%Chen Q., Wang W., Wu Y., Xu B.,
%Conformal metrics with constant curvature one and finitely many conical %singularities on compact Riemann surfaces. 
%\emph{Pacific J. Math.} 
%\textbf{273} (2015), no. 1, 75–100.
%}

%\bibitem{Dey}
%{
%Dey, S.,
%Spherical metrics with conical singularities on 2-spheres. 
%\emph{Geom. Dedicata} 
%\textbf{196} (2018), 53–61. 
%}


%\bibitem{Donaldson_1} 
%{
%Donaldson S.K.,
%Anti self-dual Yang-Mills connections over complex algebraic surfaces and stable %vector bundles. 
%Proc. London Math. Soc. \textbf{50} (1985), no. 1, 1-26.
%}

%\bibitem{Donaldson_2}
%{
%Donaldson S.K.,
%Twisted harmonic maps and the self-duality equations. 
%\emph{Proc. London Math. Soc.} 
%\emph{55} (1987), no. 1, 127-131. 
%}

%\bibitem{Donaldson_3}
%{
%Donaldson S.K.,
%\emph{Kähler metrics with cone singularities along a divisor.}  
%Essays in mathematics and its applications, 49-79, Springer, Heidelberg, 2012. 
%}

\bibitem{Donaldson_Book}
{
Donaldson S., 
Riemann Surfaces,
\emph{Oxford Graduate Texts in Mathematics},
2011
}


%\bibitem{Donaldson_Sun_1}
%{
% Donaldson S.K., Sun S., 
% Gromov-Hausdorff limits of Kähler manifolds and algebraic geometry. 
% \emph{Acta Math.} 
% \textbf{213} (2014), no. 1, 63–106.
%}


%\bibitem{Donaldson_Sun_2}
%{
%Donaldson S.K., Sun S., 
%Gromov-Hausdorff limits of Kähler manifolds and algebraic geometry, II. 
%\emph{J. Differential Geom.} 
%\textbf{107} (2017), no. 2, 327–371.
%}



%\bibitem{Eremenko_1}
%{
%Eremenko A., 
%Metrics of positive curvature with conic singularities on the sphere. 
%\emph{Proc. Amer. Math. Soc.} 
%\textbf{132} (2004), no. 11, 3349-3355.
%}

%\bibitem{Eremenko_2}
%{
%Eremenko A., 
%Metrics of constant positive curvature with four conic singularities on the sphere. 
%\emph{Proc. Amer. Math. Soc.} 
%\textbf{148} (2020), no. 9, 3957-3965
%}

%\bibitem{Eremenko_3}
%{
%Eremenko A., 
%Co-axial monodromy. 
%\emph{Ann. Sc. Norm. Super. Pisa Cl. Sci.} 
%\textbf{20} (2020), no. 2, 619-634
%}






%\bibitem{Eremenko_Gabrielov_Hinkkanen}
%{
%Eremenko A., Gabrielov A., Hinkkanen A., 
%Exceptional solutions to the Painlev\'e VI equation. 
%\emph{J. Math. Phys.} 
%\textbf{58} (2017), no. 1, 012701, 8 pp
%}

%\bibitem{Eremenko_Gabrielov_Tarasov_1}
%{
%Eremenko A.,  Gabrielov A., Tarasov V., 
%Metrics with conic singularities and spherical polygons. 
%\emph{Illinois J. Math.} 
%\textbf{58} (2014), no. 3, 739-755
%}

%\bibitem{Eremenko_Gabrielov_Tarasov_2}
%{
%Eremenko A.,  Gabrielov A., Tarasov V., 
%Spherical quadrilaterals with three non-integer angles. 
%\emph{Zh. Mat. Fiz. Anal. Geom.} 
%\textbf{12} (2016), no. 2, 134-167
%}


%\bibitem{Eremenko_Gabrielov_Tarasov_3}
%{
%Eremenko A., Tarasov V., 
%Fuchsian equations with three non-apparent singularities. 
%\emph{Symmetry Integrability Geom. Methods Appl.}
%\textbf{14} (2018), no. 058, 12 pp.
%}


\bibitem{Goncalves_Uhlenbeck}
{
Goncalves K., Uhlenbeck K.,
Moduli space theory for constant mean curvature surfaces immersed in space-forms.
\emph{Comm. Anal. Geom.}  
\textbf{15} (2007), 299-305. 
}  


\bibitem{Griffiths_Harris} 
{
Griffiths P., Harris J.,
\emph{Principles of Algebraic Geometry},
Wiley Classics Library (2014), John Wiley \& Sons
}


\bibitem{Hitchin}
{
Hitchin N., 
The self-duality equations on a  Riemann surface,
\emph{Proc. London Math. Soc}, (3) 55 (1987), no 1, 59-126
}

\bibitem{Hung_Loftin_Lucia}
{
Huang Z., Loftin J., Lucia M.,
Holomorphic cubic differentials and minimal Lagrangian surfaces in 
$\mathbb{C}\mathbb{H}^{2}$.
\emph{Math. Res. Lett.} 
\textbf{20} (2013), no. 3, 501–520. 
}

\bibitem{Huang_Lucia}
{
Huang Z., Lucia M.,
Minimal immersions of closed surfaces in hyperbolic three-manifolds. 
\emph{Geom. Dedicata}
\textbf{158} (2012), 397–411. 
} 

 
\bibitem{Huang_Lucia_Tarantello_1}
{
Huang Z., Lucia M., Tarantello G., 
Bifurcation for minimal surface equation in hyperbolic 3-manifolds. 
\emph{Ann. Inst. H. Poincaré Anal. Non Linéaire} 
\textbf{38} (2021), no. 2, 243-279
}

\bibitem{Huang_Lucia_Tarantello_2}
{
Huang Z., Lucia M., Tarantello G., 
Donaldson Functional in Teichm\"uller Theory,
Int. Math. Res. Notes, to appear.
}


%\bibitem{Jaffe_Taubes_Book}
%{
%Jaffe A., Taubes C., 
%\emph{Vortices and Monopoles, Structure of Static Gauge Theories}. 
%Progress in Physics 2, Boston-Basel-Stuttgart, Birkhäuser Verlag 1980
%}
 

\bibitem{Jost}
{
Jost, J., 
\emph{Compact Riemann surfaces. An introduction to contemporary mathematics}, 
Translated from the German manuscript,  Springer-Verlag, Berlin, 1997
}

\bibitem{Lee_Lin_Tarantello_Yang}
{
Lee Y.,  Lin C.S.,  Tarantello G.,  Yang W., 
Sharp estimates for solutions of mean field equations with collapsing singularity. 
\emph{Comm. Partial Differential Equations}
\textbf{42} (2017), no. 10, 1549–1597
}



\bibitem{Lee_Lin_Wei_Yang}
{
Lee Y., Lin C.S., Wei J., Yang W.,
Degree counting and shadow system for Toda system of rank two: one bubbling.
\emph{J. Diff. Eq.} 
\textbf{264} (2018), no. 7, 4343-4401. 
}


\bibitem{Lee_Lin_Yang_Zhang}  
{
Lee Y., Lin C.S., Yang W., Zhang L.,
Degree counting for Toda system with simple singularity: one point blow up.
\emph{J. Diff. Eq.}
\textbf{268} 
(2020), no. 5, 2163-2209
}
 

\bibitem{Li}
{
Li Q.,
\emph{An introduction to Higgs bundles via Harmonic maps},
SIGMA, \textbf{15}, 2019, 1-30 
}

\bibitem{Li_Harnack}
{
Li Y.Y., Harnack type inequality: the method of moving planes. 
\emph{Comm. Math. Phys.} 
\textbf{200} (1999), no. 2, 421-444
}

%\bibitem{Liouville}
%{
%Liouville J., 
%Sur 'equation aux deriv\'ees partielles 
%$
%\frac{\partial^{2} \log \lambda}{\partial u \partial v }\pm \frac{\lambda}{2a^{2}}%=0$.
%\emph{J. Math. Pure Appl.} 
%\textbf{8} (1853), 71-72
%}

\bibitem{Li_Shafrir}
{
Li Y.Y., Shafrir I.,
Blow-up analysis for solutions of $-\Delta u = Ve^{u}$ in dimension two. 
\emph{Indiana Univ. Math. J.} 
\textbf{43} (1994), no. 4, 1255–1270
}


%\bibitem{Lin_Nie_Wei}
%{
%Lin C.S.,  Nie Z., Wei J.,
%Toda systems and hypergeometric equations.
%\emph{Trans. Amer. Math. Soc.} 
%\textbf{370} (2018), no. 11, 7605–7626. 
%}


\bibitem{Lin_Tarantello}
{
Lin C.S, Tarantello G., When "blow-up" does not imply "concentration": a detour from Brezis-Merle's result. 
\emph{C.R. Math. Acad. Sci. Paris}
\textbf{354} (2016), no. 5, 493-498
} 

 

%\bibitem{Lin_Wang}
%{
%Lin C.S., Wang, C.L.,
%Elliptic functions, Green functions and the mean field equations on tori. 
%\emph{Ann. of Math.} 
%\textbf{2} 172 (2010), no. 2, 911–954
%} 

%\bibitem{Lin_Wei_Yang_Zhang}
%{
%Lin C.S.,  Wei J.C.,  Yang W.,  Zhang L., 
%On rank-2 Toda systems with arbitrary singularities: local mass and new estimates. 
%\emph{Anal. PDE}
%\textbf{11} (2018), no. 4, 873–898
%}

%\bibitem{Lin_Wei_Zhang}
%{
%Lin C.S.,  Wei J.C.,  Zhang L.,  
%Classification of blowup limits for SU(3) singular Toda systems. 
%\emph{Anal. PDE}
%\textbf{8} (2015), no. 4, 807-837.
%}

%\bibitem{Lin_Wei_Zhao_1}
%{ 
%Lin C.S., Wei J., Zhao C.,
%Asymptotic behavior of SU(3) Toda system in a bounded domain.
%\emph{Manuscripta Math.}
%\textbf{137} (2012), no. 1-2, 1–18.}

%\bibitem{Lin_Wei_Zhao_2}
%{
%Lin C.S., Wei J., Zhao C.,
%Sharp estimates for fully bubbling solutions of a SU(3) Toda system. 
%\emph{Geom. Funct. Anal.}
%\textbf{22} (2012), no. 6, 1591–1635. 
%}
 
\bibitem{Loftin_Macintosh_1}
{
Lofting J., Macintosh I., 
Equivariant minimal serufaces in $\mathbb{C}\mathbb{H}^{2}$ and their Higgs bundles.
\emph{Asian J. Math.} \textbf{23} (2019). 71-106
}

\bibitem{Loftin_Macintosh_2}
{ 
Lofting J., Macintosh I., 
The moduli spaces of equivariant minimal surfaces in  
$\mathbb{R}\mathbb{H}^{3}$ and $\mathbb{R}\mathbb{H}^{4}$ via Higgs bundles.
\emph{Geom. Dedicata} \textbf{201} (2019). 325-351
}



%\bibitem{Luo_Tian}
%{
%Luo F., Tian G., 
%Liouville equation and spherical convex polytopes.
%\emph{Proc. Amer. Math. Soc.} 
%\textbf{116} (1992), no. 4, 1119–1129. 
%} 




%\bibitem{Mazzeo_Weiss} 
%{
%Mazzeo R., Weiss H.,
%\emph{Teichmüller Theory for Conic Surfaces}. 
%Geometry, Analysis and Probability. Progress in Mathematics, vol 310. Birkhäuser
%}

%\bibitem{Mazzeo_Zhu_1}
%{
%Mazzeo R., Zhu X., 
%Conical metrics on Riemann surfaces, I: The compactified configuration space and %regularity.
%\emph{Geometry \& Topology} 
%\textbf{24} (2020) 309–372 
%}

%\bibitem{Mazzeo_Zhu_2}
%{
%Mazzeo R., Zhu X., 
%Conical Metrics on Riemann Surfaces, II: Spherical Metrics.
%\emph{Int. Math. Res. Not-}
%rnab011, https://doi.org/10.1093/imrn/rnab011
%}


\bibitem{Miranda}
{
Miranda R.,
\emph{Algebraic Curves and Riemann Surfaces.}
Graduate Studies in Mathematics, Vol 5., 1995,
American Mathematical Society

}
%\bibitem{Mondello_Panov_1}
%{
%Mondello G., Panov D.,
%Spherical metrics with conical singularities on a 2-sphere: angle constraints. 
%\emph{Int. Math. Res. Not.} 2016, no. 16, 4937-4995.
%}

%\bibitem{Mondello_Panov_2}
%{
%Mondello G., Panov D.,
%Spherical surfaces with conical points: systole inequality and moduli spaces with %many connected components. \emph{Geom. Funct. Anal.} 
%\textbf{29} (2019), no. 4, 1110-1193.
%}

%\bibitem{Nakayama}
%{
%Nakayama Y.,
%Liouville Field Theory: A Decade After The Revolution,
%\emph{International Journal of Modern Physics A}
%\textbf{19}19, no. 17,18, pp. 2771-2930 (2004) 
%}

\bibitem{Narasimhan}
{
Narasimhan, R., 
\emph{Compact Riemann surfaces. Lectures in Mathematics} 
ETH Zürich. Birkhäuser Verlag, Basel, 1992. iv+120 pp.
}

\bibitem{Nirenberg_Topics_In_Nonlinear_Functional_Analysis}
{
Nirenberg L.,
\emph{Topics in Nonlinear Functional Analysis},
Courant Lecture Notes
Vol. 6, 2001, 145 pp
}


\bibitem{Rossman_Umehara_Yamada}
{
Rossman W., Umehara M., Yamada K.,
Irreducible constant mean curvature 1 surfaces in hyperbolic space with positive genus,
\emph{Tohoku Math. J.} 
(2) 49(4): 449-484 (1997).
}

%\bibitem{Simpson}
%{
%Simpson C.T.,
%Constructing Variations of Hodge Structure Using Yang-Mills Theory and Applications %to Uniformization.
%\emph{Journal of the American Mathematical Society}
%\textbf{1}, no. 4 (1988), pp. 867-918 
%}

\bibitem{Struwe_Book}
{
Struwe M., 
\emph{Variational Methods},
Springer Berlin, Heidelberg, 2008.
}

\bibitem{Suzuki_Ohtusuka}
{
Ohtsuka H., Suzuki T., Blow-up analysis for Liouville type equation in self-dual 
gauge field theories. 
\emph{Comm. Contemp. Math.}
\textbf{7} (2005), no. 2, 177-205
}

%\bibitem{Song_Cheng_Li_Xu}
%{
%Song J., Cheng Y., Li B., Xu B.,
%Drawing cone spherical metrics via Strebel differentials. 
%\emph{Int. Math. Res. Not.} 
%2020, no. 11, 3341-3363. 
%}


%\bibitem{Tarantello_Harnack}
%{
%Tarantello G.,
%A Harnack inequality for Liouville-type equations with singular sources.
%\emph{Indiana Univ. Math. J.} 
%\textbf{54} (2005), no. 2, 599–615
%}

\bibitem{Tar_1}
{ 
Tarantello G.,
On the blow-up analysis 
at collapsing poles
for solutions of singular Liouville type equations.
\emph{Preprint 2021}
}

%\bibitem{Tar_2}
%{
%Tarantello G.,
%Asymptotics for minimizers of the Donaldson functional in Teichm\"uller theory.
%\emph{Preprint 2021}
%} 

 
%\bibitem{Tarantello_Book}  
%{
%Tarantello G.,
%\emph{Selfdual Gauge Field Vortices
%An Analytical Approach}.
%Progress in Nonlinear Diff. Eq. and Their Applications
%\textbf{72}, Birkhäuser, Basel, 2008.
%}



\bibitem{Taubes}
{
Taubes, C.H.,
Minimal surfaces in germs of hyperbolic 3-manifolds. 
\emph{Proceedings of the Casson Fest, 69–100},
Geom. Topol. Monogr., 7, Geom. Topol. Publ., Coventry, 2004.
53D30 (53C42 57M50 58D17)
}

%\bibitem{Troyanov_1}
%{
%Troyanov M.,
%Metrics of constant curvature on a sphere with two conical singularities. 
%\emph{Differential geometry
%(Peñíscola, 1988), 296–306}
%Lecture Notes in Math., 1410, Springer, Berlin, 1989. 
%}

%\bibitem{Troyanov_2}
%{
%Troyanov M., Prescribing curvature on compact surfaces with conical singularities. 
%\emph{Trans. Amer. Math. Soc.} 
%\textbf{324} (1991), no. 2, 793–821.
%}


\bibitem{Uhlenbeck}
{
Uhlenbeck K.,
\emph{Closed minimal surfaces in hyperbolic 3-manifolds.} 
Seminar on minimal submanifolds, 147–168, Ann. of Math. Stud., 103, Princeton Univ. Press, Princeton, NJ, 1983.
}

\bibitem{Umehara_Yamada}
{
Umehara M.,  Yamada K.,
Complete Surfaces of Constant Mean Curvature-1 in the Hyperbolic 3-space,
\emph{Annals of Mathematics}
Second Series, Vol. 137, No. 3 (May, 1993), pp. 611-638 
}


\bibitem{Voisin}  
{ 
Voisin C., 
\emph{Hodge theory and complex algebraic geometry. I.}
Cambridge Studies in Advanced Mathematics, Cambridge (2007)
}



%\bibitem{Umehara_Yamada}
%{
%Umehara M., Yamada K., 
%Metrics of constant curvature 1 with three conical singularities on the 2-sphere. 
%\emph{Illinois J. Math.} 
%\textbf{44} (2000), no. 1, 72–94. 
%}




\bibitem{Wei_Zhang_1} 
{
Wei J., Zhang L.,
Estimates for Liouville equation with quantized singularities.
\emph{Advances in Mathematics}
\textbf{380} (2021)
}

\bibitem{Wei_Zhang_2}
{
Wei J., Zhang L.,
Laplacian vanishing theorem for quantized singular Liouville equation, 
preprint 2022
}



%\bibitem{Yang_Book}
%{ Yang Y., 
%\emph{Solitons in field theory and nonlinear analysis}. 
%Springer Monographs in Mathematics. Springer-Verlag, New York, 2001}


%\bibitem{Zhu_1}
%{
%Zhu X.,
%Spherical conic metrics and realizability of branched covers.
%\emph{Proc. Amer. Math. Soc.} 
%\textbf{147} (2019), no. 4, 1805–1815. 
%}

%\bibitem{Zhu_2}
%{
%Zhu X.,
%Rigidity of a family of spherical conical metrics.
%\emph{New York J. Math.} 
%\textbf{26} (2020), 272–284. 
%}

\end{thebibliography}
\end{document}